\documentclass[11pt]{article}

\usepackage[margin=1in]{geometry}
\usepackage{amsmath,amsthm,amssymb,amsfonts,float}
\usepackage{comment}
\usepackage[utf8]{inputenc}
\usepackage[T1]{fontenc}
\usepackage{enumitem}   

\usepackage[colorlinks=true, pdfstartview=FitV, linkcolor=blue,citecolor=blue, urlcolor=blue]{hyperref}

\usepackage[abbrev,lite,nobysame]{amsrefs}
\usepackage{times}
\usepackage[usenames,dvipsnames]{color}
\usepackage{esint}

\usepackage{mathtools,enumitem,mathrsfs}

\usepackage[compact]{titlesec}

\usepackage[title]{appendix}

\usepackage{comment}

\newcommand{\leqc}{\lesssim}

\newcommand{\grad}{\nabla}

\newcommand{\norm}[1]{\left|\left| #1 \right|\right|}
\newcommand{\abs}[1]{\left| #1 \right|}
\newcommand{\set}[1]{\left\{ #1 \right\}}
\newcommand{\brak}[1]{\left\langle #1 \right\rangle} 

\newcommand{\R}{\mathbb{R}}
\newcommand{\N}{\mathbb{N}}

\newcommand{\C}{\mathbb{C}}

\renewcommand{\S}{\mathbb{S}}

\newcommand{\dee}{\mathrm{d}}

\usepackage{bbm}

\renewcommand{\P}{\mathbf{P}}

\newcommand{\E}{\mathbf{E}}
\newcommand{\EE}{\mathbf E}
\newcommand{\PP}{\mathbf P}

\newcommand{\Pt}{\mathcal{P}}

\newtheorem{theorem}{Theorem}[section]
\newtheorem{proposition}[theorem]{Proposition}
\newtheorem{corollary}[theorem]{Corollary}
\newtheorem{lemma}[theorem]{Lemma}
\newtheorem*{lemma*}{Lemma}

\newtheorem{assumption}{Assumption}

\theoremstyle{definition}

\newtheorem{remark}{Remark}

\DeclareUnicodeCharacter{2BC}{'}

\numberwithin{equation}{section}

\begin{document}

\title{Stationary measures for stochastic differential equations with degenerate damping}
%\subjclass{60H10, 60H50, 37A25 37H30}
\author{Jacob Bedrossian\thanks{\footnotesize Department of Mathematics, University of Maryland, College Park, MD 20742, USA \href{mailto:jacob@math.umd.edu}{\texttt{jacob@math.umd.edu}}. J.B. was supported by NSF CAREER grant DMS-1552826 and NSF Award DMS-2108633.}  \and Kyle Liss\thanks{\footnotesize Department of Mathematics, Duke University, Durham, NC 27710, USA \href{mailto:kyle.liss@duke.edu}{\texttt{kyle.liss@duke.edu}}. K.L was supported by NSF Award No. DMS-2038056 and DMS-1552826}}
%\keywords{}

\maketitle

\begin{abstract}
  A variety of physical phenomena involve the nonlinear transfer of energy from weakly damped modes subjected to external forcing to other modes which are more heavily damped. 
  In this work we explore this in (finite-dimensional) stochastic differential equations in $\mathbb R^n$ with a quadratic, conservative nonlinearity $B(x,x)$ and a linear damping term $-Ax$ which is degenerate in the sense that $\mathrm{ker} A \neq \emptyset$.
  
  We investigate sufficient conditions to deduce the existence of a stationary measure for the associated Markov semigroups. Existence of such measures is straightforward if $A$ is full rank, but otherwise, energy could potentially accumulate in $\mathrm{ker} A$ and lead to almost-surely unbounded trajectories, making the existence of stationary measures impossible.
We give a relatively simple and general sufficient condition based on time-averaged coercivity estimates along trajectories in neighborhoods of $\mathrm{ker} A$ and many examples where such estimates can be made. 
\end{abstract}

\setcounter{tocdepth}{2}
{\small\tableofcontents}

%% Introduction

\section{Introduction}\label{sec:Intro}

A variety of physical phenomena involve the nonlinear transfer of energy from weakly damped modes subjected to external forcing to other modes which are more heavily damped.
In hydrodynamic turbulence for example, the forcing is considered to act at large scales whereas in the high Reynolds number limit, the viscous dissipation is only strong at very high frequencies.
This leads to the phenomenon known as anomalous dissipation (see e.g. \cite{Frisch1995,BMOV05}). 
A study of such phenomena in infinite-dimensional systems remains largely out of reach (with a few exceptions, for example some simplified shell models \cite{MSE07,FGHV16} and Batchelor-regime passive scalar turbulence \cite{BBPS19}).
As suggested in e.g. \cite{Majda16}, it is natural to first study the analogues in finite-dimensional systems.
In this setting we will study systems with damping which only acts on a proper subset of the degrees of freedom and ask the question of whether or not a statistical equilibrium, i.e. a stationary measure, can still be shown to exist.
If the undamped modes are directly forced at least, for this to be possible the nonlinearity must continually pump energy away from the modes without damping into modes with damping.  

We study the following prototypical class of stochastic differential equations (SDEs) for $x_t \in \R^n$ 
\begin{equation} \label{eq:SDE}
\begin{cases}
dx_t = B(x_t,x_t)dt -Ax_tdt + \sigma dW_t \\
x_t|_{t=0} = x_0 \in \R^n. 
\end{cases}
\end{equation}
Here, $W_t = (W_t^{(1)}, \ldots, W_t^{(n)})$ is an $n$-dimensional canonical Brownian motion on a complete probability space $(\Omega, \mathcal{F}, \P)$, $A \in \R^{n \times n}$ is positive semi-definite (with $\mathrm{ker} A \neq \emptyset$), and $\sigma \in \R^{n\times n}$. We will assume for simplicity throughout this introduction that $\sigma$ is full rank, though, as discussed in the main body of the text, weaker conditions are possible for the examples we study.
The nonlinear term $B$ is bilinear such that the energy $|x|^2$ is conserved:
\begin{align} \label{eq:Bassumption}
\quad x \cdot B(x,x) = 0.  
\end{align}
Many of the specific examples we study also satisfy $\grad \cdot B = 0$, but this is not required for our methods. 
This class of systems contains Galerkin truncations of both the 2d and 3d Navier-Stokes equations, as well as Lorenz-96 \cite{Lorenz1996}, and the classical shell models of hydrodynamic turbulence, GOY \cite{Gledzer1973,YO87} and Sabra \cite{LvovEtAl98}; see e.g. \cite{Majda16} for further discussions on the motivations for studying this class of SDEs.
It is straightforward to show that the SDEs are globally well-posed and the associated Markov semigroups are well-behaved; see e.g. [Appendix A; \cite{BL21}]. 
We will refer to the ODE
\begin{align*}
\frac{d}{dt} z_t = B(z_t,z_t)
\end{align*}
as the \emph{conservative dynamics}.
This deterministic ODE plays a distinguished role, as it is the leading order dynamics at high energies, i.e. when $\abs{x} \gg 1$.

Denote the generator
\begin{equation} \label{eq:L}
\mathcal{L} = \frac{1}{2} \sigma \sigma^T : \grad^2  - Ax\cdot \grad + B(x,x)\cdot \grad
\end{equation}
and the associated Markov semigroups $\mathcal{P}_t = e^{t \mathcal{L}}$ and $\mathcal{P}_t^\ast = e^{t \mathcal{L}^\ast}$, the former acting on the space of bounded, Borel measurable observables $B_b(\R^n;\R)$ and the latter acting on Borel probability measures $\mathcal{P}(\R^n)$. 
When $A$ is positive definite, it is not hard to prove that there always exists at least one stationary measure, i.e. a measure $\mu \in \mathcal{P}(\R^n)$ such that $\mathcal{P}_t^\ast \mu = \mu$.
This is proved by the Krylov-Bogoliubov procedure (see e.g. \cite{DPZ96}) combined with the following energy balance obtained from It\^o's lemma:
\begin{align*}
\frac{1}{2}\EE \abs{x_t}^2 + \EE \int_0^t x_s \cdot A x_s \dee s = \frac{t}{2}\sum_{i,j=1}^n \sigma_{ij}^2 + \frac{1}{2}\EE \abs{x_0}^2.
\end{align*}
However, if $\mathrm{ker} A \neq \emptyset$, then there is the possibility that energy could accumulate in these degrees of freedom and the a priori estimate 
\begin{align} \label{eq:energy}
\limsup_{t \to \infty}\frac{1}{t}\EE \int_0^t x_s \cdot A x_s \dee s \leq \frac{1}{2}\sum_{i,j=1}^n \sigma_{ij}^2 
\end{align}
would not be sufficient to imply the compactness required for Krylov-Bogoliubov. 

It is well known that to prove the existence of a stationary measure it suffices to construct a \textit{Lyapunov function}, i.e., a $C^2$ function $V:\R^n \to [0,\infty)$ satisfying $\lim_{|x| \to \infty} V(x) = \infty$ and 
\begin{equation} \label{eq:Lyapgeneral}
\mathcal{L} V \le -\alpha V^p + \beta
\end{equation}
for some $\alpha, \beta > 0$ and $p \in (0,1]$. Indeed, this is a straightforward generalization of the argument recalled above using It\^{o}'s lemma and the Krylov-Bogoliubov procedure. Note that if the kernel of $A$ is trivial, then $V(x) = |x|^2$ is a Lyapunov function for \eqref{eq:SDE}, while if $\mathrm{ker}A \neq \emptyset$ then \eqref{eq:Lyapgeneral} holds only in regions where $|x| \leqc |\Pi_{\mathrm{ker}A^\perp} x|$. There are many works that have successfully constructed an invariant measure and/or obtained convergence rates to equilibrium for SDEs with partial dissipation or unstable deterministic dynamics by building a nontrivial Lyapunov function (see e.g. \cite{AKM12, HM15, FGHH21, herzog2019ergodicity, WilliamsonThesis, BirrellTransition, S93}). A general strategy for constructing a Lyapunov function is to patch together a sequence of local Lyapunov functions, each satisfying \eqref{eq:Lyapgeneral} in a different part of phase space. In regions where \eqref{eq:Lyapgeneral} is not obviously satisfied by some natural energy-type function, a common approach is to perform a scaling analysis and show \eqref{eq:Lyapgeneral} for a reduced generator, and then justify the full inequality by an approximation argument. For a discussion of scaling arguments and a meta-algorithm for constructing Lyapunov functions, see \cite{AKM12}. The Lyapunov functions obtained by such methods tend to be quite involved, even in low dimensional, relatively simple systems (see e.g. \cite{AKM12, HM15, FGHH21} and [section 2, \cite{WilliamsonThesis}]), and require a careful gluing of separate local Lyapunov functions. 

In this paper, we develop a framework for constructing invariant measures for partially damped systems based on returning to the simple a priori energy estimate \eqref{eq:energy}. Rather than directly building a Lyapunov function, the idea is to recover compactness by proving that the time-averaged dissipation controls the average of some simple coercive function. More precisely, our strategy is to prove the following time-averaged coercivity estimate for some $T \in (0,2)$ and $r \in (0,1]$, 
\begin{align}
\frac{1}{T}\EE \int_0^T \brak{x_t}^{2r} \dee t \lesssim 1 +  \frac{1}{T}\EE \int_0^T x_t \cdot Ax_t \dee t, \label{ineq:timeavcoer}
\end{align}
which we show is sufficient to imply existence in Lemma~\ref{lem:time-averaged_main} by a straightforward iteration procedure. In Lemma~\ref{lem:time-averaged_step1}, we reduce this to short-time coercivity estimates for trajectories starting in a relatively small neighborhood of $\mathrm{ker} A$ at high energy. Specifically, we show that it suffices to prove \eqref{ineq:timeavcoer} for initial conditions $x_0 \in \R^n$ satisfying $$|\Pi_{\mathrm{ker}A^\perp} x_0| \ll |\Pi_{\mathrm{ker}A}x_0|^r, \quad |x_0| \gg 1$$ 
and the time $T$ depending on the initial energy $|x_0|$. The goal is thus to prove that at high energies, where the conservative dynamics dominate, solutions that start near $\mathrm{ker}A$ must depart rapidly (on average) due to some kind of instability. Our strategy to prove the necessary time-averaged coercivity estimates is to use a suitable approximation of the solution when $|\Pi_{\mathrm{ker} A^\perp} x_t| \ll |\Pi_{\mathrm{ker} A} x_t|^r$, show that this approximate solution rapidly enters the region $|\Pi_{\mathrm{ker} A^\perp} x| \gtrsim |\Pi_{\mathrm{ker} A} x|^r$, and then argue that the approximation remains valid for as long as $|\Pi_{\mathrm{ker} A^\perp} x_t| \lesssim |\Pi_{\mathrm{ker} A} x_t|^r$. 

The time-averaged coercivity framework is convenient in that it allows one to leverage in a natural way assumptions on the instability of $\mathrm{ker}A$ under the dynamics to obtain existence of an invariant measure and an explicit convergence rate to equilibrium. Moreover, it avoids the need to carefully patch together separate local Lyapunov functions, which is required even if one uses a construction based on local exit times. We will showcase the flexibility of our methods by presenting a variety of examples to which they apply, in each case showing a different potential case that arises with degenerate damping. The examples below are chosen to show qualitatively distinct cases where the approximation procedure described above can be justified, although a different choice of approximate solution is used in each type of example. 

\subsection{Main results}
We now discuss our main results and their connection to some of the existing literature on related SDEs.

Below, denote the set of undamped configurations on the unit energy sphere by
\begin{align*}
\mathcal{U} = \mathrm{ker} A \cap \S^{n-1}. 
\end{align*}

The first theorem considers the case where $\mathcal{U}$ contains no sets which are invariant under the conservative dynamics.  
This case is analogous to the settings considered by hypocoercivity, which usually studies nontrivial interplay between degenerate elliptic operators and conservative first order operators (such as transport) to obtain decay estimates, despite the lack of coercivity; see discussions in e.g. \cite{HN04,Villani2009}.
Indeed, the results we are proving are quite similar to (sub-exponential) hypocoercivity results for the associated Markov semigroups (although here we use different, essentially probabilistic, methods). See \cite{BCG08,CCEH20} for further discussion on the relationship between Harris' theorems and commonly used hypocoercivity methods. 
The intuition is clear: if $\mathcal{U}$ contains no sets which are invariant under the conservative dynamics, then at high energies any trip to a small neighborhood of $\mathrm{ker} A$ must necessarily be short lived. 

\begin{theorem} \label{thm:hypo}
	Suppose that $\exists J \in \N$ such that $\forall x \in \mathcal{U}$, if $X_t$ solves the conservative dynamics
	\begin{equation} \label{eq:transversedeterm}
	\begin{cases}
	\frac{d}{dt} X_t = B(X_t,X_t) \\
	X_0 = x,
	\end{cases}
	\end{equation}
	then
	\begin{align}
	\exists j \leq J, \; \Pi_{\mathrm{ker} A^{\perp}} \frac{d^j}{dt^j} X_t |_{t = 0} \neq 0. \label{def:transverse}
	\end{align}
	Then, there exists at least one stationary measure $\mu$ and $\brak{x}^p \in L^1(\dee \mu)$ for all $p < \infty$. 
\end{theorem}
\begin{remark}
	Condition \eqref{def:transverse} implies that solutions to \eqref{eq:transversedeterm} that start on $\mathrm{ker} A$ instantly depart it (at least at a rate like $\gtrsim (Kt)^{J}$ if $\abs{x} \approx K$; see Lemma \ref{lem:transversegrowth}).
	Note that the condition in \eqref{def:transverse} is purely algebraic, that is, in principle it could be investigated using methods from algebraic geometry, rather than being an abstract condition on trajectories. Related algebraic conditions describing the instability of a set under some conservative dynamics have appeared in \cite{WilliamsonThesis}.
\end{remark}

\begin{remark}
	As to be expected, Theorem~\ref{thm:hypo} requires no assumptions on $\sigma$.
\end{remark}

In \cite{FGHH21}, the authors consider the stochastically driven Lorenz-63 model, a classical three dimensional model introduced in \cite{lorenz1967nature}. This model does not take exactly the form of \eqref{eq:SDE} due to the presence of a non-dissipative linear term, but the setting is essentially the same since there still exists a natural energy function that yields an invariant measure when $\mathrm{ker}A = \emptyset$. The authors consider the case where $\mathrm{ker} A = \mathrm{span} \set{e_k}$ for some canonical unit vector $e_k$ and consists of conservative equilibria that to leading order at high energies exhibit a Jordan block instability. They prove using a Lyapunov function approach that if the noise directly excites the instability, then there always exists a stationary measure. The next theorem is a similar kind of result but generalized to higher dimensional systems in which $\mathrm{ker} A  = \mathrm{span}\set{e_k}$ for $e_k$ a general unstable equilibrium point of the conservative dynamics. Unlike in the setting of Theorem~\ref{thm:hypo}, in this case we cannot depend purely on the conservative dynamics to simply transport the $x_t$ away from $\mathrm{ker} A$.
Instead, we must rely on the noise to push the dynamics off of the equilibrium and its stable manifold so that $x_t$ is repelled quickly from neighborhoods of $\mathrm{ker} A$ at high energy. We denote the (instantaneous) linearization of the conservative nonlinearity around any fixed $x$ as
\begin{align} \label{eq:Lxdef}
L_x v = B(x,v) + B(v,x)
\end{align}
and for the restriction to $\mathrm{ker}A^\perp$ we write 
$$ L_x^\perp v = \Pi_{\mathrm{ker}A^\perp}L_x \Pi_{\mathrm{ker}A^\perp}.$$
Recall that for simplicity we assume for now unless otherwise stated that $\mathrm{rank}(\sigma) = n$.

\begin{theorem} \label{thm:basic1}
	Suppose that $\mathcal{U} = \set{x_0,-x_0}$ for some unit vector $x_0$  and that for each $x \in \mathcal{U}$ there holds
	\begin{align*}
	B(x,x) = 0, \quad \lim_{t \to \infty} \norm{e^{tL^\perp_{x}}} = \infty.
	\end{align*}
	Then, there exists at least one stationary measure $\mu$ and $\brak{x}^p \in L^1(\dee \mu)$ for all $ p < 1/3$. 
\end{theorem}

\begin{remark}
	If $x_0$ and $-x_0$ are spectrally unstable, i.e. $L_{\pm x_0}$ has an eigenvalue $\lambda$ with $\mathrm{Re} \lambda > 0$, then the stationary measure satisfies $\brak{x}^p \in L^1(\dee \mu)$ for all $p < \infty$.
	Notice however, that in general we do \emph{not} require that $x_0$ is spectrally unstable, that is, it is sufficient for the equilibria to have an $O(t)$ growth coming from a non-trivial Jordan block. We did not take care in this paper to optimize the moment bounds on the stationary measures that we construct and in general they are probably far from sharp. For example, it is likely that $\mu$ has exponential moments in many cases. In fact, the existence of an invariant measure with exponential moments was proven for a 3d model satisfying the conditions of Theorem~\ref{thm:basic1} in \cite{WilliamsonThesis}.
\end{remark}

\begin{remark}
	The condition $\mathrm{rank}(\sigma) = n$ is not necessary. What is used in the proof is essentially that the range of $\sigma$ contains at least one eigenvector or generalized eigenvector associated with the fastest instability of $L^\perp_x$. For the precise statement of Theorem~\ref{thm:basic1} with weaker assumptions on $\sigma$, see Theorem~\ref{thm:basic1sigma}. In fact, none of the theorems we prove require the forcing to act on all variables.
	We expect that all of the theorems that rely on unstable equilibria hold only under the assumption that the forcing is hypoelliptic if all of the instabilities are spectral, however, we did not pursue this direction here. Similarly, we expect variations of these results to be valid with multiplicative stochastic forcing under suitable assumptions. 
\end{remark}

We can also treat cases with $\mathrm{dim}(\mathrm{ker}A)>1$ provided that $\mathcal{U}$ consists either entirely of spectrally unstable equilibria or Jordan block unstable equilibria. In the latter case we require an additional cancellation condition due to the slower timescale of the instability.
\begin{theorem} \label{thm:basic2}
	Suppose that $B(x,x) = 0$ for every $x \in \mathcal{U}$ and that there exists a constant $C > 0$ so that 
	\begin{equation} \label{eq:JNFAssumption}
	\sup_{x \in \mathcal{U}}(\|P_x\| + \|P^{-1}_x\|) \le C,
	\end{equation}
	where $J_x^\perp = P_x^{-1}L_x^\perp P_x$ is the Jordan canonical form of $L_x^\perp$. Then, we have the following results.
	\begin{itemize}
		\item If for every $x \in \mathcal{U}$ there is an eigenvalue of $L_x^\perp$ with positive real part, then there exists at least one stationary measure $\mu$ and $\brak{x}^p \in L^1(\dee \mu)$ for every $p < 2/3$.
		\item Assume that
		$$\Pi_{\mathrm{ker}A}\left(B(\Pi_{\mathrm{ker}A} x, \Pi_{\mathrm{ker}A^\perp} x) + B(\Pi_{\mathrm{ker}A^\perp} x, \Pi_{\mathrm{ker}A} x)\right) = 0 $$
		for every $x \in \R^n$. If for every $x \in \mathcal{U}$ there exists $J \in \{1,2,\ldots,n-2\}$ so that there holds
		$$t^J \leqc_x \|e^{tL_x^\perp}\| \leqc_x (1+t^J)$$
		for all $t \ge 0$, then there exists at least one stationary measure $\mu$ and $\brak{x}^p \in L^1(\dee \mu)$ for every $p < 1/3$.
	\end{itemize}
\end{theorem}

\begin{remark}
	Analogous criteria to Theorems \ref{thm:basic1}, \ref{thm:basic2}, and Theorem \ref{thm:hypo} can be found for much more general nonlinearities, i.e. systems of the form $dx_t = F(x_t) dt - Ax_t + \sigma dW_t$ with $x \cdot F(x) = 0$, however, the lack of scaling invariance requires slightly more care.
\end{remark}

A first natural question is whether or not Theorem \ref{thm:hypo} and Theorems \ref{thm:basic1},\ref{thm:basic2} can be combined into one.
We do not know how to do this in reasonable generality due to difficulties in dealing with transitions between ``transverse'' regions as in Theorem \ref{thm:hypo}
and unstable equilibria as in Theorems \ref{thm:basic1} and \ref{thm:basic2}.  However, in Section~\ref{sec:combination} we prove Theorem~\ref{thm:combination}, which provides at least one general setting where this is possible. Specifically, we consider systems for which $\mathrm{ker}A = V_1 \oplus V_2$ for subspaces $V_1, V_2 \subseteq \R^n$ consisting of spectrally unstable equilibria and such that the region where $\Pi_{V_1}x$ and $\Pi_{V_2}x$ are both sufficiently large can be treated as a transverse zone. Note that in this setting the instability of $\Pi_{V_j}x$ need not cause growth of the damped modes directly, but could instead cause the solution to enter a transverse region, where it is then subsequently expelled from $\mathrm{ker}A$ in a manner similar to Theorem~\ref{thm:hypo}. While we require some additional structural assumptions to justify the approximations, Theorem~\ref{thm:combination} applies to several well-known examples,
for example the Sabra model with $\mathrm{ker} A$ given by the first two frequency shells (which means $\mathrm{dim} (\mathrm{ker} A) = 4$) and the 2d Galerkin-Navier-Stokes equations with $\mathrm{ker} A$ consisting of a four-dimensional subspace of suitably chosen shear flows. We will state here our result on the Navier-Stokes equations, and defer the general result and application to Sabra to Section~\ref{sec:combination}.

Recall the 2d Navier-Stokes equations in vorticity form on a square torus $\mathbb T^2$ subjected to stochastic forcing:
\begin{align*}
d  w + (u \cdot \grad w- \Delta w) dt & =  \sum_{k \in \mathbb Z^2 : k \neq 0} \sigma_k^{(1)} \cos(k \cdot x) dW_t^{(k;1)} + \sigma_k^{(2)} \sin(k \cdot x) dW_t^{(k;2)} \\
u & = \begin{pmatrix} -\partial_{x_2} \\ \partial_{x_1} \end{pmatrix} (-\Delta)^{-1} w. 
\end{align*}
Let $\Pi_{\leq N}$ be the projection to the modes such that $\max(\abs{k_1},\abs{k_2}) =: \abs{k}_\infty  \leq N$ (any choice of $\ell^p$ works). 
Then the Galerkin Navier-Stokes equations are given by the SDE defined for mean-zero $w \in \mathrm{Im} \Pi_{\leq N}$ by 
\begin{align*}
d w + (\Pi_{\leq N}(u \cdot \grad w) + Aw) dt  & =  \sum_{0 < \abs{k}_\infty \leq N} \sigma_k^{(1)} \cos(k \cdot x) dW_t^{(k;1)} + \sigma_k^{(2)} \sin(k \cdot x) dW_t^{(k;2)} \\
u & = \begin{pmatrix} -\partial_{x_2} \\ \partial_{x_1} \end{pmatrix} (-\Delta)^{-1} w, 
\end{align*}
where we have replaced the matrix $-\Pi_{\leq N} \Delta \Pi_{\leq N}$ with a general positive semi-definite matrix $A$. 
\begin{theorem}\label{thm:NS}
	Let $N \geq 3$ be arbitrary and define the two subspaces of $\mathrm{Im} \Pi_{\leq N}$
	\begin{align*}
	V_1 \oplus V_2 = \mathrm{span}(\cos \ell x_1, \sin \ell x_1) \oplus \mathrm{span}(\cos k x_2, \sin k x_2), 
	\end{align*}
	for two arbitrary integers $\ell,k \geq 2$ such that $\ell \neq k$ and $\max(\ell,k) \leq N$.
	Suppose further that the forcing coefficients $\sigma_p^{(j)}$ are all non-zero. If $\mathrm{ker}A = V_1 \oplus V_2$ then there exists a (unique) invariant measure $\mu_\ast$ of the Galerkin Navier-Stokes equations with truncation $N$ and for all $p < 2/3$ there holds 
	\begin{align*}
	\int_{\mathrm{Im} \Pi_{\leq N}} \abs{w}^p d\mu_\ast  < \infty. 
	\end{align*}
\end{theorem}

As an additional example in a setting similar to Theorem~\ref{thm:combination} described above, we consider the Lorenz-96 model, put forward by Lorenz in \cite{Lorenz1996}, for $n$ real-valued unknowns $u_1,...,u_n$ in a periodic ensemble $u_{i+kn} = u_i$:
\begin{align} 
d u_m = (u_{m+1}-u_{m-2})u_{m-1} dt - (Au)_m dt + q_m dW_t^{(m)}. \label{def:L96punch}
\end{align}
Here, $\set{W_t^{(m)}}$ are independent Brownian motions and $\set{q_m}$ are fixed parameters. This model has been studied as a prototypical high dimensional chaotic system (see e.g. \cite{Majda16,KP10,LK98}). We consider \eqref{def:L96punch} with 
$$ \mathrm{ker}A = \{u_1 = u_2 = 0\}.$$
Similar to the general setting of Theorem~\ref{thm:combination}, this example contains a mixture of all of Theorems \ref{thm:hypo}, \ref{thm:basic1}, and \ref{thm:basic2}
in the sense that $\mathcal{U}$ contains both unstable equilibria and a region in which the conservative dynamics expel from $\mathrm{ker} A$ as in Theorem \ref{thm:hypo}. However, the equilibria are only Jordan block unstable, so Theorem~\ref{thm:combination} (the proof of which relies crucially on the exponential instability of the equilibria in $V_j$) does not apply. The linear instability of the equilibria defined by $u = \alpha e_2$ (i.e. only supported in the second mode) causes growth of the $e_1$ direction, rather than a mode in $\mathrm{ker}A^\perp$.
In this region of $\mathrm{ker}A$, a careful (and somewhat nonlinear) argument is used to show that the linear instability moves the dynamics into a region where the nonlinearity can then transport the dynamics out of $\mathrm{ker} A$. Despite the lack of unstable eigenvalues, using the precise structure of \eqref{def:L96punch} we can justify the approximations needed to apply our methods and construct an invariant measure.

\begin{theorem} \label{thm:L96}
	Let $6 \leq n < \infty$ and suppose that $q_{n-1},q_n$ are both non-zero. Suppose that $\mathrm{ker} A = \set{u_1 = u_2 = 0}$. Then, \eqref{def:L96punch} admits at least one stationary measure $\mu$ and $\brak{x}^p \in L^1(\dee \mu)$ for all $p < 1/3$.
\end{theorem}
\begin{remark}
	After completion of this work, we have been made aware of a similar result for Lorenz-96 in the upcoming thesis \cite{EvanThesis}, which considers the case where $n = 4$, $\mathrm{ker}A$ consists of two modes, and the forcing acts only on the two modes in $\mathrm{ker}A^\perp$. 
\end{remark}

The above theorems do not contain all of the interesting possible relationships between $\mathrm{ker}A$ and the dynamics of $B$. In particular, none of the above examples consider a case in which $\mathcal{U}$ contains a non-equilibrium invariant set for the conservative dynamics. We give one such example where our methods apply, based on the following simple ``stochastic triad'' model \cite{Majda16} defined by the nonlinearity
\begin{align}
B(x,x) =
\begin{pmatrix}
x_2x_3 \\ x_1x_3 \\ -2 x_1 x_2
\end{pmatrix}
. \label{def:Bbb}
\end{align}
The $x_3$-axis contains unstable equilibria and so Theorem \ref{thm:basic1} shows that if $\mathrm{ker}A = \mathrm{span}\set{e_3}$, then there exists a stationary measure (this result was already proven in \cite{WilliamsonThesis}). To contrast, the plane defined by $\set{x : x_1=x_2}$ consists of heteroclinic connections between the unstable equilibria with $x_3 > 0$ and those with $x_3 < 0$, and so neither Theorem \ref{thm:basic1} nor Theorem \ref{thm:basic2} apply to the case that $\mathrm{ker} A = \set{x : x_1=x_2}$.
Nevertheless, we are able to adapt our methods to cover this case since we can precisely describe the conservative dynamics restricted to $\mathrm{ker} A$.
\begin{theorem} \label{thm:BadBad}
	Consider the stochastic triad model defined by \eqref{def:Bbb} in $\R^3$ and suppose $\mathrm{ker} A = \set{x : x_1 = x_2}$. Then, there exists at least one stationary measure $\mu$ and $\brak{x}^p \in L^1(\dee \mu)$ for all $p < 2/3$.
\end{theorem}

In all of the above examples, existing results give uniqueness and regularity of the stationary measure once existence is proved; see e.g. the Doob-Khasminskii theorem \cite{DPZ96}.
Moreover, the proof yields a sub-geometric Lyapunov function and one can apply a suitable variation of Harris' theorem to obtain explicit convergence estimates on the Markov semigroups in the total variation norm \cite{DFG09}. 
\begin{corollary} \label{cor:SGD}
	If $\mathrm{rank}\,(\sigma) = n$ and if $x_t$ is irreducible, that is, if $\forall x \in \R^n$, open sets $\mathcal{O} \subset \R^n$ and $\forall t > 0$, 
	\begin{align*}
	\PP\left(x_t \in \mathcal{O} \right) > 0, 
	\end{align*}
	then in any of the above examples, there is a \emph{unique} stationary measure $\mu$, and this stationary measure is $C^\infty$.
	Moreover, if $V(x) = \brak{x}^2$, then for $T$ and $r$ as in \eqref{ineq:timeavcoer}, 
	\begin{align*}
	\tilde{V}(x) = \frac{1}{T}\int_0^T \mathcal{P}_t V(x) dt
	\end{align*}
	satisfies for $r$ as above and some constants $c,C > 0$,
	\begin{align*}
	\mathcal{L} \tilde{V} \leq -c\tilde{V}^r + C, 
	\end{align*}
	and hence by results in \cite{DFG09}, for any $x \in \R^n$, there holds (with the convention that if $r =1$, then the decay is exponential)
	\begin{align*}
	\norm{\mathcal{P}_t(x,\cdot) - \mu}_{TV} \lesssim \brak{t}^{-\frac{r}{1-r}}\tilde{V}(x). 
	\end{align*}
\end{corollary}
\begin{remark}
	A simple energy estimate (Lemma \ref{lem:basicenergy}) shows that necessarily $\tilde{V}(x) \gtrsim \brak{x}^2$. 
\end{remark} 

\subsection{Discussion and related work}

As alluded to above, the proofs of Theorems \ref{thm:hypo}--\ref{thm:BadBad} are all about ruling out the possibility that energy accumulates into $\mathrm{ker} A$ which is be done by demonstrating a time-averaged coercivity estimate of the form \eqref{ineq:timeavcoer}. 
For \eqref{ineq:timeavcoer} to hold, we see that it would suffice to show that the solution does not spend a significant percentage of its time near $\mathrm{ker} A$.
In fact, at higher energies, we show that the dynamics are expelled from neighborhoods of $\mathrm{ker} A$ faster. 
This has a clear analogy with variations of hypocoercivity that emphasize this aspect (see discussions in \cite{Villani2009}), however, these previously existing works are all essentially in the case of Theorem \ref{thm:hypo}.

Section \ref{sec:Prelim} provides two important lemmas: Lemma \ref{lem:time-averaged_main} shows that \eqref{ineq:timeavcoer} suffices to prove the existence of stationary measures and Lemma \ref{lem:time-averaged_step1} reduces this to short-time coercivity estimates in a small region of $\mathrm{ker} A$ (see Assumption \ref{ass:time-average}). 
Moreover, \eqref{ineq:timeavcoer} implies the existence of a sub-geometric drift condition as pointed out in Corollary \ref{cor:SGD}. 

In order to prove \eqref{ineq:timeavcoer} (via Lemma \ref{lem:time-averaged_step1}), it makes sense to proceed by contradiction.
When \eqref{ineq:timeavcoer} fails at high energy, it is necessary for the majority of the energy to be concentrated in a small region around $\mathrm{ker} A$, which could allow a perturbative treatment for as long as the dynamics remain in the small region.
Theorem \ref{thm:hypo} simply uses the pure conservative dynamics as the approximate solution, whereas Theorems \ref{thm:basic1} and \ref{thm:basic2} use the linearization around $\Pi_{\mathrm{ker} A}x_0$ (frozen in time) to justify the expulsion. Notice that the noise remains important here near the stable manifold of the equilibria.
Theorems \ref{thm:NS}, \ref{thm:L96}, and \ref{thm:BadBad} use more careful approximations based on what region of $\mathrm{ker} A$ the solution is close to.
For example,  to prove Theorem \ref{thm:BadBad}, if one is near an equilibrium  $(0,0,x_3)$ with $x_3 > 0$, then we first prove that with high probability the solution is rapidly transported away along the heteroclinic connections that run through $\set{x:x_1 = x_2}$, and then show that it is likely to be expelled from $\mathrm{ker} A$ along the unstable manifold of the corresponding equilibrium at $(0,0,-x_3)$.  

We remark here that many aspects of our work are not specific to the system \eqref{eq:SDE} and could easily be adapted to various other regimes.
However, systems with an underlying conservative dynamics which is a homogeneous polynomial (and hence a scale invariance is available), linear damping, and additive noise seem to be the simplest case to consider. 

There are several works in the literature related to ours. The works that consider settings most similar to what we study here are \cite{FGHH21} and \cite{WilliamsonThesis}. In addition to the existence result discussed above around Theorem~\ref{thm:basic1}, it is proven in \cite{FGHH21} that if the noise does not excite the instability and $e_k$ is directly forced, then no stationary measure exists. The work \cite{WilliamsonThesis} considers \eqref{eq:SDE} with an additional structural assumption on $B$ motivated by the nonlinearity in the Navier-Stokes equations. When the deterministic invariant subset of $\mathrm{ker}A$, denoted by $\mathcal{N}$, consists only of spectrally unstable equilibria, existence of an invariant measure is proven under an algebraic assumption that describes growth of the damped modes for initial conditions near $\mathcal{N}$. In the context of our work, this main result of \cite{WilliamsonThesis} seems closely related to Theorem~\ref{thm:combination} and can be viewed essentially as sufficiently strict assumptions under which the combination of Theorems~\ref{thm:hypo} and~\ref{thm:basic2} is possible.

A set of works with close links to ours considers noise-induced stabilization for systems with deterministic dynamics that contain finite-time blow-up solutions; see e.g. \cite{S93,M20,AKM12,HM15}. In these works, despite the finite-time blow-up of certain deterministic trajectories, depending on the noise or whether the blow-ups are unstable, one can nevertheless obtain almost-sure global well-posedness and prove the existence of stationary measures.
The works using additive noise proceed by a Lyapunov function approach and so are closely related to \cite{FGHH21}.  
Another related work is that of Coti Zelati and Hairer \cite{CZH21}, which considers the Lorenz-63 system with  $\mathrm{ker} A = \emptyset$, but where the forcing \emph{only} acts on $\mathrm{span} \set{e_3}$.
This makes $\mathrm{span}\set{e_3}$ an almost-surely invariant set for the stochastically forced system, in which case, an argument based on transverse Lyapunov exponents can be made, providing another method for dynamically driving solutions away.

\section{Time-averaged coercivity near $\mathrm{ker} A$} \label{sec:Prelim}

The purpose of this section is to prove a useful general result that will be applied to construct an invariant measure in each of the examples discussed in Section~\ref{sec:Intro}.
The main abstract condition for the existence of invariant measures is stated below as Assumption~\ref{ass:time-average}.
Intuitively, the condition requires that if the process enters the vicinity of $\mathrm{ker}A$ at high energies, then it is quickly ejected and subsequently stays away from $\mathrm{ker}A$ for some amount of time.

\subsection{Time-averaged coercivity implies existence} \label{sec:TAC}

In what follows denote
\begin{align*}
D(x) := x \cdot A x. 
\end{align*}
Notice that
\begin{align*}
|\Pi_{\mathrm{ker}A^\perp} x|^2 \lesssim D(x) \lesssim |\Pi_{\mathrm{ker}A^\perp} x|^2.
\end{align*}

We begin with a preliminary lemma which reduces the existence of a stationary measure to the kind of time-averaged coercivity alluded to in \eqref{ineq:timeavcoer}. 
The proof follows in a straightforward way from the Krylov-Bogoliubov procedure and the energy conservation property of $B$, however, we include it for the sake of completeness.

\begin{lemma} \label{lem:time-averaged_main}
	Let $V(x) = \brak{x}^2$ and let $\Pt_t$ be the Feller Markov semigroup on $\R^n$ generated by $\mathcal{L}$ (defined in \eqref{eq:L}).
	If there exists $r \in (0,1]$, $1 \le p < (1-r)^{-1}$, $C > 0$ and $T \in (0,2)$ such that
	\begin{align}
	\frac{1}{T}\int_0^T \Pt_t V^{rp}(x) \dee t \leq C\left(1 + \frac{1}{T}\int_0^T \Pt_t D^p(x) \dee t\right) \quad \forall x \in \R^n, \label{ineq:TAC}
	\end{align}
	then there exists at least one stationary measure $\mu_\ast$ of $\Pt_t$ such that $V^{rp} \in L^1(\dee \mu_\ast)$. Moreover, there exist $\alpha, \beta > 0$ such that the function 
	$$ \tilde{V}(x) = \frac{1}{T} \int_0^T \Pt_t V^p(x)dt$$
	satisfies
	\begin{equation}\label{eq:subgeometric1}
	\mathcal{L} \tilde{V} \le -\alpha \tilde{V}^{r} + \beta.
	\end{equation}
	That is, $\tilde{V}$ is a sub-geometric (when $r < 1$) Lyapunov function.
\end{lemma} 

\begin{proof}
	Let $\mu$ be any Borel measure on $\R^n$ with $\int V^p(x) \mu(dx) < \infty$ and let $T$ be given as in the assumption.
	We first claim that for every $n \in \N$ there holds 
	\begin{equation} \label{eq:nT}
	\int_0^{nT} \int V^{rp}(x) \Pt_s^* \mu(dx) ds \le C\left(n T+ \int_0^{nT} \int D^p(x)\Pt_s^* \mu(dx)ds\right). 
	\end{equation}
	By the assumption \eqref{ineq:TAC} and Fubini's theorem, we have 		
	\begin{align*}
	\int_0^T \int V^{rp}(x)\Pt_s^* \mu(dx)ds = \int \left(\int_0^T \Pt_s V^{rp}(x)ds\right)\mu(dx) \le C\left(T + \int \left(\int_0^T \Pt_s D^p(x) ds\right)\mu(dx)\right),
	\end{align*}
	and hence 
	\begin{equation} \label{eq:1}
	\int_0^T \int V^{rp}(x)\Pt_s^* \mu(dx)ds\le C\left(T+\int_0^T \int D^p(x)\Pt_s^*\mu(dx)ds\right).
	\end{equation}
	By the semigroup property and \eqref{eq:1}, for any $m \in \N$ we have
	\begin{align}
	\int_{ m T }^{(m+1)T} \int  V^{rp}(x) \Pt_s^* \mu(dx) ds
	&  \le C\left(T + \int_0^T \int D^p(x) \Pt_s^*(\Pt_{mT}^* \mu)(dx) ds\right) \nonumber \\ 
	& = C\left(T + \int_{mT}^{(m+1)T} \int D^p(x)\Pt_s^* \mu(dx)ds\right). \label{eq:mtomplus1}
	\end{align}
	Summing \eqref{eq:mtomplus1} over $0 \le m < n$ yields \eqref{eq:nT}.
	
	Next, notice a direct computation using $B(x,x) \cdot x = 0$ shows that 
	\begin{align} \label{eq:LVp}
	\mathcal{L} V^{p}(x) \le C_1 V(x)^{p-1} - C_2 D^p(x)
	\end{align}
	for some constants $C_1, C_2 > 0$. Thus,
	\begin{equation}
	\frac{d}{dt} \int V^p(x) \Pt_t^* \mu(dx) \le C_1\int V(x)^{p-1}\Pt_t^* \mu(dx) - C_2 \int D^p(x) \Pt_t^* \mu(dx). 
	\end{equation}
	Let $n \in \N$. Integrating the previous inequality over $0 \le t \le nT$ we see that
	\begin{equation} \label{eq:dissipationaverage}
	\int_0^{nT} \int D^p(x)\Pt_s^*\mu(dx)ds \lesssim 1 + \int_0^{nT}\int V(x)^{p-1}\Pt_s^*\mu(dx)ds,
	\end{equation}
	where the implicit constant depends on $\int V^p(x) \mu(dx)$. 
	Applying \eqref{eq:nT} and \eqref{eq:dissipationaverage} gives 
	\begin{equation} \label{eq:tobeabsorbed}
	\int_0^{nT} \int V^{rp}(x) \Pt_s^*\mu(dx) ds \le C(1+nT) + C\int_0^{nT}\int V(x)^{p-1}\Pt_s^*\mu(dx)ds. 
	\end{equation}
	The choice $p < (1-r)^{-1}$ ensures that $p-1 < rp$, and so for every $\epsilon > 0$ there is $C_\epsilon$ such that
	$$V^{p-1} \le \epsilon V^{rp} + C_\epsilon.$$
	Hence, the integral on the right-hand side of \eqref{eq:tobeabsorbed} can be absorbed into the left-hand side, yielding
	$$ 	\int_0^{nT} \int V^{rp}(x) \Pt_s^*\mu(dx) ds \lesssim (1+nT).$$
	Therefore 
	\begin{equation} \label{eq:tight}
	\sup_{n \in \N} \frac{1}{nT} \int_0^{nT} \int V^{rp}(x)\Pt_s^*\mu(dx)ds < \infty.
	\end{equation}
	Using the tightness implied by \eqref{eq:tight}, the existence of a stationary measure $\mu_*$ with $V^{rp} \in L^1(d\mu_*)$ follows by the usual Krylov-Bogoliubov method (see e.g. \cite{DaPratoZabczyk1996}). 
	
	It remain to prove \eqref{eq:subgeometric1}. First, by \eqref{eq:LVp}, \eqref{ineq:TAC}, and $p<(1-r)^{-1}$ there exist constants $c,C > 0$ such that
	\begin{equation} \label{eq:Lyapunov1}
	\mathcal{L} \tilde{V} \le - \frac{c}{T}\int_0^T \Pt_t V^{rp} dt + C.
	\end{equation}
	Next, one can show using Gr\"{o}nwall's lemma that for any $q \ge 1$ there exists $C_q \ge 1$ such that 
	$$
	\brak{x}^{2q} - C_qt(\brak{x}^{2q}+1) \le \Pt_t V^q(x) \le C_q(\brak{x}^{2q} + 1) \quad \forall t\in [0,2].
	$$
	It follows that for all $x$ with $|x|$ sufficiently large and $T \in (0,2)$ there holds 
	\begin{equation} \label{eq:timeaverageequiv}
	\frac{1}{T} \int_0^T \Pt_t V^q(x) dt \approx_q \brak{x}^{2q}.
	\end{equation}
	Using \eqref{eq:timeaverageequiv} in \eqref{eq:Lyapunov1} completes the proof.
\end{proof}

\subsection{Short-time coercivity near $\mathrm{ker} A$}

Next, we formulate a sufficient condition for \eqref{ineq:TAC} based on short-time (time-averaged) coercivity of solutions near $\mathrm{ker} A$.
Intuitively, this is similar to estimating average exit times from the vicinity of $\mathrm{ker} A$, but not quite the same.  

\begin{assumption} \label{ass:time-average}
	Let $\Pt_t$ be the Feller Markov semigroup on $\R^n$ generated by $\mathcal{L}$ defined in \eqref{eq:L} . 
	We say that $\Pt_t$ satisfies Assumption~\ref{ass:time-average} if there exist $r \in (0,1]$, $K_* \ge 1$, $c_* > 0$,  $\delta \in (0,1)$, and a finite collection of times $\eta_j: [0,\infty) \to (0,1]$, $1 \le j \le m$, such that:
	\begin{itemize}
		\item $\lim_{K \to \infty} \sup_{1\le j \le m}\eta_j(K) = 0$;
		\item for every $K\ge K_*$ the set 
		$$ B_K = \{x \in \R^n: |\Pi_{\mathrm{ker}A^\perp} x|^2 \le \delta |\Pi_{\mathrm{ker}A}x|^{2r} \text{ and } (1-\delta) K^2 \le |x|^2 \le (1+\delta)K^2\}$$
		admits a decomposition 
		$$ B_K = \bigcup_{j=1}^m B_{K,j}$$
		for which $B_{K,j} \subseteq \R^n$ is such that $x \in B_{K,j}$ implies 
		\begin{equation} \label{eq:time-averaged_main}
		\frac{1}{\eta_j(K)} \int_0^{\eta_j(K)}\Pt_t D(x)dt \ge c_* K^{2r}. 
		\end{equation}
	\end{itemize}
\end{assumption}

\begin{remark}
	By the H\"{o}lder and Jensen inequalities, \eqref{eq:time-averaged_main} implies that for any $p \ge 1$ and $x \in B_{K,j}$ there holds
	\begin{equation}\label{eq:ass1expectation}
	\frac{1}{\eta_j(K)}\int_0^{\eta_j(K)}\Pt_t D^p(x)dt \ge c_*^p K^{2rp}.
	\end{equation}
\end{remark}

The main result of this section is the following lemma, which shows that Assumption~\ref{ass:time-average} implies \eqref{ineq:TAC}, which is a sufficient condition for the existence of an invariant measure due to Lemma \ref{lem:time-averaged_main}. 

\begin{lemma} \label{lem:time-averaged_step1}
	Suppose that Assumption~\ref{ass:time-average} holds for some $r \in (0,1]$. Then, there exists at least one stationary measure $\mu_*$ of $\Pt_t$ and $\brak{x}^q\in L^1(d \mu_*)$ for every $q < 2r/(1-r)$. Moreover, for every $1 \le p < (1-r)^{-1}$ the function $\tilde{V}$ defined in Lemma~\ref{lem:time-averaged_main} satisfies \eqref{eq:subgeometric1}.
\end{lemma}

Before proceeding to the details of the proof we give a few remarks on the intuition behind Lemma~\ref{lem:time-averaged_step1}. As we will see at the beginning of the proof below, the lemma reduces to showing that there are $T \in (0,2)$ and $C > 0$ so that for every $x \in \R^n$ with $|x| = K \gg 1$ there holds 
\begin{equation} \label{eq:time_average_intuition}
K^{2rp} \le \frac{C}{T}\int_0^T \Pt_t D^p(x)dt. 
\end{equation}
The idea behind proving \eqref{eq:time_average_intuition} is to first note that if the process ever enters the set
\begin{equation} \label{eq:tildeB}
\tilde{B} = \{z \in \R^n: |\Pi_{\mathrm{ker}A^\perp} z|^2 \le \delta |\Pi_{\mathrm{ker}A}z|^{2r}\}, 
\end{equation}
then, provided $T$ is small, by Lemma~\ref{lem:basicenergy} it is in $B_K$ with high probability.
Hence one can essentially assume that $x_t(\omega) \in B_K$ whenever $x_t(\omega) \in \tilde{B}$.
Now, the time average of $D(x_t)^p$ controls $K^{2rp}$ when the process is not in $\tilde{B}$ and, by the discussion above and Assumption~\ref{ass:time-average}, with high probability it controls $K^{2rp}$ on some short time interval if the process ever does enter $\tilde{B}$.
By tracking the return times of the process to $\tilde{B}$, \eqref{eq:time_average_intuition} follows by a suitable iteration of Assumption~\ref{ass:time-average}.

We now give the proof of Lemma~\ref{lem:time-averaged_step1}. 

\begin{proof}[Proof of Lemma~\ref{lem:time-averaged_step1}]
	First notice that by Lemma~\ref{lem:time-averaged_main} and \eqref{eq:timeaverageequiv} we just need to show that for every $1 \le p < (1-r)^{-1}$ there exists $C > 0$ and $T \in (0,2)$ so that for every $x \in \R^n$ 
	\begin{equation} \label{eq:time-averaged_step1_1}
	\brak{x}^{2rp} = V^{rp}(x) \le C\left(1 + \frac{1}{T}\int_0^{T} \Pt_t D^p(x)dt\right).
	\end{equation}

	We now set out to prove \eqref{eq:time-averaged_step1_1}. Let $\delta$, $c_*$, and $K_*$ be as in Assumption~\ref{ass:time-average} and let $x_t$ denote the solution to \eqref{eq:SDE} with $x_0 = x$. The bound \eqref{eq:time-averaged_step1_1} is trivial if $|x| < K_*$ because we can take $C$ sufficiently large depending on $K_*$, and so we need only consider when $|x| := K \ge K_*$. In this case we make precise the intuition described directly after the statement of the lemma. Let $\tilde{B}$ be as in \eqref{eq:tildeB} and define the sequence of stopping times $\tau_0(\omega) = 0$, $\tau_1(\omega) = \inf\{t \ge 0: x_t(\omega) \in \tilde{B}\} \wedge T$, and for $n \ge 1$ 
	$$ 
	\tau_{n+1}(\omega) =
	\begin{cases}
	\tau_n(\omega) & \text{if }x_{\tau_n(\omega)}(\omega) \not \in B_K \\ 
	\inf\{t \ge \tau_n(\omega)+\eta_1(K): x_{t}(\omega) \in \tilde{B}\} \wedge T & \text{if } x_{\tau_n(\omega)}(\omega) \in B_{K, 1} \\ 
	\vdots \\
	\inf\{t \ge \tau_n(\omega)+\eta_m(K): x_{t}(\omega) \in \tilde{B}\} \wedge T & \text{if } x_{\tau_n(\omega)}(\omega) \in B_{K, m}.
	\end{cases}
	$$
	Moreover, define
	$$ \bar{\tau}(\omega) = \inf\{t\ge 0: \left||x_t(\omega)|^2 - K^2\right| \ge \delta K^2\}.$$
	Due to Lemma~\ref{lem:basicenergy} applied with $\epsilon = \delta$, by taking $T$ sufficiently small and $K_*$ sufficiently large (both depending only on $\delta$) we may assume that 
	\begin{equation} \label{eq:taubarbound}
	\P(\bar{\tau} \ge T) \ge 1/2.
	\end{equation}
	With $T$ fixed, choosing $K_*$ perhaps even larger and recalling $\lim_{K' \to \infty} \sup_{1\le j \le m}\eta_j(K') = 0$ implies that we may assume
	\begin{equation} \label{eq:etabound}
	\sup_{1\le j \le m} \eta_j(K) \le T/2.
	\end{equation}
	Let now
	$$A_n = \{\omega \in \Omega: \tau_n(\omega) \le T/2\}$$
	and
	$$ A_{n,j} = \{\omega \in \Omega: x_{\tau_n(\omega)}(\omega) \in B_{K,j}\}.$$
	Using that $\tau_n$ is an increasing sequence with $\lim_{n \to \infty} \tau_n(\omega) \le T$ we have 
	\begin{align} 
	\int_0^T \Pt_t D^{p}(x)dt & \ge \int_\Omega \sum_{n=0}^\infty \int_{\tau_n(\omega)}^{\tau_{n+1}(\omega)} D^p(x_t(\omega))dt d\P \nonumber \\ 
	& = \int_\Omega \sum_{n=0}^\infty \int_{0}^{\tau_{n+1}(\omega) - \tau_n(\omega)} D^p(x_{\tau_n +t}(\omega))dt d\P \nonumber  \\ 
	& \ge \sum_{n=0}^\infty \sum_{j=1}^m \int_{A_{n,j} \cap A_n} \int_0^{\tau_{n+1}(\omega) - \tau_n(\omega)} D^p(x_{\tau_n +t}(\omega)) dt d\P. \label{eq:doublesum}
	\end{align}
	Now, if $x_{\tau_n(\omega)}(\omega) \in B_{K,j}$ and $\tau_n(\omega) \le T/2$, then $\tau_{n+1}(\omega) - \tau_n(\omega) \ge \eta_j(K)$ due to \eqref{eq:etabound} and the definition of $\tau_{n+1}$. For $1 \le j \le m$ we thus have
	\begin{equation} \label{eq:bothpieces}
	\begin{split} 
	&\int_{A_{n,j} \cap A_n} \int_0^{\tau_{n+1}(\omega) - \tau_n(\omega)} D^p(x_{\tau_n +t}(\omega)) dt d\P \\ 
	\quad \quad & = \int_{A_{n,j} \cap A_n}\int_0^{\eta_j(K)} D^p(x_{\tau_n +t}(\omega))dt d\P + \int_{A_{n,j} \cap A_n}\int_{\eta_j(K)}^{\tau_{n+1}(\omega) - \tau_n(\omega)} D^p(x_{\tau_n +t}(\omega))dt d\P.  
	\end{split}
	\end{equation}
	For the first piece, observe that $A_{n,j} \cap A_n$ is measurable with respect to $\mathcal{F}_{\tau_n}$, the $\sigma$-algebra of events determined prior to the stopping time $\tau_n$ (this agrees with the $\sigma$-algebra generated by $\{W_{s \wedge \tau_n}; s \ge 0\}$). Thus,
	$$
	\int_{A_{n,j} \cap A_n}\int_0^{\eta_j(K)} D^p(x_{\tau_n +t}(\omega))dt d\P  = \int_0^{\eta_j(K)} \int_{A_{n,j} \cap A_n} \E(D^p(x_{\tau_n+t})|\mathcal{F}_{\tau_n}) d\P dt.
	$$
	It then follows from the strong Markov property and Assumption~\ref{ass:time-average} in the form \eqref{eq:ass1expectation} that 
	\begin{equation}
	\label{eq:strongMarkov}
	\int_{A_{n,j} \cap A_n}\int_0^{\eta_j(K)} D^p(x_{\tau_n +t}(\omega))dt d\P =\int_{A_{n,j} \cap A_n} \int_0^{\eta_j(K)}  \Pt_t D^p (x_{\tau_n}) dt d\P \ge \eta_j(K) c_*^p K^{2rp}\P(A_{n,j} \cap A_n).
	\end{equation}
	For the second piece, first note that
	\begin{align*}
	\int_{A_{n,j} \cap A_n}\int_{\eta_j(K)}^{\tau_{n+1}(\omega) - \tau_n(\omega)} D^p(x_{\tau_n +t}(\omega))dt d\P & \ge \int_{A_{n,j} \cap A_n \cap \{\bar{\tau} \ge T\}}\int_{\eta_j(K)}^{\tau_{n+1}(\omega) - \tau_n(\omega)} D^p(x_{\tau_n +t}(\omega))dt d\P.
	\end{align*}
	Now, by construction, if $\omega \in A_{n,j} \cap A_n \cap \{\bar{\tau} \ge T\}$ then for each $t \in (\eta_{j}(K), \tau_{n+1}(\omega) - \tau_n(\omega))$ one has 
	$$ |\Pi_{\mathrm{ker}A^\perp} x_{\tau_n + t}(\omega)|^2 \ge \delta |\Pi_{\mathrm{ker}A} x_{\tau_n + t}(\omega)|^{2r}   \quad \text{and} \quad (1-\delta)K^2 \le |x_{\tau_n + t}(\omega)|^2 \le (1+\delta)K^2. $$
	Thus, there is a constant $c_0 \in (0,1)$ so that, for any $c \in (0,c_0)$, over the same time interval there holds 
	$$ D^p(x_{\tau_n+t}(\omega)) \ge c \delta^p K^{2rp}. $$
	Consequently, 
	\begin{equation}  \label{eq:piece2}
	\begin{split}
	\int_{A_{n,j} \cap A_n}&\int_{\eta_j(K)}^{\tau_{n+1}(\omega) - \tau_n(\omega)} D^p(x_{\tau_n +t}(\omega))dt d\P \\ 
	& \ge c \delta^p K^{2rp} \int_{A_{n,j} \cap A_n \cap \{\bar{\tau} \ge T\}} (\tau_{n+1}(\omega) - \tau_n(\omega))d\P - c\eta_j(K) \delta^p K^{2rp}\P(A_{n,j} \cap A_n).
	\end{split}
	\end{equation} 
	Choosing $c < \min(c_0, c_*^p)$ and then putting \eqref{eq:strongMarkov} and \eqref{eq:piece2} into \eqref{eq:bothpieces} we find
	\begin{align*}
	\int_{A_{n,j} \cap A_n} \int_0^{\tau_{n+1}(\omega) - \tau_n(\omega)} D^p(x_{\tau_n +t}(\omega)) dt d\P \ge c \delta^p K^{2rp} \int_{A_{n,j} \cap A_n \cap \{\bar{\tau} \ge T\}} (\tau_{n+1}(\omega) - \tau_n(\omega))d\P. 
	\end{align*}
	Using this bound in \eqref{eq:doublesum} and noting that if $\bar{\tau}(\omega) \ge T$ and $\tau_n(\omega) \le T/2$ then $x_{\tau_n(\omega)} \in B_K$  gives
	\begin{align*}
	\int_0^T \Pt_t D^p(x) dt &\ge c \delta^p K^{2rp}\sum_{n=0}^\infty\sum_{j=1}^m \int_{A_{n,j} \cap A_n \cap \{\bar{\tau} \ge T\}} (\tau_{n+1}(\omega) - \tau_n(\omega))d\P \\ 
	& = c\delta^p K^{2rp}\sum_{n=0}^\infty \int_{\{\bar{\tau} \ge T\}} \mathbf{1}_{\{\tau_n\le T/2\}} (\tau_{n+1}(\omega) - \tau_n(\omega))d\P. 
	\end{align*} 
	By the telescoping summation and the definition of $\tau_n$, we have  
	$$ \sum_{n=0}^\infty \mathbf{1}_{\{\tau_n\le T/2\}}  (\tau_{n+1}(\omega) - \tau_n(\omega)) \ge T/2  $$
	whenever $\bar{\tau}(\omega) \ge T$, and hence we conclude 
	$$ \frac{1}{T} \int_0^T \Pt_t D^p(x)dt \ge \frac{c}{2} \delta^p K^{2rp} \P(\bar{\tau} \ge T) \ge \frac{c}{4} \delta^p K^{2rp},$$
	which completes the proof. 
\end{proof}

\section{Conservative flow transverse to the kernel} \label{sec:transverse}

Theorem \ref{thm:hypo} is an immediate consequence of Lemma \ref{lem:time-averaged_step1} together with the following proposition. 
The proof consists of two main steps: the first is to deduce growth of the damped modes for a suitable approximate solution (in this case, the deterministic, conservative dynamics) and the second is to justify the approximation on a long enough time-scale to verify Assumption~\ref{ass:time-average} for the true solution. 

\begin{proposition} \label{thm:transverse}
	Let $B$ and $A$ satisfy the conditions of Theorem \ref{thm:hypo}.
	Then, Assumption~\ref{ass:time-average} holds for $r = 1$. 
\end{proposition}
\begin{proof}
	Let $x_0 \in \R^n$ with $K/2 \le |x_0| \le 2K$ and $|\Pi_{\mathrm{ker}A^\perp} x_0| \le \delta |\Pi_{\mathrm{ker}A} x_0|$  for some $K \ge 1$ and $\delta \in (0,1)$. Note that the assumptions on $x_0$ imply that
	\begin{equation} \label{eq:transverselbd1}
	|\Pi_{\mathrm{ker}A}x_0| \ge \frac{K}{\sqrt{8}}.
	\end{equation}
	Let $X_t$ solve 
	\begin{equation} \label{eq:transversedeterm2}
	\begin{cases}
	\frac{d}{dt} X_t = B(X_t,X_t) \\ 
	X_t|_{t=0} = x_0
	\end{cases}
	\end{equation}
	and $\tilde{X}_t$ solve 
	\begin{equation} \label{eq:transversedeterm3}
	\begin{cases}
	\frac{d}{dt} \tilde{X}_t = B(\tilde{X}_t,\tilde{X}_t) \\ 
	\tilde{X}_t|_{t=0} = \Pi_{\mathrm{ker}A}x_0.
	\end{cases}
	\end{equation}
	By taking successive time derivatives of \eqref{eq:transversedeterm3}, we see that
	\begin{align*}
	\frac{d^j}{dt^j} \tilde{X}_t |_{t = 0} 
	\end{align*}
	is a homogeneous $j+1$ degree polynomial in $\Pi_{\mathrm{ker}A}x_0$.
	Therefore, \eqref{eq:transverselbd1} and the condition \eqref{def:transverse} imply that $\exists C_J \ge 1$ and $j \le J$ such that 
	$$
	\left|\Pi_{\mathrm{ker} A^\perp} \frac{d^j}{dt^j} \tilde{X}_t |_{t = 0}\right|  \geq \frac{2}{C_J} K^{j+1}. 
	$$
	It follows that for $\delta$ sufficiently small there holds
	\begin{align} \label{eq:transverselbd2}
	\left|\Pi_{\mathrm{ker} A^\perp} \frac{d^j}{dt^j} X_t |_{t = 0}\right|  \geq \frac{1}{C_J} K^{j+1}. 
	\end{align}
	
	\textbf{Step 1} (growth for the deterministic dynamics):
	The first step quantifies how condition \eqref{def:transverse} implies growth of $\Pi_{\mathrm{ker} A^\perp} X_t$ for short times.  
	\begin{lemma} \label{lem:transversegrowth}
		Let $\delta \in (0,1)$ be small enough so that \eqref{eq:transverselbd2} holds and fix $\tau = K^{-1}$. There is a constant $c_*>0$ depending on $C_J$ and $J$ such that
		\begin{equation} \label{eq:transversedetermgrowth}
		\frac{1}{\tau} \int_0^\tau |\Pi_{\mathrm{ker}A^\perp} X_t| dt \ge c_* K.
		\end{equation}
	\end{lemma}
	\begin{proof}
		We first claim that there exists $\gamma_0 \in (0,1)$ so that if for some $1 \le \ell \le J$ and $\gamma \in (0,\gamma_0)$ there holds
		\begin{equation} \label{eq:transvesegrowth1} \left|\Pi_{\mathrm{ker}A^\perp}\frac{d^\ell}{dt^\ell}X_t|_{t = 0}\right| \ge \frac{\gamma}{C_J}K^{\ell+1},
		\end{equation}
		then there is $t_0 \in \{0,3\gamma^2 K^{-1}\}$ such that
		\begin{equation} \label{eq:transversegrowth2}
		\left|\Pi_{\mathrm{ker}A^\perp}\frac{d^{\ell-1}}{dt^{\ell-1}}X_t|_{t = t_0}\right| \ge \frac{\gamma^3}{C_J}K^\ell.
		\end{equation}
		The claim is trivial if the desired bound holds for $t_0 = 0$. So, suppose otherwise and expand to first order to obtain
		$$ \left|\Pi_{\mathrm{ker}A^\perp}\frac{d^{\ell-1}}{dt^{\ell-1}}X_t\right| \ge \frac{\gamma}{C_J} K^{\ell+1}t - \frac{\gamma^3}{C_J} K^\ell - Ct^2 K^{\ell+2} $$
		for some constant $C\ge 1$ that does not depend on $\ell$ (it depends on $J$ and the size of $B$ on the unit sphere). Thus, for $t_0 = 3\gamma^2 K^{-1}$ we have 
		$$
		\left|\Pi_{\mathrm{ker}A^\perp}\frac{d^{\ell-1}}{dt^{\ell-1}}X_t|_{t=t_0}\right| \ge \frac{2\gamma^3}{C_J} K^\ell - 9C \gamma^4 K^\ell.
		$$
		The bound \eqref{eq:transversegrowth2} then follows for $\gamma_0 < (9C C_J)^{-1}$. 
		
		Using \eqref{eq:transverselbd2} and iterating the claim we just proved, it is straightforward to show that if $$\gamma < \min((9CC_J)^{-1}, 1/\sqrt{2}),$$ then there exists $t_0 \in [0,6\gamma^2 K^{-1}]$ such that 
		$$
		|\Pi_{\mathrm{ker}A^\perp}X_{t_0}| \ge \frac{\gamma^{3^J}}{C_J} K.
		$$
		Taylor expanding to first order at $t = t_0$ then gives 
		$$ |\Pi_{\mathrm{ker}A^\perp} X_{t_0+t}| \ge \frac{\gamma^{3^J}}{C_J} K - t C K^2 $$
		and hence for $\tau_1 = (1/2)\gamma^{3^J}(KCC_J)^{-1}$ we have 
		$$ \frac{1}{\tau_1}\int_0^{\tau_1} |\Pi_{\mathrm{ker}A^\perp}X_{t_0 + t}| dt \ge \frac{\gamma^{3^J}}{2C_J}K. $$
		Supposing that $\gamma$ is small enough so that $\tau_1 + t_0 \le K^{-1}$, it follows that for $\tau = K^{-1}$ there holds 
		$$ \frac{1}{\tau}\int_0^{\tau} |\Pi_{\mathrm{ker}A}^\perp X_{t}| dt \ge \frac{\gamma^{3^J}}{2C_J}K \frac{\tau_1}{\tau} \ge \frac{\gamma^{2\cdot 3^J}}{4 C C_J^2} K, $$
		which completes the proof. 
	\end{proof}
	
	\textbf{Step 2} (approximating with $X_t$):
	In order to make use of Lemma~\ref{lem:transversegrowth} we need to show that $X_t$ is a sufficiently good approximation of $x_t$ for $t \leqc K^{-1}$. To this end, we have the following lemma.
	\begin{lemma} \label{lem:transverseapprox}
		Let $x_0$ and $K$ be as defined at the beginning of proof and set $\tau = K^{-1}$. With $X_t$ given by \eqref{eq:transversedeterm2} and $x_t$ given by \eqref{eq:SDE}, there are $K_* \ge 1$ and $C > 0$ (both independent of $x_0$) so that for $K\ge K_*$ there holds 
		$$ \P\left(\sup_{0 \le t \le \tau}|X_t - x_t| \le C\right)\ge \frac{1}{2}. $$
	\end{lemma}
	\begin{proof}
		The error $X_t-x_t$ solves
		$$  d(X_t-x_t) = L_{X_t}(X_t - x_t)dt + A X_t dt - A(X_t - x_t)dt - B(X_t-x_t,X_t-x_t)dt -\sigma dW_t, $$
		where the operator $L_x$ is as defined in \eqref{eq:Lxdef}. In what follows, denote by $S_X(t,s)$ the two-time linear propagator of the time-inhomogeneous ODE associated to $L_{X_t}$, i.e. 
		\begin{align*}
		\frac{d}{dt} S_X(t,s)v & = L_{X_t} S_X(t,s)v \\ 
		S_X(s,s)v & = v. 
		\end{align*}
		Since $|X_t| \le 2K$ for all $t$ there is a constant $C_1$ that does not depend on $K$ such that for $t > s$, 
		\begin{equation} \label{eq:transversesemibound}
		\|S_{X_\cdot}(t,s)\| \le e^{C_1K(t-s)}.
		\end{equation}
		Therefore, there is $C_2 > 0$ independent of $K$ so that for $T \le \tau = K^{-1}$ we have
		\begin{align}
		\sup_{0 \le t \le T}|X_t - x_t| &\lesssim \int_0^{t} e^{CK(t-s)}(|X_s-x_s|^2 + |X_s| + |X_s - x_s|) ds + \sup_{0 \le t \le \tau}\left|\int_0^t S_{X_\cdot}(t,s)\sigma dW_s\right| \nonumber \\ 
		& \le C_2\left( K^{-1}\sup_{0 \le t \le T}|X_t - x_t|^2 + 1 + \sup_{0 \le t \le T}\left|\int_0^t S_{X_\cdot}(t,s)\sigma dW_s\right|\right).  \label{eq:transverseapprox1}
		\end{align}
		By the Doob martingale inequality, the It\^{o} isometry, and \eqref{eq:transversesemibound}, there is $C_3 > 0$ depending on $\sigma$ but independent of $K$ and so that for all $R \geq 1$ there holds
		\begin{align*} 
		\P\left(\sup_{0 \le t \le \tau} \left|\int_0^t S_{X}(t,s)\sigma dW_s\right| \ge R/\sqrt{K} \right) \leqc \frac{K}{R^2} \int_0^\tau e^{2K(\tau-s)} \|\sigma\|^2 ds \leq \frac{C_3}{R^2}. 
		\end{align*}
		Therefore, there is $R_* \ge 1$ independent of $K$ so that $\P(\Omega_0) \ge 1/2$ if we define $\Omega_0$ as  
		$$ \Omega_0 = \left\{\omega \in \Omega: \sup_{0 \le t \le \tau} \left|\int_0^t S_{X_\cdot}(t,s)\sigma dW_s\right| \le R_*/\sqrt{K}\right\}.$$
		Fix $\omega \in \Omega_0$ and let $T_\omega$ be the maximal time such that $\sup_{0 \le t \le T_\omega} |X_t(\omega) - x_t(\omega)| \le 2C_2$, where $C_2$ is as in \eqref{eq:transverseapprox1}.  Since $x_t$ and $X_t$ take values continuously in time, $T_\omega> 0$. Moreover, by \eqref{eq:transverseapprox1} and the definition of $\Omega_0$, for $T \le \min(T_\omega, \tau)$ we have
		\begin{equation}
		\sup_{0 \le t \le T} |X_t - x_t| \le \frac{4C_2^3}{K} + C_2 + \frac{C_2 R_*}{\sqrt{K}}.
		\end{equation}
		Thus, $T_\omega \ge \tau$ as soon as $K \ge \max(4R_*^2,8C_2^2)$. This completes the proof.
	\end{proof}
	
	With Lemma~\ref{lem:transversegrowth} and Lemma~\ref{lem:transverseapprox} in hand, the proof of Proposition~\ref{thm:transverse} follows quickly.
	Indeed, let $x_0$, $\delta > 0$, and $K \ge 1$ be as defined at the beginning of the proof and $K_*$ be as in Lemma \ref{lem:transverseapprox}. We need to show that there is $c > 0$ so that for all $K \ge K_*$ there holds 
	$$ \frac{1}{\tau} \int_0^\tau \E |\Pi_{\mathrm{ker}A^\perp} x_t|^2 dt \ge cK^2,  $$
	where as before $\tau = K^{-1}$. First, by Chebyshev's inequality we have for all $K \geq 1$ and $\delta' > 0$, 
	\begin{align*}
	\delta' K^2 \PP\left( \frac{1}{\tau}\int_0^\tau \abs{\Pi_{\mathrm{ker} A^\perp} x_t}^2 dt \ge \delta' K^2 \right) \le \frac{1}{\tau} \int_0^\tau \E |\Pi_{\mathrm{ker}A^\perp} x_t|^2 dt. 
	\end{align*}
	Let $\Omega_1$ be the set such that 
	\begin{align*}
	\sup_{0 \le t \le \tau}|X_t - x_t| \le C, 
	\end{align*}
	where $C$ is as in Lemma \ref{lem:transverseapprox}, which implies $\PP(\Omega_1) > 1/2$ (note that while $\Omega_1$ can depend on $K$ and $x_0$, the associated estimates do not). 
	By Lemma \ref{lem:transversegrowth}, for $\omega \in \Omega_1$ we have 
	\begin{align*}
	\int_0^\tau \abs{\Pi_{\mathrm{ker} A^\perp} x_t}^2 dt \geq \frac{c_*^2}{2} K^2 \tau - C \tau.
	\end{align*}
	Therefore, assuming also $K_* \ge 2\sqrt{C}/c_*$, for $\delta' < c_*^2/4$ and $K \geq K_*$
	we have 
	\begin{align*}
	\frac{1}{2}\delta' K^2 \le \frac{1}{\tau} \int_0^\tau \E |\Pi_{\mathrm{ker}A^\perp}x_t|^2 dt, 
	\end{align*}
	which implies Assumption \ref{ass:time-average}. 
\end{proof} 

\section{Unstable equilibria in the kernel} \label{sec:1dkernel}

In this section, we consider the case where $\mathcal{U} = \mathrm{ker}(A) \cap \S^{n-1}$ consists entirely of unstable equilibria of the conservative dynamics (i.e., $B(x,x) = 0$ for every $x \in \mathrm{ker}A$) and prove Theorems~\ref{thm:basic1} and~\ref{thm:basic2}. As in Section \ref{sec:transverse}, the proofs are based on a two-step procedure that consists of first deducing growth of the damped modes for a suitable approximate solution and second justifying the approximation on a long enough time-scale to verify Assumption~\ref{ass:time-average}. In the present setting, for an initial condition $x \in B_K$ (where $B_K$ is as in the statement of Assumption~\ref{ass:time-average}), the approximation of the damped modes that we consider is obtained simply by linearizing $B$ around the equilibrium $\Pi_{\text{ker}A}x$. 

Recall that for $x \in \R^n$ we define $L_x:\R^n \to \R^n$ by $L_x v = B(x,v) + B(v,x)$ and that we denote $L_x^\perp = \Pi_{\mathrm{ker}A^\perp} L_x \Pi_{\mathrm{ker}A^\perp}$. In the setting of Theorems~\ref{thm:basic1} and~\ref{thm:basic2}, for any $x \in \mathcal{U}$, $L_x^\perp$ either has an eigenvalue $\lambda$ with $\text{Re}(\lambda) > 0$ or an unstable Jordan block corresponding to $\text{Re}(\lambda) = 0$. In studying the properties of linearized solutions we must consider separately these two scenarios. In Section~\ref{sec:spectralestimates} we prove the necessary growth and approximation estimates in the spectrally unstable case, and in Section~\ref{sec:jordanestimates} we treat the Jordan block unstable case. In Section~\ref{sec:basicproof} we use the results of Sections~\ref{sec:spectralestimates} and~\ref{sec:jordanestimates} to complete the proofs of Theorems~\ref{thm:basic1} and~\ref{thm:basic2}.

\subsection{Spectrally unstable estimates} \label{sec:spectralestimates}

We begin by considering the case where for $z \in \mathcal{U}$, $L^\perp_{z}$ is spectrally unstable. In this setting, the result concerning growth of the damped modes for the linear approximation is given as follows. Recall from Section~\ref{sec:Intro} that we denote the Jordan normal form of $L^\perp_{z}$ by 
$$ L^\perp_{z} = P_{z}J_{z}^\perp P^{-1}_{z}.$$

\begin{lemma} \label{lem:1dspectralgrowth}
	Let $z \in \mathrm{ker}A$ and $r \in (0,1)$. Assume that the eigenvalue $\lambda = \lambda_R + i \lambda_I$ of $L_{z/|z|}^\perp$ with largest real part is such that $\lambda_R > 0$. Suppose further that there exists a generalized eigenvector $v = v_R + iv_I$ corresponding to eigenvalue $\lambda$ such that $\mathrm{Ran}(\sigma) \cap \{v_R,v_I\}\setminus\{0\} \neq \emptyset$ and, defining $V = \mathrm{span}\{P_{z/|z|}^{-1}v\}$, there holds
	\begin{equation} \label{eq:geigenassumption}
	\Pi_V J_{z/|z|}^\perp x = \lambda \Pi_V x \quad \forall x \in \C^n.
	\end{equation}
	Let $Y_t: [0,\infty) \to \mathrm{ker}A^\perp$ solve 
	\begin{equation}\label{eq:linearized1}
	\begin{cases}
	dY_t =  L_{z}^\perp Y_t dt + \Pi_{\mathrm{ker}A^\perp}\sigma dW_t \\ 
	Y_t|_{t=0} = Y_0 \in \mathrm{ker}A^\perp.
	\end{cases}
	\end{equation}
	For any $\epsilon \in(0,1)$, there is $K_*(\epsilon) \ge 1$ and constants $c_*, \beta > 0$ that do not depend on $\epsilon$, $r$, or $|z|$ so that for $|z|\ge K_*$ and 
	$$\tau = \frac{(1/2 + r+\epsilon)\log(|z|)}{\lambda_R |z|}$$ there holds 
	\begin{equation} \label{eq:1dspectralgrowth}
	\P\left(\frac{1}{\tau}\int_0^{\tau} |Y_t(\omega)|dt \ge c_* |z|^r\right)\ge \beta.
	\end{equation}
\end{lemma}

\begin{remark}
	The condition \eqref{eq:geigenassumption} just says that $v$ is the first generalized eigenvector in a Jordan chain corresponding to eigenvalue $\lambda$.
\end{remark}

\begin{remark}\label{rem:spectralconstants1}
	It follows directly from the proof below that if $\sigma$ is invertible, then for any $C > 0$ the constants $c_*$ and $\beta$ can be chosen uniformly for $\lambda_R \ge C^{-1}$ and $\|P_{z/|z|}\| + \|P^{-1}_{z/|z|}\| \le C$. The main observation here is that when $\sigma$ is invertible the constant $c_1$ in the proof below depends only on $\|\sigma^{-1}\|$ and $\|P_{z/|z|}\|$.
\end{remark}

\begin{proof}
	We will consider the case where $\lambda_I \neq 0$, as the situation where $\lambda_I = 0$ follows from the same argument. Without loss of generality, we may suppose that $v_R \in \mathrm{Ran}(\sigma)$. Let $\tilde{Y}_t = P_{z/|z|}^{-1} Y_t \in \C^n$. For simplicity of notation we write $J^\perp$ for $J_{z/|z|}^\perp$ and $P$ for $P_{z/|z|}$. Then, $\tilde{Y}_t$ solves
	\begin{equation}
	\begin{cases}
	d\tilde{Y}_t = |z| J^\perp \tilde{Y}_t dt + P^{-1} \Pi_{\mathrm{ker}A^\perp}\sigma dW_s \\ 
	\tilde{Y}_0 = P^{-1} Y_0 \in \C^n.
	\end{cases}
	\end{equation}
	Since
	\begin{equation} \label{eq:coordinateequiv}
	|\tilde{Y}_t| = |P^{-1}Y_t| \le \|P^{-1}\|_{\C^n \to \C^n} |Y_t|,
	\end{equation}
	it suffices to prove \eqref{eq:1dspectralgrowth} with $Y_t(\omega)$ replaced by $\tilde{Y}_t(\omega)$. Define the subspace $V = \mathrm{span}\{P^{-1}v\}$ of $\C^n$ and let $t_* = (\lambda_R |z|)^{-1}$. The plan is to first show that for every $R > 0$ there exists $\beta_1 > 0$ so that 
	\begin{equation} \label{eq:diffusive}
	\P\left(|\Pi_V \tilde{Y}_{t_*}| \ge \frac{R}{\sqrt{\lambda_R |z|}}\right)\ge \beta_1.
	\end{equation}
	We will then prove that there is $R \ge 1$ and $c_* > 0$ so that 
	\begin{equation} \label{eq:transport}
	\P\left( \frac{1}{\tau} \int_0^{\tau - t_*} |\tilde{Y}_t| dt \ge c_*|z|^r \text{ }\bigg|\text{ } |\Pi_V \tilde{Y}_0| \ge \frac{R}{\sqrt{\lambda_R |z|}} \right)\ge \frac{1}{2}.
	\end{equation} 
	Together, \eqref{eq:diffusive} and \eqref{eq:transport} yield the bound \eqref{eq:1dspectralgrowth} for $\tilde{Y}_t$.

	We now prove \eqref{eq:diffusive}. The formula for $\tilde{Y}_t$ reads
	\begin{equation}\label{eq:Ytbar}
	\tilde{Y}_t = e^{|z|J^\perp t}\tilde{Y}_0 + \int_0^t e^{|z|J^\perp(t-s)}P^{-1} \Pi_{\mathrm{ker}A^\perp}\sigma dW_s.
	\end{equation} 
	By the It\^{o} isometry, the variance of $\Pi_V \tilde{Y}_{t_*}$ is given by 
	\begin{equation} \label{eq:Ytilvar1}
	\text{Var}(\Pi_V \tilde{Y}_{t_*}) = \int_0^{t_*} \left\| \Pi_V e^{|z|J^\perp(t_*-s)}P^{-1}\Pi_{\mathrm{ker}A^\perp}\sigma \right\|_F^2 ds,
	\end{equation}
	where $\|\cdot\|_{F}$ denotes the Frobenius norm on $\C^{n\times n}$. Observe now that for any $t \ge 0$ there holds
	$$  \Pi_V e^{|z| J^\perp t} P^{-1}v_R = \frac{1}{2}e^{\lambda |z| t}P^{-1}v, $$
	which gives 
	\begin{equation} \label{eq:PiVsize1}
	\left| \Pi_V e^{|z| J^\perp t} P^{-1}v_R\right|^2 \ge \frac{e^{2|z|\lambda_R t}}{4}.
	\end{equation}
	Since $v_R \in \mathrm{Ran}(\sigma)$, by \eqref{eq:PiVsize1} we have
	\begin{equation} \label{eq:CntoCnbound}
	\left\| \Pi_V e^{|z|J^\perp(t_*-s)}P^{-1}\Pi_{\mathrm{ker}A^\perp}\sigma \right\|_{\C^n \to \C^n}^2 \ge c_1 e^{2|z| \lambda_R (t_*-s)} 
	\end{equation}
	for some $c_1 > 0$ depending only on $\sigma$ and $v_R$. Thus, from \eqref{eq:Ytilvar1} and the equivalence of norms in finite dimensions there holds 
	\begin{align} \label{eq:Ytilvar2}
	\text{Var}(\Pi_V \tilde{Y}_{t_*}) \gtrsim \int_0^{t_*} \left\| \Pi_V e^{|z|J^\perp(t_*-s)}P^{-1}\Pi_{\mathrm{ker}A^\perp}\sigma \right\|_{\C^n \to \C^n}^2 ds & \gtrsim c_1\int_0^{t_*} e^{2|z| \lambda_R (t_*-s)} ds \gtrsim \frac{c_1}{\lambda_R |z|}.
	\end{align}
	The claim \eqref{eq:diffusive} then follows from \eqref{eq:Ytilvar2} and the fact that the real and imaginary parts of $\Pi_V \tilde{Y}_{t_*}$ are both Gaussian.
	
	We now turn to \eqref{eq:transport}. First, note that 
	\begin{align*}
	\left|\Pi_V e^{|z| J^\perp t}\tilde{Y}_0\right|^2 & = \left|e^{|z| J^\perp t}\Pi_{V}\tilde{Y}_0\right|^2 = e^{2|z|\lambda_R t}\left|\Pi_{V}\tilde{Y}_0\right|^2.
	\end{align*} 
	Therefore, 
	$|\Pi_V \tilde{Y}_0| \ge R/\sqrt{\lambda_R |z|}$ implies that 
	\begin{equation}\label{eq:transport1}
	\frac{1}{\tau}\int_0^{\tau-t_*} |\Pi_V e^{|z| J^\perp t}\tilde{Y}_0| dt \ge \frac{R}{\tau \sqrt{\lambda_R|z|}}\int_0^{\tau-t_*} e^{|z| \lambda_R t}dt \ge \frac{R|z|^{r+ \epsilon}}{6e\log(|z|)\sqrt{\lambda_R}},
	\end{equation}
	where in the second inequality we have assumed that $K_* \ge 2e$. Taking $K_*(\epsilon)$ even larger to ensure that $|z|^\epsilon \ge \log(|z|)$, it follows from \eqref{eq:transport1} and
	\begin{align*} 
	\frac{1}{\tau}\int_0^{\tau - t_*}|\tilde{Y}_t| dt 
	& \ge \frac{1}{\tau} \int_0^{\tau - t_*} |\Pi_V e^{|z| J^\perp t}\tilde{Y}_0| dt - \frac{1}{\tau}\int_0^{\tau}\left|\int_0^t \Pi_V e^{|z|J^\perp(t-s)}P^{-1}\Pi_{\mathrm{ker}A^\perp} \sigma dW_s \right| dt, 
	\end{align*}
	that to complete the proof of \eqref{eq:transport} it suffices to show that
	\begin{equation} \label{eq:transport2}
	\P\left(\frac{1}{\tau}\int_0^{\tau}\left|\int_0^t \Pi_V e^{|z|J^\perp(t-s)}P^{-1}\Pi_{\mathrm{ker}A^\perp} \sigma dW_s \right| dt \le \frac{R |z|^{r+\epsilon}}{12 e \log(|z|)\sqrt{\lambda_R}} \right) \ge \frac{1}{2}
	\end{equation}
	for some $R\ge 1$. By the It\^{o} isometry, we have 
	\begin{align*} \frac{1}{\tau}\E \int_0^\tau \left|\int_0^t  \Pi_V e^{|z|J^\perp(t-s)}P^{-1}\Pi_{\mathrm{ker}A^\perp} \sigma dW_s \right| dt & \le \frac{1}{\tau}\int_0^\tau \left(\int_0^t \|\Pi_V e^{|z| J^\perp (t-s)}P^{-1}\Pi_{\mathrm{ker}A^\perp}\sigma\|_F^2 ds\right)^{1/2} dt \\ 
	& \leqc \frac{1}{\tau}\int_0^\tau \left(\int_0^t e^{2|z|\lambda_R(t-s)} \|P^{-1}\|^2 \|\sigma\|^2 ds\right)^{1/2} dt \\ 
	& \leqc \frac{\|P^{-1}\| \|\sigma\| |z|^{r + \epsilon}}{\sqrt{\lambda_R}\log(|z|)}.
	\end{align*}
	Then, \eqref{eq:transport2} follows by taking $R$ sufficiently large and using Chebyshev's inequality, completing the proof.
\end{proof}

We now use Lemma~\ref{lem:1dspectralgrowth} to prove the time-averaged growth estimate \eqref{eq:time-averaged_main} required by Assumption~\ref{ass:time-average} when the initial condition $x \in \R^n$ is such that $\Pi_{\mathrm{ker}A}x$ is a spectrally unstable equilibrium point for $B$. In what follows, for $x \in \R^n$ we write $z = \Pi_{\mathrm{ker}A}x$ and $y = x-z = \Pi_{\mathrm{ker}A^\perp}x$.

\begin{lemma} \label{lem:1dspectralapprox}
	Suppose that $B(x,x) = 0$ for every $x \in \mathrm{ker}A$ and let $x_0 \in \R^n$ be such that $L_{z_0/|z_0|}^\perp$ has  maximally unstable eigenvalue $\lambda = \lambda_R + i\lambda_I$ with $\lambda_R > 0$. Suppose further that there exists a generalized eigenvector $v = v_R + iv_I$ satisfying the conditions of Lemma~\ref{lem:1dspectralgrowth}. Fix $r \in (0,1/4)$ and for $K \ge 1$ set 
	$$ \eta(K) = 10 \left(\frac{(1/2 + r)\log(K)}{\lambda_R K}\right).$$
	There exist $K_* \ge 1$, $c_* > 0$, and a universal constant $\delta_* \in (0,1/4]$ so that if
	$$|y_0| \le \delta |z_0|^{r} \text{ and } (1-\delta) K \le |x_0| \le (1+\delta)K $$
	for $\delta \in (0,\delta_*]$ and $K \ge K_*$, then there holds
	\begin{equation} \label{eq:1dspectralgoal}
	\frac{1}{\eta(K)} \int_0^{\eta(K)} \E |y_t|^2 dt \ge c_* K^{2r}.
	\end{equation}
	Moreover, if $\sigma$ is invertible, $r_0 \in (0,1/4)$ is fixed, and $C_0 \ge 1$ is such that $\lambda_R \ge C_0^{-1}$ and 
	\begin{equation} \label{eq:spectralupb}
	\|P^{-1}_{z_0/|z_0|}\| + \|P_{z_0/|z_0|}\| \le C_0,
	\end{equation}
	then the constants $c_*$ and $K_*$ can be chosen to depend only on $C_0$ and $r_0$ for $r \le r_0$.
\end{lemma}

\begin{proof}
	We first assume only that $\{v_R, v_I\}\setminus \{0\} \cap \mathrm{Ran}(\sigma) \neq \emptyset$. For $\epsilon \in (0,1)$ to be chosen, let
	\begin{equation} \label{eq:1dspectraltau}
	\tau = \frac{(1/2 + r + \epsilon)\log(|z_0|)}{\lambda_R |z_0|}
	\end{equation}
	and suppose that 
	\begin{equation} \label{eq:contradiction1}
	\int_0^\tau \E |y_t|^2 dt \le \tau \delta_1 K^{2r}
	\end{equation}
	for some $\delta_1 \in (0,1)$. We will obtain a contradiction for $\delta_1$ sufficiently small.
	
	The first step is to use the contradiction hypothesis \eqref{eq:contradiction1} to obtain bounds on $|z_t - z_0|$. Since $B(z_t,z_t) = 0$ by assumption, we have
	\begin{equation}\label{eq:frozenz2.1}
	dz_t = (\Pi_{\mathrm{ker}A}(B(y_t,z_t) + B(z_t,y_t) + B(y_t,y_t)-Ay_t)dt + \Pi_{\mathrm{ker}A}\sigma dW_t. 
	\end{equation}
	Using \eqref{eq:contradiction1}, the Cauchy-Schwarz inequality, $\E |x_t| \leqc K$ for $t \leqc 1$ (this follows from \eqref{eq:exptail2}), and Doob's  martingale inequality we obtain
	\begin{equation} \label{eq:frozenz}
	\E \sup_{0 \le t \le \tau}|z_t-z_0| \leqc \sqrt{\delta_1} K^{1+r}\tau + \sqrt{\tau} \leqc \max(\sqrt{\delta_1}, K_*^{-1/2})K^{1+r}\tau.
	\end{equation}
	Define 
	$$
	\Omega_0 = \left\{\omega \in \Omega: \int_0^\tau |y_t|^2 dt \le \sqrt{\delta_1} \tau K^{2r}, \sup_{0 \le t \le \tau}|z_t-z_0| \le K^{1+r}\tau \right\}
	$$ 
	and let $\beta, c_* > 0$ be as in Lemma~\ref{lem:1dspectralgrowth} applied with $z = z_0$ and the chosen $r \in (0,1/4)$. Recall here that $\beta$ and $c_*$ do not depend on $r$ or $\epsilon$. By \eqref{eq:contradiction1} and \eqref{eq:frozenz}, for $\delta_1$ sufficiently small and $K_*$ sufficiently large depending only on $\beta$ there holds \begin{equation} \label{eq:Omega_0}
	\P(\Omega_0) > 1-\beta/2.
	\end{equation}
	
	Let $Y_t$ solve 
	$$
	\begin{cases}
	dY_t =  L_{z_0}^\perp Y_t dt + \Pi_{\mathrm{ker}A^\perp}\sigma dW_t \\ 
	Y_t|_{t=0} = \Pi_{\mathrm{ker}A^\perp}x_0.
	\end{cases}
	$$
	We will show that the exact solution $y_t$ is well approximated by the linearized dynamics $Y_t$ on the set $\Omega_0$. The difference $Y_t - y_t$ solves 
	\begin{align*}
	\frac{d}{dt}(Y_t - y_t) &= \Pi_{\mathrm{ker}A^\perp}(B(Y_t,z_0) + B(z_0,Y_t) - B(y_t,z_t) - B(z_t,y_t) - B(y_t,y_t) + Ay_t) \\ 
	& = L_{z_0}^{\perp}(Y_t-y_t) + \Pi_{\mathrm{ker}A^\perp}(B(y_t,z_0-z_t) + B(z_0-z_t,y_t) - B(y_t,y_t) + Ay_t).
	\end{align*}
	Therefore,
	\begin{equation} \label{eq:1dspectralerror}
	Y_t - y_t = \int_0^t e^{L_{z_0}^\perp(t-s)}\Pi_{\mathrm{ker}A^\perp}(B(y_s,z_0-z_s) + B(z_0-z_s,y_s) - B(y_s,y_s) + Ay_s)ds.
	\end{equation}
	Now, from the Jordan canonical form, for $t \le \tau$ and 
	$$C_1 = \|P_{z_0/|z_0|}\|\|P^{-1}_{z_0/|z_0|}\|$$
	there holds 
	$$\|e^{L_{z_0}^\perp t}\| \leqc  C_1(1+(|z_0|t)^n)e^{\lambda_R |z_0|t} \leqc C_1(1+\lambda_R^{-n}) |\log(K)|^n e^{\lambda_R |z_0|t}. $$
	Thus, by applying Young's convolution inequality in \eqref{eq:1dspectralerror}, for $\omega_0 \in \Omega_0$ we have the estimate
	\begin{align} 
	\int_0^{\tau}|Y_t(\omega_0) - y_t(\omega_0)| dt &\leqc C_1(1+\lambda_R^{-n}) |\log(K)|^n \left(\int_0^\tau e^{\lambda_R |z_0| t}dt\right) \nonumber \\ 
	& \quad \times \left(\int_0^\tau (|y_t(\omega_0)|^2 + |y_t(\omega_0)|+ |y_t(\omega_0)||z_0-z_t(\omega_0)|)dt\right) \nonumber \\ 
	& \leqc C_1(1+\lambda_R^{-n-1}) |\log(K)|^n K^{r-1/2 + \epsilon}\left(\delta_1^{1/4}\tau K^{2r} + \delta_1^{1/4}\tau K^r + \delta_1^{1/4} \tau^2 K^{1+2r}\right) \nonumber  \\ 
	& \leqc C_1(1+\lambda_R^{-n-2})|\log(K)|^{n+1}K^{\epsilon}K^{2r - 1/2} (\delta_1^{1/4}K^{r} \tau), \label{eq:errorbound}
	\end{align} 
	where in the last inequality above we have assumed that $K$ is large enough so that $|z_0| \ge K/2$ (and consequently $K \tau \leqc \lambda_R^{-1} \log(K) $). Assuming $r \le r_0 < 1/4$, we may take $\epsilon = 1/4-r_0 > 0$ 
	to obtain
	\begin{equation} \label{eq:approx1}
	\int_0^{\tau}|Y_t(\omega_0) - y_t(\omega_0)| dt \le C_1 C_2 \delta_1^{1/4}\tau K^r
	\end{equation}
	for some constant $C_2 > 0$ satisfying 
	\begin{equation} \label{eq:C2}
	C_2 \leqc \left(1+\lambda_R^{-n-2}\right)\sup_{K \ge 1} \{K^{-\epsilon}\log(K)^{n+1}\}.
	\end{equation}
	
	With \eqref{eq:approx1} established we are now ready to use Lemma~\ref{lem:1dspectralgrowth} to complete the proof.
	Applying Lemma~\ref{lem:1dspectralgrowth} and using again $|z_0| \ge K/2$, we obtain that for $K$ sufficiently large depending only on $\epsilon$ there holds
	\begin{equation}
	\P\left(\int_0^\tau |Y_t(\omega)|dt \ge \frac{c_*}{2} \tau K^r\right) \ge	\P\left(\int_0^{\tau} |Y_t(\omega)| dt \ge c_* \tau |z_0|^r\right) \ge \beta,
	\end{equation}
	where the constants $c_*$ and $\beta$ are as defined after \eqref{eq:frozenz}. From $\P(\Omega_0) \ge 1-\beta/2$ we thus have 
	\begin{equation} \label{eq:beta/2}
	\P\left(\Omega_0 \cap \left\{\int_0^{\tau} |Y_t(\omega)| dt \ge \frac{c_*}{2}\tau K^r\right\}\right)\ge \frac{\beta}{2}.
	\end{equation}
	By \eqref{eq:beta/2}, the reverse triangle inequality, Cauchy-Schwarz, and \eqref{eq:approx1} we deduce that for $C_1 C_2 \delta_1^{1/4} \le c_*/4$ there holds 
	\begin{equation} \label{eq:spectralcontradiction1}
	\E \int_0^\tau |y_t|^2 dt \ge \frac{\beta c_*^2}{32} \tau K^{2r}.
	\end{equation}
	Taking $\delta_1$ even smaller to ensure $\delta_1 < \frac{\beta c_*^2}{32}$ gives the desired contradiction with \eqref{eq:contradiction1}.
	
	In the calculations above, $\delta_1$ is chosen small depending on $C_1$, $C_2$, $\beta$, and $c_*$, while $K$ is chosen sufficiently large depending only on $\beta$ and $\epsilon = 1/4-r_0$. We conclude that there is a constant $c_*'(C_1,C_2,\beta, c_*) > 0$ and $K_*(\beta, r_0)$ so that for $K \ge K_*$ there holds
	\begin{equation} \label{eq:c_*'}
	\E \int_0^\tau |y_t|^2 dt \ge c_*' \tau K^{2r}.
	\end{equation}
	To obtain \eqref{eq:1dspectralgoal} from \eqref{eq:c_*'}, observe that $|y_0| \le \delta |z_0|^{r}$ and $(1-\delta)K \le |x_0| \le (1+\delta)K$ imply that $(1-2\delta)K \le |z_0| \le (1+\delta)K$ for $\delta$ small enough. Therefore, taking $\delta_*$ sufficiently small and $K_*$ perhaps larger yields
	$$ \frac{1}{40} \eta(K) \le \tau \le  \eta(K), $$
	which when combined with \eqref{eq:c_*'} gives
	\begin{equation} \label{eq:c_*''}
	\E \int_0^{\eta(K)} |y_t|^2 dt \ge \frac{c_*'}{40} \eta(K) K^{2r}: = c_*'' \eta(K) K^{2r}.
	\end{equation} 
	
	It remains only to argue that if $\sigma$ is invertible and $C_0 \ge 1$ is such that $\lambda_R \ge C_0^{-1}$ and \eqref{eq:spectralupb} holds, then $K_*$ and the constant $c_*''$ in \eqref{eq:c_*''} can be taken to depend only on $C_0$ and $r_0$. Since in the proof of \eqref{eq:c_*'} we took $K_* = K_*(\beta,r_0)$ and $c_*' = c_*'(C_1,C_2,\beta,c_*)$, it suffices to show that $\beta$, $C_1$, $C_2$, and $c_*$ can be taken to depend only on $C_0$ and $r_0$. By Lemma~\ref{lem:1dspectralgrowth} and Remark~\ref{rem:spectralconstants1}, both $\beta$ and $c_*$ depend only on $C_0$ when $\sigma$ is invertible. Regarding $C_1$ and $C_2$, following the proof above we see that \eqref{eq:spectralupb} and $\lambda_R \ge C_0^{-1}$ imply 
	$C_1 \le C_0^2$
	and 
	$$ C_2 \leqc (1 + C_0^{n+2})\sup_{K \ge 1} \{K^{1/4 - r_0}\log(K)^{n+1}\}. $$
	This completes the proof.
\end{proof}

\subsection{Jordan block unstable estimates} \label{sec:jordanestimates}

In this section, we consider the case where for each $z \in \mathrm{ker}A$ the eigenvalues of $L_{z/|z|}^\perp$ all have non-positive real part, but there exists an unstable Jordan block of size greater than or equal to two corresponding to an eigenvalue $\lambda$ with $\text{Re}(\lambda) = 0$. In other words, there exists $1 \le J \le n-2$ such that
\begin{equation}\label{eq:jordangrowth}
t^J \leqc \|e^{J_{z/|z|}^\perp t}\|_{\R^n \to \R^n} \leqc (1+t^J)
\end{equation}
for all $t \ge 0$. Note that when \eqref{eq:jordangrowth} holds there necessarily exists a generalized eigenvector $v = v_R + iv_I$ of $L_{z/|z|}^\perp$ corresponding to eigenvalue $\lambda$ such that, defining $V = \mathrm{span}\{P_{z/|z|}^{-1}v\}$, there holds both
\begin{equation}\label{eq:jordaneigass}
\Pi_V J_{z/|z|}^\perp x = \lambda J_{z/|z|}^\perp \Pi_V x \quad \forall x \in \C^n
\end{equation}
and
\begin{equation}\label{eq:jordangrowthv}
\left|e^{J_{z/|z|}^\perp t}P^{-1}_{z/|z|}v\right| \gtrsim t.
\end{equation}

In this setting, the analogue of Lemma~\ref{lem:1dspectralgrowth} is stated as follows.

\begin{lemma}\label{lem:1dJordangrowth}
	Let $z \in \mathrm{ker}A$ and assume that $L_{z/|z|}^\perp$ has an unstable Jordan block in the sense that \eqref{eq:jordangrowth} holds. Suppose that there exists a generalized eigenvector $v = v_R + iv_I$ satisfying  \eqref{eq:jordaneigass} and \eqref{eq:jordangrowthv} above as well as $\{v_R,v_I\}\setminus \{0\} \cap \mathrm{Ran}(\sigma) \neq \emptyset$. Let $\tilde{v}$ be the generalized eigenvector such that $(L_{z/|z|}^\perp - \lambda)v = \tilde{v}$ and define $\tilde{V} = \mathrm{span}\{P_{z/|z|}^{-1}\tilde{v}\}$. For $p \in (0,2/3)$, set $r=1-3p/2 > 0$ and
	$$\tau(|z|) = |z|^{-p}.$$
	There are constants $c_*, \beta > 0$ that do not depend on $|z|$ so that the solution to \eqref{eq:linearized1} satisfies
	\begin{equation} \label{eq:1djordangrowth}
	\P\left(\frac{1}{\tau(|z|)} \int_0^{\tau(|z|)} |\Pi_{\tilde{V}} P^{-1}_{z/|z|}Y_t(\omega)| dt \ge c_* |z|^{r}\right) \ge \beta. 
	\end{equation}
\end{lemma}

\begin{remark}\label{rem:Jordan}
	The assumptions above imply that $v$ is the first vector in a Jordan chain of length greater than or equal to two corresponding to eigenvalue $\lambda$. Thus, $\tilde{v}$ is a generalized eigenvector in the same chain and for any $x \in \C^n$ there holds 
	\begin{equation} \label{eq:jordantgrowth}
	\Pi_{\tilde{V}} e^{J_{z/|z|}^\perp t} x = e^{\lambda t}(\Pi_{\tilde{V}}x + t\Pi_{V} x).
	\end{equation}
\end{remark}

\begin{remark}\label{rem:jordanconstants1}
	Similar to Lemma~\ref{lem:1dspectralgrowth}, if $\sigma$ is invertible and $\|P_{z/|z|}\| + \|P_{z/|z|}^{-1}\| \le C$, then $c_*$ and $\beta$ can be chosen depending only on $C$ and $\sigma$. This follows directly from the proof below after noting that when $\sigma$ is invertible, the constant $c_1$ in \eqref{eq:YtilGaussian} satisfes $c_1 \gtrsim (\|\sigma^{-1}\| \|P_{z/|z|}\|)^{-2}$.
\end{remark}

\begin{proof} 
	As in the proof of Lemma~\ref{lem:1dspectralgrowth}, we assume that $v_R \in \mathrm{Ran}(\sigma)$ and $\lambda \neq 0$ (so that $\lambda$ is pure imaginary); the case $\lambda = 0$ is a straightforward variation. Denote $P_{z/|z|}$ and $J_{z/|z|}^\perp$ by $P$ and $J^\perp$, respectively. Let $\tilde{Y}_t = P^{-1}Y_t \in \C^n$, which is given by the formula
	\begin{equation}
	\label{eq:jordanYtil}
	\tilde{Y}_t = e^{|z| J^\perp t}\tilde{Y}_0 + \int_0^t e^{|z| J^\perp (t-s)}P^{-1}\Pi_{\mathrm{ker}A^\perp}\sigma dW_s.
	\end{equation}
	We will follow the same general strategy as in the proof of Lemma~\ref{lem:1dspectralgrowth}. 
	
	We first show that for every $R > 0$ there exists $\beta_1 > 0$ so that 
	\begin{equation}\label{eq:1djordanstep1}
	\P\left( |\Pi_{V} \tilde{Y}_{\tau/2}| \ge R \sqrt{\tau} \right) \ge \beta_1.
	\end{equation}
	Since $v_R \in \mathrm{Ran}(\sigma)$ and 
	\begin{equation}
	|\Pi_V e^{|z| J^\perp (t-s)} P^{-1}v_R| \gtrsim 1,
	\end{equation}
	we have
	\begin{equation}\label{eq:YtilGaussian}
	\begin{aligned}
	\text{Var}(\Pi_V \tilde{Y}_{\tau/2}) &= \int_0^{\tau/2} \left\| \Pi_V e^{|z| J^\perp (t-s)}P^{-1}\Pi_{\mathrm{ker}A^\perp} \sigma \right\|_F^2 ds  
	\ge c_1 \tau
	\end{aligned}
	\end{equation}
	for a constant $c_1$ depending $v_R$ and $\sigma$. The bound \eqref{eq:1djordanstep1} now follows from the fact that $\Pi_V \tilde{Y}_{\tau/2}$ is Gaussian.
	
	Next, as in Lemma~\ref{lem:1dspectralgrowth}, to complete the proof it suffices to show that there is $c_* > 0$ and $R$ sufficiently large so that
	\begin{equation} \label{eq:jordantransport}
	\P\left( \frac{1}{\tau} \int_0^{\tau/2}|\Pi_{\tilde{V}} \tilde{Y}_t|dt \ge c_* |z|^r \bigg | |\Pi_V \tilde{Y}_0| \ge R\sqrt{\tau} \right) \ge \frac{1}{2}.
	\end{equation}
	First, note the elementary fact that for any $a,b \in \R$ and $T > 0$ there holds 
	\begin{equation} \label{eq:noncancellation}
	\int_0^{T}|a + bt| dt \gtrsim |b|T^2 
	\end{equation}
	with the implicit constant independent of $a$, $b$ or $T$. One can see this easily by dividing the integral into $t \leq \min(-\frac{a}{b},T)$ and $t \geq \min(-\frac{a}{b},T)$. By \eqref{eq:noncancellation}, \eqref{eq:jordantgrowth}, and $|e^{\lambda t}| = 1$ we have 
	\begin{equation} \label{eq:jordanstep2.1}
	\frac{1}{\tau}\int_0^{\tau/2}|\Pi_{\tilde{V}} e^{|z| J^\perp t}\tilde{Y}_0|dt = \frac{1}{\tau}\int_0^{\tau/2} |\Pi_{\tilde{V}} \tilde{Y}_0 + t|z| |\Pi_V \tilde{Y}_0| dt \gtrsim R |z|\tau^{3/2}.
	\end{equation}
	Moreover, by the It\^o isometry,  
	\begin{equation} \label{eq:jordanstep2.2}
	\begin{aligned}
	\frac{1}{\tau} \E \int_0^{\tau} \left|\int_0^t \Pi_{\tilde{V}} e^{|z| J^\perp (t-s)}P^{-1} \Pi_{\mathrm{ker}A^\perp} \sigma dW_s\right|dt &\leqc \frac{\|P^{-1}\| \|\sigma\|}{\tau} \int_0^\tau \left(\int_0^t (1+|z|s)^2ds\right)^{1/2}dt \\ 
	& \leqc \|P^{-1}\| \|\sigma\| |z| \tau^{3/2}.
	\end{aligned} 
	\end{equation}
	Using the reverse triangle inequality and Chebyshev's inequality as in the proof of \eqref{eq:transport}, the estimates \eqref{eq:jordanstep2.1} and \eqref{eq:jordanstep2.2} together yield, for $R \gg \|P^{-1}\| \|\sigma\|$,
	$$
	\P\left( \frac{1}{\tau} \int_0^{\tau/2}|\Pi_{\tilde{V}} \tilde{Y}_t|dt \gtrsim R |z| \tau^{3/2} \bigg | |\Pi_V \tilde{Y}_0| \ge R\sqrt{\tau} \right) \ge \frac{1}{2}.
	$$ 
	Since $r(p)$ is such that $|z| \tau^{3/2} = |z|^r$ we obtain \eqref{eq:jordantransport}, completing the proof.
\end{proof}

We now turn to the analogue of Lemma~\ref{lem:1dspectralapprox} in the Jordan block unstable case. The idea is the same as in the spectrally unstable case. However, due to the slower timescale of the instability (i.e., $p < 1$ in Lemma~\ref{lem:1dJordanapprox}) we need to make use of the cancellation 
\begin{equation} \label{eq:cancellation}
\Pi_{\mathrm{ker}A} (B(y,z) + B(z,y)) = 0 \quad \forall z \in \mathrm{ker}A, y \in \mathrm{ker}A^\perp
\end{equation}
assumed in Theorem~\ref{thm:basic2}. We have not assumed \eqref{eq:cancellation} in Theorem~\ref{thm:basic1} since, as we will show in Lemma~\ref{lem:Bproperties}, the cancellation condition is automatically satisfied in the case that $\mathrm{dim}(\mathrm{ker} A) =1$ due to $B(x,x)\cdot x = 0$. 

\begin{lemma} \label{lem:1dJordanapprox}
	Suppose that $B(x,x) = 0$ for every $x \in \mathrm{ker}A$ and that the cancellation condition \eqref{eq:cancellation} is satisfied. Let $x_0 \in \R^n$ be such that $L_{z_0/|z_0|}^\perp$ is Jordan block unstable in the sense that \eqref{eq:jordangrowth} holds and suppose that there exists a generalized eigenvector $v$ satisfying the conditions in Lemma~\ref{lem:1dJordangrowth}. Fix any $r \in (0,1/7)$ and for $K \ge 1$ set 
	\begin{equation} \label{eq:jordantaudef}
	\eta(K) = 4 K^{\frac{2r-2}{3}}.
	\end{equation}
	There exists $c_* > 0$ and a universal constant $\delta_* \in (0,1)$ so that if
	$$|y_0| \le \delta |z_0|^{r} \text{ and } K/2 \le |x_0| \le 2K $$
	for $\delta \in (0,\delta_*)$ and $K \ge 1$, then 
	$$ \frac{1}{\eta(K)} \int_0^{\eta(K)} \E |y_t|^2 dt \ge c_* K^{2r}.$$
	Moreover, if $\sigma$ is invertible and 
	\begin{equation} \label{eq:jordanuniformass2}
	\|P_{z_0/|z_0|}\| + \|P^{-1}_{z_0/|z_0|}\| \le C_0
	\end{equation}
	for some $C_0 \ge 1$, then $c_*$ can be chosen depending only on $\sigma$ and $C_0$.
\end{lemma}

\begin{proof}
	We will consider the case where $\sigma$ is invertible and \eqref{eq:jordanuniformass2} holds. The proof when one only assumes that $\{v_R,v_I\}\setminus \{0\} \cap \mathrm{Ran}(\sigma) \neq \emptyset$ follows from exactly the same argument. Let $K \ge 1$ be such that $K/2 \le |x_0| \le 2K$ and suppose for the sake of contradiction that 
	\begin{equation} \label{eq:jordancontr}
	\E \int_0^{\eta} |y_t|^2 dt \le \delta_1 \eta K^{2r}
	\end{equation}
	for $\delta_1 \in (0,1)$ and $\eta = \eta(K)$ given by \eqref{eq:jordantaudef}. As in the proof of Lemma~\ref{lem:1dspectralapprox}, we will obtain a contradiction for $\delta_1$ sufficiently small.
	
	By the cancellation condition \eqref{eq:cancellation}, the equation for $z_t$ is given by 
	\begin{equation} \label{eq:frozenz2.2}
	dz_t = \Pi_{\mathrm{ker}A}(B(y_t,y_t)-Ay_t) dt + \Pi_{\mathrm{ker}A}\sigma dW_t. 
	\end{equation}
	It follows that
	\begin{equation} \label{eq:frozenz2}
	\E \sup_{0 \le t \le \eta}|z_t - z_0| \leqc \eta \sqrt{\delta_1} K^{2r} + \sqrt{\eta} \leqc K^{\frac{r-1}{3}},
	\end{equation}
	where we have used that the choices of $\eta$ and $r$ are such that 
	$$\eta K^{2r} \leqc \sqrt{\eta} \leqc K^{\frac{r-1}{3}}.$$ Let $c_*, \beta > 0$ be as in Lemma~\ref{lem:1dJordangrowth} applied with $z = z_0$ and $p = (2-2r)/3$. Recall from Remark~\ref{rem:jordanconstants1} that $c_*$ and $\beta > 0$ depend only on $C_0$ and $\sigma$. By \eqref{eq:jordancontr} and \eqref{eq:frozenz2}, for $\delta_1$ sufficiently small and $R$ sufficiently large, both depending on $\beta$, we have 
	\begin{equation}
	\P\left(\Omega_0 = \left\{\omega \in \Omega: \int_0^\eta |y_t|^2 dt \le \sqrt{\delta_1}\eta K^{2r}, \sup_{0 \le t \le \eta}|z_t-z_0| \le RK^{\frac{r-1}{3}}\right\} \right) \ge 1-\frac{\beta}{2}.
	\end{equation}
	Now, as in the proof of Lemma~\ref{lem:1dspectralapprox}, let $Y_t$ solve 
	$$
	\begin{cases}
	dY_t =  L_{z_0}^\perp Y_t dt + \Pi_{\mathrm{ker}A^\perp}\sigma dW_t \\ 
	Y_t|_{t=0} = \Pi_{\mathrm{ker}A^\perp}x_0.
	\end{cases}
	$$
	Let $\tilde{v}$ and $\tilde{V}$ be as in Lemma~\ref{lem:1dJordangrowth}. By the Jordan canonical form, we have
	$$\|\Pi_{\tilde{V}}P^{-1}_{z_0/|z_0|} e^{L_{z_0}^\perp t}\| \le C_0 (1+|z_0|t) \le C_0(1+2Kt).$$
	Using this in \eqref{eq:1dspectralerror}
	we obtain, for $\omega_0 \in \Omega_0$, 
	\begin{align*}
	\int_0^\eta |\Pi_{\tilde{V}}P_{z_0/|z_0|}^{-1}(Y_t(\omega_0) - y_t(\omega_0))|dt &\leqc C_0\left(\int_0^\eta (1+Kt)dt\right) \left(\delta_1^{1/4} \eta K^{2r} + R\delta_1^{1/4}\eta K^{\frac{r-1}{3}}K^r\right) \\ 
	& \leqc C_0\delta_1^{1/4}\eta K^r (K\eta)\left(\eta K^r + R\eta K^{\frac{r-1}{3}}\right) \\ 
	& \leqc C_0 R\delta_1^{1/4}\eta K^r (K\eta) \eta K^r,
	\end{align*}
	where in the last line we noted that trivially $K^{\frac{r-1}{3}} \le K^r$. Observe now that the restriction $r < 1/7$ and the formula for $\eta$ imply that 
	$$ (K\eta)\eta K^r\le 1,$$
	and thus we have 
	\begin{equation} \label{eq:approx2}
	\int_0^\eta |\Pi_{\tilde{V}}P_{z_0/|z_0|}^{-1}(Y_t(\omega_0) - y_t(\omega_0))| dt \leqc C_0 R \delta_1^{1/4} \eta K^r.
	\end{equation}
	
	We now use \eqref{eq:approx2} and Lemma~\ref{lem:1dJordangrowth} to complete the proof. Suppose that $\delta$ is small enough so that $K/4 \le |z_0| \le 2K$. Then, the choice of $\eta(K)$ ensures that 
	$$|z_0|^{\frac{2r-2}{3}}\le  \eta(K) \le 8 |z_0|^{\frac{2r-2}{3}}. $$
	Thus, from Lemma~\ref{lem:1dJordangrowth} we have 
	\begin{equation} \label{eq:1dJordanapprox1}
	\P\left( \frac{1}{\eta(K)}\int_0^{\eta(K)}|\Pi_{\tilde{V}}P^{-1}_{z_0/|z_0|}Y_t| dt \ge \frac{c_*}{32}K^r \right) \ge \beta.
	\end{equation}
	It follows from \eqref{eq:approx2}, \eqref{eq:1dJordanapprox1}, and $\P(\Omega_0) \ge 1-\beta/2$ that for $\delta_1$ sufficiently small depending only on $c_*$, $\beta$, and $C_0$ there holds 
	$$
	\E \int_0^\eta |\Pi_{\tilde{V}}P^{-1}_{z_0/|z_0|} y_t|^2 dt \ge \left( \frac{c_*}{64}\right)^2 \frac{\beta}{2} \eta K^{2r}.
	$$
	Therefore,
	$$ \E \int_0^\eta |y_t|^2 dt \ge \frac{1}{C_0^2} \left( \frac{c_*}{64}\right)^2 \frac{\beta}{2} \eta K^{2r}. $$
	We obtain a contradiction by taking $\delta_1$ perhaps even smaller to guarantee 
	$$ \delta_1 \le \frac{1}{C_0^2} \left(\frac{c_*}{64}\right)^2 \frac{\beta}{4} .$$
	Since $c_*$ and $\beta$ depend only on $C_0$ and $\sigma$ we obtain 
	$$ \E \int_0^\eta |y_t|^2 dt \ge c\eta K^{2r}$$
	for a constant $c$ depending only on $C_0$ and $\sigma$, which completes the proof.
\end{proof}

\begin{remark} \label{rem:1dcancellation}
	If the cancellation condition \eqref{eq:cancellation} is assumed in the spectrally unstable case, so that \eqref{eq:frozenz2.1} can be replaced with \eqref{eq:frozenz2.2}, one can show that any $r \in (0,1)$ is permissible in   Lemma~\ref{lem:1dspectralapprox}.
\end{remark}

\subsection{Concluding the proofs of Theorems~\ref{thm:basic1} and~\ref{thm:basic2}} \label{sec:basicproof}

In this section we use the results from Sections~\ref{sec:spectralestimates} and~\ref{sec:jordanestimates} to prove Theorems~\ref{thm:basic1} and~\ref{thm:basic2}.

\subsubsection{$\text{dim}(\text{ker}A) = 1$}\label{sec:basic1dproof}

In this section we prove Theorem~\ref{thm:basic1}. It is a special case of the result below, stated for more general assumptions on $\sigma$. Recall that we denote $\mathcal{U} = \mathrm{ker}A \cap \S^{n-1}$ and for $x \in \mathcal{U}$ write 
$$ L_x^\perp = P_x J_x^\perp P_x^{-1} $$
for the Jordan normal form of $L_x^\perp$.
\begin{theorem} \label{thm:basic1sigma}
	Suppose that $\mathcal{U} = \set{x_0,-x_0}$ for some unit vector $x_0$  and that for each $x \in \mathcal{U}$ there holds
	\begin{align} \label{eq:1dassumption}
	B(x,x) = 0, \quad \lim_{t \to \infty} \norm{e^{tL^\perp_{x}}} = \infty.
	\end{align}
	Moreover, let $\sigma$ satisfy the following conditions (which hold trivially when $\mathrm{rank}(\sigma) = n$).
	\begin{itemize}
		\item If $x \in \mathcal{U}$ is such that $L_x^\perp$ has an eigenvalue with positive real part, then there is a generalized eigenvector $v = v_R + iv_I$ associated with the eigenvalue $\lambda$ of $L_x^\perp$ with maximal real part such that $\{v_R, v_I\}\setminus \{0\} \cap \mathrm{Ran}(\sigma) \neq \emptyset$ and 
		\begin{equation}\label{eq:endchain} \Pi_{\mathrm{span}\{P_x^{-1}v\}} J_x^\perp y = \lambda J_x \Pi_{\mathrm{span}\{P_x^{-1}v\}}y \quad \forall y\in \C^n. 
		\end{equation}
		\item If $x \in \mathcal{U}$ is such that $t^J \leqc \|e^{L_x^\perp t}\| \leqc 1 + t^J$ for some $1 \le J \le n-2$, then there is a generalized eigenvector $v = v_R + iv_I$ associated with an eigenvalue $\lambda$ of $L_x^\perp$ with $\mathrm{Re}(\lambda) = 0$ that satisfies $\{v_R, v_I\}\setminus \{0\} \cap \mathrm{Ran}(\sigma) \neq 0$, \eqref{eq:endchain}, and $|e^{J_x^\perp t} P_x^{-1} v| \gtrsim t$.
	\end{itemize}
	Then, there exists at least one stationary measure $\mu$ and $\brak{x}^p \in L^1(\dee \mu)$ for all $ p < 1/3$. 
\end{theorem}

We begin by showing that the cancellation condition \eqref{eq:cancellation} is automatically satisfied in one dimension due to $B(x,x)\cdot x = 0$.

\begin{lemma} \label{lem:Bproperties}
	Let $\Pi: \R^n \to \R^n$ be a projection onto a one-dimensional subspace of $\R^n$. Suppose $B:\R^n \times \R^n \to \R^n$ is a bilinear function satisfying $B(x,x) \cdot x = 0$ and 
	$$ B(\Pi x, \Pi x) = 0 \quad \forall x\in \R^n. $$
	Then, 
	$$\Pi  B(\Pi x,\Pi^\perp x) +  \Pi B(\Pi^\perp x, \Pi x) = 0$$
	for every $x \in \R^n$.
\end{lemma}

\begin{proof}
	The property $B(x,x) \cdot x = 0$ remains true after any orthogonal coordinate transform, and so without loss of generality we may assume that $\Pi$ is the projection onto the subspace $\{(x_1,0,\ldots,0):x_1 \in \R\}$. In this setting, we need to show that 
	$$B_1(\Pi x,\Pi^\perp x) + B_1(\Pi^\perp x,\Pi x) = 0$$
	for any $x \in \R^n$. The condition $B(\Pi x, \Pi x) = 0$ implies that $\partial_{x_1}^2 B(x,x) = 0$. Hence, differentiating $B(x,x) \cdot x = 0$ twice with respect to $x_1$ gives
	$$
	\partial_{x_1}[B_1(x,x)] = 0.
	$$
	Substituting $x = \Pi x + \Pi^\perp x$ we find
	$$\partial_{x_1}(B_1(\Pi x,\Pi^\perp x) + B_1(\Pi^\perp x,\Pi x)) = 0.$$
	Noting that $B_1(\Pi x,\Pi^\perp x) + B_1(\Pi^\perp x,\Pi x) = 0$ when $x_1 = 0$ completes the proof.
\end{proof}

We are now ready to prove Theorem~\ref{thm:basic1sigma}: 
\begin{proof}[Proof of Theorem~\ref{thm:basic1sigma}]
	Let $\mathcal{U} = \{x_0,-x_0\}$ for $x_0 \in \mathrm{ker}A \cap \S^{n-1}.$ Fix $r < 1/7$ and for $\delta \in (0,1)$ to be chosen let
	$$B^\delta_K = \{x \in \R^n: |\Pi_{\mathrm{ker}A^\perp}x|^2 \le \delta |\Pi_{\mathrm{ker}A}x|^{2r} \text{ and }(1-\delta)K^2 \le |x|^2 \le (1+\delta)K^2\}. $$
	For any $x \in \R^n\setminus\{0\}$ there is $c > 0$ and $j \in \{1,2\}$ such that $\Pi_{\mathrm{ker}A}x = c(-1)^jx_0$. Therefore, defining for $j \in \{1,2\}$ the sets
	$$B^\delta_{K,j} = \{x \in B^\delta_K: \Pi_{\mathrm{ker}A}x = c(-1)^jx_0 \text{ for some }c > 0\},$$
	we have $B^\delta_K = B^\delta_{K,1} \cup B^\delta_{K,2}$. The assumptions in \eqref{eq:1dassumption} imply that $x_0$ and $-x_0$ are both equilibria of $B$ with $L_{x_0}^\perp$ and $L_{-x_0}^\perp$ spectrally or Jordan block unstable. Observe now that $r < 1/7$ is always permitted in Lemmas~\ref{lem:1dspectralapprox} and \ref{lem:1dJordanapprox}. Moreover, the associated
	$\eta$ always satisfies 
	$$ \eta(K) \leqc K^{-4/7}.$$
	Therefore, by Lemmas~\ref{lem:1dspectralapprox} and \ref{lem:1dJordanapprox} (note that we may apply Lemma~\ref{lem:1dJordanapprox} in the present one-dimensional setting due to Lemma~\ref{lem:Bproperties}) there are constants $\delta > 0$, $c_* > 0$, and $K_* \ge 1$ along with functions $\eta_j(K)$ satisfying $\lim_{K \to \infty}\sup_{j = 1,2}\eta_j(K) = 0$ such that for $K \ge K_*$ there holds 
	$$ x_0 \in B_{K,j} \implies \frac{1}{\eta_j(K)}\int_0^{\eta_j(K)} |\Pi_{\mathrm{ker}A^\perp}x_t|^2 dt \ge c_* K^{2r}.$$
	Thus, Assumption~\ref{ass:time-average} is satisfied for any $r < 1/7$. Theorem~\ref{thm:basic1sigma} then follows from Lemma~\ref{lem:time-averaged_main}.
\end{proof}

\begin{remark} \label{rem:moments1}
	Let $\mu$ be the stationary measure constructed in Theorem~\ref{thm:basic1sigma}. If $L_{x_0}^\perp$ and $L_{-x_0}^\perp$ are both spectrally unstable, then by Remark~\ref{rem:1dcancellation} and Lemma~\ref{lem:Bproperties} it holds that $\brak{x}^p \in L^1(\dee \mu)$ for every $p > 0$.
\end{remark}

\subsubsection{$\text{dim}(\text{ker}A) > 1$} \label{sec:spectralgeneral}

\begin{proof}[Proof of Theorem~\ref{thm:basic2}]
	
	We will give the details only for the spectrally unstable case, i.e., the case where there exists $C_0 \ge 1$ so that for every $z \in \mathcal{U}$ there is a maximally unstable eigenvalue $\lambda(z)$ of $L_z^\perp$ satisfying $\text{Re}(\lambda(z)) > 0$. Fix $r < 1/4$ and for $\delta \in (0,1/4)$, $K\ge 1$, and $z \in \mathcal{U}$, define the sets
	$$ B^\delta_{K} = \left\{x \in \R^n: |\Pi_{\mathrm{ker}A^\perp} x|^2 \le \delta |\Pi_{\mathrm{ker}A}x|^{2r}, (1-\delta)K^2 \le |x|^2 \le (1+\delta)K^2\right\}$$
	and 
	$$B^\delta_{K,z} = B_{K}^\delta \cap \left\{x \in \R^n: \Pi_{\mathrm{ker}A}x/|\Pi_{\mathrm{ker}A}x| = z \in \mathcal{U} \right\}.$$
	Since $\mathcal{U}$ is compact and the eigenvalues of a matrix vary continuously with respect its entries, we have 
	$$0 < \lambda_-:= \min_{z \in \mathcal{U}} \text{Re}(\lambda(z)) \le \max_{z \in \mathcal{U}} \text{Re}(\lambda(z)):=\lambda_+ < \infty.$$
	Therefore, by Lemma~\ref{lem:1dspectralapprox} and \eqref{eq:JNFAssumption} there exist $\delta_* \in (0,1/4)$, $K_* \ge 1$, and $c_* > 0$ so that for every $z \in \mathcal{U}$ and $K \ge K_*$, defining 
	$$
	\eta(K,z) = 10\left( \frac{(1/2+r)\log(K)}{\text{Re}(\lambda(z))K} \right),
	$$
	there holds 
	$$
	x_0 \in B_{K,z}^{\delta_*} \implies \frac{1}{\eta(K,z)}\E \int_0^{\eta(K,z)}|\Pi_{\mathrm{ker}A^\perp} x_t|^2 dt \ge c_* K^{2r}.
	$$
	We used here the statement at the end of Lemma~\ref{lem:1dspectralapprox} about the dependence of the constants when $\sigma$ is invertible and \eqref{eq:jordanuniformass2} holds. Since 
	$$ 10\left( \frac{(1/2+r)\log(K)}{\lambda_+ K} \right) \le \eta(K,z) \le 10\left( \frac{(1/2+r)\log(K)}{\lambda_-K} \right) $$
	for every $z \in \mathcal{U}$ it follows that for 
	$$ \eta_*(K):= 10\left( \frac{(1/2+r)\log(K)}{\lambda_-K} \right) $$
	and $K \ge K_*$ there holds 
	$$ x_0 \in B_{K}^{\delta_*} \implies \frac{1}{\eta_*(K)} \E \int_0^{\eta_*(K)} |\Pi_{\mathrm{ker}A^\perp} x_t|^2 dt \ge \frac{\lambda_-}{\lambda_+}c_* K^{2r}. $$
	Thus, Assumption~\ref{ass:time-average} is satisfied for any $r < 1/4$, which due to Lemma~\ref{lem:time-averaged_step1} completes the proof.
\end{proof}

\section{Sabra and Galerkin Navier-Stokes}\label{sec:combination}
\subsection{Statement and proof of general result}

In this section we state and prove the general theorem that will be used to obtain Theorem~\ref{thm:NS} announced earlier. 

\begin{theorem} \label{thm:combination}
	Let $\mathrm{rank}(\sigma) = n$ and suppose that $\mathrm{ker}A = V_1 \oplus V_2$ for orthogonal subspaces of $\R^n$ satisfying the following properties.
	\begin{itemize}
		\item For any $x \in V_1 \cup V_2$, $B(x,x) = 0$, i.e., $V_1$ and $V_2$ consist of deterministic equilibria.
		\item There is $C > 0$ so that 
		\begin{equation} \label{eq:JCFassumption2}
		\max_{j=1,2} \sup_{x \in V_j \cap \S^{n-1}}(\|P_{x,j}\| + \|P_{x,j}^{-1}\|) \le C,
		\end{equation}
		where $P_{x,j}J^\perp_j P^{-1}_{x,j}$ denotes the Jordan canonical form of $\Pi_{V_j^\perp} L_x \Pi_{V_j^\perp}$, with $L_x$ as defined in \eqref{eq:Lxdef}.
		\item There is $\lambda_{\mathrm{min}} > 0$ such that for any $j \in \{1,2\}$ and $x \in V_j \cap \S^{n-1}$ there is an eigenvalue $\lambda$ of  $\Pi_{V_j^\perp} L_x \Pi_{V_j^\perp}$ with $\mathrm{Re}(\lambda) \ge \lambda_{\mathrm{min}}$.
		\item There exists $c > 0$ so that for any $v_1 \in V_1$ and $v_2 \in V_2$ there holds
		\begin{equation}  \label{eq:combcancellation}
		\Pi_{\mathrm{ker}A} B(v_1,v_2) + \Pi_{\mathrm{ker}A} B(v_2,v_1)  = 0 
		\end{equation}
		and 
		\begin{equation}\label{eq:combolbd}
		|\Pi_{\mathrm{ker}A^\perp} B(v_1,v_2)+\Pi_{\mathrm{ker}A^\perp} B(v_2,v_1)| \ge c |v_1| |v_2|.
		\end{equation}
	\end{itemize}
	Then, there exists at least one stationary measure $\mu$ and $\brak{x}^p \in L^1(\dee \mu)$ for every $p < 2/3$.
\end{theorem}	

The proof of Theorem~\ref{thm:combination} will proceed roughly as follows. As before we will verify Assumption~\ref{ass:time-average}. For initial conditions $x_0$ near $\mathrm{ker}A$ with $\min(|\Pi_{V_1} x_0|,|\Pi_{V_2} x_0|)$ sufficiently large, we use \eqref{eq:combolbd} and arguments similar to those in Section~\ref{sec:transverse} to obtain growth of the damped modes. If instead $x_0$ is concentrated in one of the $V_j$, we proceed similarly to Section~\ref{sec:1dkernel} and use the spectral instability to deduce growth into $V_j^\perp$. This either causes the damped modes to grow directly or the solution to enter a region where $\min(|\Pi_{V_1} x_t|,|\Pi_{V_2} x_t|)$ is large enough to subsequently apply \eqref{eq:combolbd} as in the first case. The cancellation \eqref{eq:combcancellation} is used throughout to justify certain approximations.

We begin with a lemma that describes growth of the damped modes for initial conditions with $|\Pi_{V_1} x_0|$ and $|\Pi_{V_2} x_0|$ both sufficiently large.  

\begin{lemma} \label{lem:combotransverse}
	Fix $r \in (0,1]$ and $\delta_0 \in (0,1)$. There are $c_*(\delta_0) > 0$ and $K_*(\delta_0) \ge 1$ so that for any $x_0 \in \R^n$ satisfying 
	$$ K/2 \le |x_0| \le 2K, \quad |\Pi_{\mathrm{ker}A^\perp} x_0|^2 \le \delta K^{2r}, \quad \min\left( |\Pi_{V_1} x_0|, |\Pi_{V_2} x_0| \right) \ge \delta_0^{1/8} K^r$$
	for $0 \le \delta \le \epsilon \delta_0^{3/4}$ and $K \ge K_*$, where $\epsilon$ is a sufficiently small constant independent of $\delta_0$, there holds
	$$ \frac{1}{K^{-1}} \E \int_0^{K^{-1}} |\Pi_{\mathrm{ker}A^\perp} x_t|^2 dt \ge c_* K^{2r}. $$
\end{lemma}

\begin{proof}
	Let $X_t$ solve 
	\begin{equation}
	\begin{cases}
	\frac{d}{dt}X_t = B(X_t,X_t) \\ 
	X_0 = x_0
	\end{cases}
	\end{equation}
	and define $\eta(K) = K^{-1}$. We claim that there is $c_*(\delta_0) > 0$ so that for all $K \ge 1$, $\delta \ll \delta_0^{3/4}$, and $x_0$ as in the statement of the lemma there holds
	\begin{equation}\label{eq:combotransversegoal}
	\frac{1}{\eta(K)}\int_0^{\eta(K)} |\Pi_{\mathrm{ker}A^\perp}X_t|^2 dt \ge c_* K^{2r}.
	\end{equation}
	From here the lemma follows by taking $K_*$ large enough so that $c_* K_*^{2r} \gg 1$ and applying Lemma~\ref{lem:transverseapprox}. We now prove \eqref{eq:combotransversegoal}. For $\gamma \in (0,1)$ to be chosen sufficiently small and $\tilde{\eta}(K) = \gamma \delta_0^{1/4} K^{-1}$, suppose for the sake of contradiction that 
	\begin{equation} \label{eq:combocontradiction2}
	\frac{1}{\tilde{\eta}}\int_0^{\tilde{\eta}}|\Pi_{\text{ker}A^\perp}X_t|^2 dt \le \delta_1 K^{2r}
	\end{equation}
	for $\delta_1 \in (0,1)$. By performing a Taylor expansion and using \eqref{eq:combolbd}, $|\Pi_{\mathrm{ker}A^\perp}X_0| \le \sqrt{\delta}K^r$, and $|X_t| \le 2K$ we obtain, for $t \le \tilde{\eta}$,
	\begin{align}
	|\Pi_{\mathrm{ker}A^\perp} X_t| &= \left| \Pi_{\mathrm{ker}A^\perp} X_0 + t \Pi_{\mathrm{ker}A^\perp}B(X_0,X_0)+ \Pi_{\mathrm{ker}A^\perp}\int_0^t (t-s) \frac{d}{ds}B(X_s,X_s)ds\right| \nonumber \\
	& \ge  c t|\Pi_{V_1}X_0||\Pi_{V_2}X_0| - C\sqrt{\delta} K^{r} - C K \tilde{\eta} \int_0^t |B(X_s,X_s)| ds, \label{eq:integralerror}
	\end{align}
	where in the second inequality we used that 
	$$  \frac{d}{ds}B(X_s,X_s) = B(B(X_s,X_s),X_s) + B(X_s,B(X_s,X_s)). $$
	The goal is now to bound the integral in \eqref{eq:integralerror}. First, by writing 
	$$ X_s = \Pi_{\mathrm{ker}A}X_s + \Pi_{\mathrm{ker}A^\perp}X_s = \Pi_{V_1}X_s + \Pi_{V_2}X_s + \Pi_{\mathrm{ker}A^\perp}X_s $$
	and using the triangle inequality we deduce
	\begin{equation} \label{eq:integralerror1}
	|B(X_s,X_s)| \leqc K|\Pi_{\mathrm{ker}A}X_s - \Pi_{\mathrm{ker}A}X_0| + K|\Pi_{\mathrm{ker}A^\perp}X_s| + |\Pi_{V_1}X_0||\Pi_{V_2}X_0|.
	\end{equation}
	Now, by \eqref{eq:combocontradiction2} and \eqref{eq:combcancellation}, for all $t \le \tilde{\eta}$ there holds
	\begin{equation}
	\begin{aligned}
	|\Pi_{\mathrm{ker}A}X_t - \Pi_{\mathrm{ker}A}X_0| &\le \int_0^t |B(\Pi_{\mathrm{ker}A}X_s,\Pi_{\mathrm{ker}A^\perp} X_s)|ds + \int_0^t |B(\Pi_{\mathrm{ker}A^\perp}X_s,X_s)|ds \\ 
	& \leqc K\int_0^t |\Pi_{\mathrm{ker}A^\perp}X_s|ds  \leqc \gamma \delta_0^{1/4}K^r.
	\end{aligned}
	\end{equation}
	Putting this bound into \eqref{eq:integralerror1} and using \eqref{eq:combocontradiction2} again gives, for $t \le \tilde{\eta}$, 
	\begin{equation} \label{eq:integralerror2}
	\int_0^t |B(X_s,X_s)|ds \leqc \gamma \delta_0^{1/4} K^{r} + \tilde{\eta}|\Pi_{V_1}X_0||\Pi_{V_2}X_0|.
	\end{equation}
	Inserting \eqref{eq:integralerror2} in \eqref{eq:integralerror} and integrating the resulting bound over $[0,\tilde{\eta}]$ yields, for $\gamma$ sufficiently small (and new constants $c$ and $C$ which may change from line to line),
	\begin{align*}
	\frac{1}{\tilde{\eta}} \int_0^{\tilde{\eta}}|\Pi_{\mathrm{ker}A^\perp}X_t|dt &\ge c \tilde{\eta} |\Pi_{V_1}X_0| |\Pi_{V_2}X_0| - C\sqrt{\delta} K^r - C \gamma^2 \sqrt{\delta_0} K^r - C \gamma \tilde{\eta} |\Pi_{V_1}X_0| |\Pi_{V_2}X_0|  \\ 
	&  \ge \tilde{\eta} |\Pi_{V_1}X_0| |\Pi_{V_2}X_0| (c - C\gamma) - C\sqrt{\delta}K^r - C\gamma^2 \sqrt{\delta_0} K^r \\ 
	& \ge c \gamma \delta_0^{3/8}K^r - C\sqrt{\delta}K^r.
	\end{align*}
	In the last inequality we have used the fact that $|\Pi_{V_j} X_0| \gtrsim K$ for some $j \in \{1,2\}$ when $\delta \ll 1$. With $\gamma$ now fixed we may take $\delta \ll \gamma^2 \delta_0^{3/4}$ to obtain, for some new constant $c \in (0,1)$, 
	$$ \frac{1}{\tilde{\eta}}\int_0^{\tilde{\eta}}|\Pi_{\mathrm{ker}A^\perp} X_t|^2 dt \ge c \gamma^2 \delta_0^{3/4}K^{2r}.$$
	We obtain a contradiction with \eqref{eq:combocontradiction2} for $\delta_1 \le (c/2) \gamma^2 \delta_0^{3/4}$, and so we conclude 
	$$ \frac{1}{\eta} \int_0^\eta |\Pi_{\mathrm{ker}A^\perp} X_t|^2 dt \ge \frac{\tilde{\eta}}{\eta}  \frac{c}{2} \gamma^2 \delta_0^{3/4} K^{2r} = \frac{c}{2} \gamma^3 \delta_0 K^{2r}, $$ 
	which implies \eqref{eq:combotransversegoal} and completes the proof.
\end{proof}

\begin{proof}[Proof of Theorem~\ref{thm:combination}]
	Fix any $r \in (0,1/4)$ and for $\delta \in(0,1/100)$ and $K\ge 2$ define
	$$B_K^\delta = \{x \in \R^n:|\Pi_{\mathrm{ker}A^\perp}x|^2 \le \delta |\Pi_{\mathrm{ker}A}x|^{2r}, (1-\delta)K^2 \le |x|^2 \le (1+\delta)K^2 \}. $$
	Let
	$$\ell(j)=\begin{cases}
	1 & j=2 \\ 
	2 & j=1.
	\end{cases}
	$$
	We split the set $B_K^\delta$ as 
	$$ B_K^\delta = B_{K,1}^\delta \cup B_{K,2}^\delta \cup B_{K,3}^\delta, $$
	where 
	$$B_{K,j}^\delta = \{x \in B_K^\delta: |\Pi_{V_{\ell(j)}}x| \le \delta^{1/8}K^r \} \quad \text{ if }j\in\{1,2\} $$
	and 
	$$B_{K,3}^\delta = \{x \in B_K^\delta: \min\left(|\Pi_{V_1}x|, |\Pi_{V_2}x|\right)> \delta^{1/8}K^r\}. $$
	By Lemma~\ref{lem:time-averaged_step1}, to complete the proof it suffices to show that there are $K_* \ge 1$, $\delta>0$, $c_* > 0$, and times $\{\eta_j\}_{j=1}^3$ with $\lim_{K \to \infty}\sup_{j}\eta_j(K) = 0$ so that for $K \ge K_*$ there holds 
	\begin{equation} \label{eq:combogoal}
	x_0 \in B_{K,j}^{\delta} \implies \frac{1}{\eta_j(K)}\int_0^{\eta_j(K)} |\Pi_{\mathrm{ker}A^\perp}x_t|^2 dt \ge c_*K^{2r}. 
	\end{equation}
	Due to Lemma~\ref{lem:combotransverse}, for all $\delta$ sufficiently small there is $c_*(\delta) > 0$ and $K_*(\delta) \ge 1$ so that  \eqref{eq:combogoal} is satisfied for $j = 3$ by taking  $\eta_3(K) = K^{-1}$. Thus we must only consider the case where $j \in \{1,2\}$.
	
	Let $x_0 \in B_{K,j}^\delta$ for $j \in \{1,2\}$ and fix any $\bar{r}$ with $r < \bar{r} < 1/4$. Suppose that the maximally unstable eigenvalue of $\Pi_{V_j^\perp} L_{\Pi_{V_j}x_0/|\Pi_{V_j} x_0|} \Pi_{V_j^\perp}$ has real part $\lambda \ge \lambda_{\mathrm{min}} > 0$ and define, for $\epsilon \in (0,1)$ to be chosen,
	$$ \tau_1 = \frac{(1/2+\bar{r}+\epsilon)\log(|\Pi_{V_j} x_0|)}{\lambda |\Pi_{V_j} x_0|}$$
	and $\tau = \tau_1 + K^{-1}$. Let the approximate solution $Y_t:[0,\infty) \to V_j^\perp$ solve
	\begin{equation}
	\begin{cases}
	dY_t = (\Pi_{V_j^\perp} L_{\Pi_{V_j}x_0} \Pi_{V_j^\perp}) Y_tdt + \Pi_{V_j^\perp}\sigma dW_t \\ 
	Y_0 = \Pi_{V_j^\perp}x_0.
	\end{cases}
	\end{equation}
	By Lemma~\ref{lem:1dspectralgrowth} (with $\Pi_{V_j}x_0$ and $V_j$ playing the roles of $z$ and $\mathrm{ker}A$, respectively), there are $K_j(\epsilon) \ge 1$ and $c_j, \beta_j > 0$ that do not depend on $\epsilon$, $\bar{r}$, or $x_0$ so that for $|\Pi_{V_j}x_0| \ge K_j$ there holds 
	\begin{equation}\label{eq:modifiedgrowthlemma}
	\P\left(\frac{1}{\tau_1} \int_0^{\tau_1} |Y_t(\omega)|dt \ge c_j |\Pi_{V_j}x_0|^{\bar{r}}\right) \ge \beta_j. 
	\end{equation}
	Towards a contradiction, suppose that 
	\begin{equation} \label{eq:combocontr3}
	\E\int_0^\tau |\Pi_{\mathrm{ker}A^\perp} x_t|^2 dt \le \delta \tau K^{2r}
	\end{equation}
	for $K \ge K_* \ge \delta^{-1}$. Note that since $\tau/\tau_1 \leqc 1$ this implies
	\begin{equation}  \label{eq:combocontr4}
	\E \int_0^{\tau_1} |\Pi_{\mathrm{ker}A^\perp} x_t|^2 dt \leqc \delta \tau_1 K^{2r}. 
	\end{equation}
	The condition \eqref{eq:combcancellation} and the fact that $V_1 \cup V_2$ consists of deterministic equilibria imply that the equation for $z_t = \Pi_{\mathrm{ker}A}x_t$ is exactly \eqref{eq:frozenz2.1}. Thus, using \eqref{eq:combocontr4}, the proof of \eqref{eq:frozenz} applies and gives 
	\begin{equation} \label{eq:frozenz3}
	\E \sup_{0 \le t \le \tau_1}|z_t - z_0| \leqc \sqrt{\delta}K^{1+r}\tau_1 + \sqrt{\tau_1} \leqc \sqrt{\delta} \log(K) K^r,
	\end{equation}
	where in the second inequality we used the assumption that $K \ge \delta^{-1}$. Define 
	$$ \Omega_0 = \left\{ \omega \in \Omega: \int_0^{\tau_1}|\Pi_{V_j^\perp} x_t|^2 dt \le \delta^{1/8} \tau_1 (\log(K))^2 K^{2r}, \sup_{0 \le t \le \tau_1} |z_t - z_0| \le \delta^{1/4} K^r \log(K) \right\}.$$
	Since $|\Pi_{V_{\ell(j)}}x_0| \le \delta^{1/8}K^r$ by the definition of $B_{K,j}^\delta$, it follows from \eqref{eq:combocontr4} and \eqref{eq:frozenz3} that for $\delta$ sufficiently small depending on $\beta_j$ there holds
	\begin{equation}\label{eq:combocontradiction}
	\P(\Omega_0) \ge 1 - \frac{\beta_j}{2}. 
	\end{equation}	
	Obtaining estimates on $Y_t(\omega) - \Pi_{V_j^\perp} x_t(\omega)$ for $\omega \in \Omega_0$ as in proof of Lemma~\ref{lem:1dspectralapprox} (we make the choice $\epsilon = 1/4 - \bar{r}$) and then using \eqref{eq:modifiedgrowthlemma}, we deduce that  there is $c_0 \in (0,1)$ depending only on $c_j$ so that for $\delta$ sufficiently small and $K_*$ sufficiently large there holds
	\begin{equation} \label{eq:comboVjgrowth}
	\P\left( \frac{1}{\tau_1}\int_0^{\tau_1} |\Pi_{V_j^\perp} x_t(\omega)| dt \ge c_0 K^{\bar{r}} \right) \ge \frac{\beta_j}{2}.
	\end{equation}
	It follows that 
	\begin{equation} \label{eq:combocases}
	\P\left(\frac{1}{\tau_1} \int_0^{\tau_1} |\Pi_{\mathrm{ker}A^\perp}x_t(\omega)|dt \ge \frac{c_0}{2} K^{\bar{r}}\right) \ge \frac{\beta_j}{4} \quad \text{or} \quad \P\left(\frac{1}{\tau_1} \int_0^{\tau_1} |\Pi_{V_{\ell(j)}}x_t(\omega)|dt \ge \frac{c_0}{2} K^{\bar{r}}\right) \ge \frac{\beta_j}{4}.  
	\end{equation}
	In the first case, we immediately obtain a contradiction to \eqref{eq:combocontr4} for $\delta$ sufficiently small depending on $c_0$ and $\beta_j$. In the second case, define the stopping time 
	$$ \bar{\tau}(\omega) = \inf\left\{t\ge 0: \min(|\Pi_{V_1} x_t|, |\Pi_{V_2} x_t|) \ge \frac{c_0}{2} K^{\bar{r}}, K/2 \le |x_t| \le 2K, |\Pi_{\mathrm{ker}A^\perp} x_t|^2 \le \delta^{1/8} K^{2\bar{r}} \right\}. $$ 
	Now, using $|\Pi_{V_{\ell(j)}} x_0| \le \delta^{1/8} K^r$ and $|\Pi_{\mathrm{ker}A^\perp} x_0| \le \sqrt{\delta} K^{r}$ we can show 
	\begin{align*}
	\sup_{0 \le t \le \tau_1} |\Pi_{\mathrm{ker}A^\perp} x_t| &\leqc \sqrt{\delta}K^r + K^{1+r}\delta^{1/8}\tau_1 + \tau_1 K \sup_{0 \le t \le \tau_1}|z_t - z_0| + \sup_{0 \le t \le \tau_1}|W_t|\\ 
	& \quad + \tau_1 \left(\sup_{0 \le t \le \tau_1} |z_t - z_0|\right)^2 + \int_0^{\tau_1} |\Pi_{\mathrm{ker}A^\perp} x_t|^2 dt  \\ 
	& \quad + \sqrt{\tau_1}\left(\sup_{0 \le t \le \tau_1} |z_t - z_0| + K\right)\left( \int_0^{\tau_1} |\Pi_{\mathrm{ker}A^\perp}x_t|^2 dt\right)^{1/2}.
	\end{align*}
	It follows then from \eqref{eq:combocontr4}, \eqref{eq:frozenz3}, and $K \ge \delta^{-1}$ that for $\delta$ sufficiently small there holds 
	\begin{equation} \label{eq:combostop1}
	\P \left( \sup_{0 \le t \le \tau_1} |\Pi_{\mathrm{ker}A^\perp} x_t| \le \delta^{1/16} (\log(K))^3 K^r \right) \ge 1 - \frac{\beta_j}{32}.
	\end{equation}
	By Lemma~\ref{lem:basicenergy}, \eqref{eq:combostop1}, \eqref{eq:frozenz3}, and assuming the second case in \eqref{eq:combocases}, for $K_*$ sufficiently large depending on $\beta_j$, $r$, and $\bar{r}$ we have
	$$\P(\bar{\tau} \le \tau_1) \ge \frac{\beta_j}{8}.$$
	Thus, by Lemma~\ref{lem:combotransverse} and the strong Markov property, there is $c_0'$ depending on $c_0$ so that for all $\delta$ sufficiently small and $K$ sufficiently large (both depending only on $c_0$) there holds
	\begin{align*} 
	\frac{1}{\tau}\E \int_0^\tau |\Pi_{\mathrm{ker}A^\perp} x_t|^2 dt &\ge \frac{1}{\tau} \int_{\Omega} \int_0^{\tau - \tau_1 \wedge \bar{\tau}} D(x_{\tau_1 \wedge \bar{\tau}+t})dt d\P  \\ 
	& \ge \frac{1}{\tau} \frac{\beta_j}{8} \inf_{\bar{\tau}(\omega) \le \tau_1} \int_0^{K^{-1}} \Pt_t D(x_{\bar{\tau}}(\omega)) dt \\ 
	& \ge \frac{1}{\tau K} \frac{\beta_j}{8} c_0' K^{2\bar{r}} \gtrsim \frac{\beta_j c_0' K^{2 \bar{r}}}{\log (K)}.
	\end{align*}
	Taking $K$ large enough so that $K^{2\bar{r}} \log(K)^{-1} \ge K^{2r}$ yields a contradiction with \eqref{eq:combocontr3} for $\delta$ small enough. Overall, we have shown that for all $\delta$ sufficiently small there are $K_* \ge 1$ and $c_* > 0$ so that, for $j \in \{1,2\}$ and $K \ge K_*$,
	$$ x_0 \in B_{K,j}^\delta \implies \frac{1}{\tau(x_0)}\E \int_0^{\tau(x_0)} |\Pi_{\mathrm{ker}A^\perp} x_t|^2 dt \ge c_* K^{2r}. $$
	The desired bound \eqref{eq:combogoal} then follows for $j \in \{1,2\}$ by setting 
	$$ \eta_1(K) = \eta_2(K) = \frac{C \log(K)}{\lambda_{\mathrm{min}} K} $$
	for some $C$ sufficiently large.
\end{proof}

\subsection{Applications to the Sabra shell model and Galerkin Navier-Stokes}

In this section we first apply Theorem~\ref{thm:combination} to prove Theorem~\ref{thm:NS} on the 2d Galerkin Navier-Stokes equations and then we give an application of Theorem~\ref{thm:combination} to the Sabra shell model.

\begin{proof}[Proof of Theorem~\ref{thm:NS}] 
	Note that the nonlinear structure implies that $L_x = \Pi_{V_j^\perp}L_x \Pi_{V_j^\perp}$ for any $x \in V_j$ and $j \in \{1,2\}$. By translation invariance, the linearization around $A \cos \ell x_1 + B\sin \ell x_1$ is unitarily conjugate to the linearization around $\sqrt{A^2 + B^2}\cos \ell x_1$,
	and so the uniformity of eigenvalues and $\norm{P}$ of the Jordan canonical form follows immediately on each $V_j$ once it is verified for $\cos \ell  x_1$ (by discrete rotation invariance, the analysis in the $x_2$ direction is also the same as the $x_1$ direction).
	That for any $\ell \geq 2$, the linearization of $\cos \ell x_1$ in $\mathbb T^2$  has an eigenvalue with positive real part in the conservative system is a well-known variation of the classical results of Meshalkin and Sinai \cite{MS61}. It remains only to verify the conditions \eqref{eq:combcancellation} and \eqref{eq:combolbd}.
	For $v_1 = \alpha \cos \ell x_1 + \beta \sin \ell x_1$ and $v_2 = \sin k x_2$ (this is sufficient by translation invariance), we may compute 
	\begin{align*} 
	B(v_1,v_2) + B(v_2,v_1) &= \frac{k}{\ell}(-\alpha \sin \ell x_1 + \beta \cos \ell x_1) \cos k x_2 - \frac{\ell}{k}\cos k x_2 (-\alpha \sin \ell x_1 + \beta \cos \ell x_1) \\
	& = \left(\frac{k}{\ell} - \frac{\ell}{k}\right)(-\alpha \sin \ell x_1 + \beta \cos \ell x_1) \cos k x_2. 
	\end{align*}
	For our choices of $k$,$\ell$ conditions \eqref{eq:combcancellation} and \eqref{eq:combolbd} follow immediately and hence Theorem \ref{thm:combination} applies. 
\end{proof}

The Sabra shell model was first introduced in \cite{LvovEtAl98}. Here we consider the model truncated to finite dimensions. Denoting the dependent variable $(u_1, \ldots, u_J) \in \C^J$, the equation reads
\begin{equation}\label{eq:SABRAcomplex}
\begin{aligned}
d u_m  &= i2^m\left(\overline{u_{m+1}}u_{m+2} - \frac{\delta}{2} \overline{u_{m-1}}u_{m+1} - \frac{\delta -1}{4} u_{m-2}u_{m-1}\right) \\ 
& \quad  - \delta 2^{2m} u_m + q_m dW_t^{(m;R)} + ip_m dW_t^{(m;I)},
\end{aligned}
\end{equation}
where $q_m$, $p_m$ are real parameters and $\delta \in (0,2) \setminus \{1\}$. The boundary conditions are $u_{-1} = u_0 = u_{J+1} = u_{J+2} = 0$. When $\delta \in (0,1)$ the system has just one positive invariant and is considered a model for 3d turbulence. If instead $\delta \in (1,2)$ then there are two positive invariants and the equations are meant to capture properties of 2d turbulence. For additional discussion of Sabra and other shell models, see \cite{Ditlevsen2010}. Rewriting the system in real variables $u_m = a_m + ib_m$ and introducing parameters $c_m \in \{0,1\}$ that determine whether or not there is damping on shell $m$, we obtain the system
\begin{equation} \label{eq:SABRAreal}
\begin{aligned}
da_m & = 2^m(a_{m+2}b_{m+1} - a_{m+1}b_{m+2}) + \delta 2^{m-1}(a_{m-1}b_{m+1} - a_{m+1}b_{m-1}) \\ 
& \quad + (\delta - 1)2^{m-2}(a_{m-2}b_{m-1} + a_{m-1}b_{m-2}) - \delta 2^{2m} c_m a_m + q_m dW_t^{(m;R)}, \\ 
db_m & = 2^m(a_{m+1}b_{m+2} + b_{m+1}b_{m+2}) - \delta 2^{m-1}(a_{m-1}a_{m+1} + b_{m+1}b_{m-1}) \\ 
& \quad - (\delta - 1)2^{m-2}(a_{m-2}a_{m-1} - b_{m-1}b_{m-2}) - \delta 2^{2m} c_m b_m + p_m dW_t^{(m;I)}.
\end{aligned}
\end{equation}
\begin{theorem}\label{thm:SABRA}
	Assume that $\delta \in (1/4,1)$, $c_1 = c_2 = 0$, $c_m > 0$ for $3 \le m \le J$, and $q_m,p_m \neq 0$ for all $m$. Then, system \eqref{eq:SABRAreal} admits a unique invariant measure $\mu$ and 
	$$ \int_{\R^{2J}} (|a| + |b|)^p \mu(da,db) < \infty $$
	for every $p < 2/3$.
\end{theorem}

\begin{proof}
	We denote the solution $(a,b) = (a_1,\ldots,a_J,b_1,\ldots,b_J) \in \R^J \times \R^J$ and the natural canonical basis vectors by $\{\hat{a}_m\}_{m=1}^J$, $\{\hat{b}_m\}_{m=1}^J$. Observe that \eqref{eq:SABRAreal} takes the form of \eqref{eq:SDE} with 
	$$\mathrm{ker}A = \{(a_1,a_2,0,\ldots,0,b_1,b_2,0,\ldots,0): a_1,b_1,a_2,b_2 \in \R\},$$
	and the drift $B$ given by 
	\begin{equation} \label{eq:SABRAdrift}
	\begin{aligned}
	B\left((a,b),(\tilde{a},\tilde{b})\right) &= \sum_{m=1}^J [2^m (a_{m+2}\tilde{b}_{m+1} - a_{m+1}\tilde{b}_{m+2}) +\delta 2^{m-1}(a_{m-1}\tilde{b}_{m+1} - a_{m+1} \tilde{b}_{m-1}) \\ 
	& \quad + (\delta - 1)2^{m-2}(a_{m-2}\tilde{b}_{m-1}+ a_{m-1}\tilde{b}_{m-2})] \hat{a}_m \\ 
	& \quad + \sum_{m=1}^J [2^m (a_{m+1}\tilde{b}_{m+2} + b_{m+1}\tilde{b}_{m+2})-\delta 2^{m-1}(a_{m-1}\tilde{a}_{m+1} + b_{m+1}\tilde{b}_{m-1}) \\ 
	& \quad -(\delta-1)2^{m-2}(a_{m-2}\tilde{a}_{m-1} - b_{m-1}\tilde{b}_{m-2})]\hat{b}_m.
	\end{aligned}
	\end{equation}
	
	We will verify the conditions of Theorem~\ref{thm:combination}. For $j = 1,2$ let 
	$$V_j = \mathrm{span}\{\hat{a}_j, \hat{b}_j\}.$$
	It is immediate from the lack of self-interactions in \eqref{eq:SABRAdrift} that each $V_j$ consists entirely of deterministic equilibria, and so the first condition in Theorem~\ref{thm:combination} is satisfied. Similarly, since the $m$'th coordinates of $(a,b)$ and $(\tilde{a},\tilde{b})$ do not show up in $B \cdot \hat{a}_m$ or $B \cdot \hat{b}_m$ it is easy to see that \eqref{eq:combcancellation} is satisfied. To verify the lower bound \eqref{eq:combolbd} we compute, for $v_1 = a_1 \hat{a}_1 + b_1 \hat{b}_1 \in V_1$ and $v_2 = a_2 \hat{a}_2 + b_2 \hat{b}_2 \in V_2$, 
	\begin{align*}
	|\Pi_{\mathrm{ker}A^\perp}(B(v_1,v_2) + B(v_2,v_1))|& = |\Pi_{\mathrm{span}\{\hat{a}_3,\hat{b}_3\}}(B(v_1,v_2) + B(v_2,v_1))| \\ 
	& = |2(\delta-1)| |(a_1 b_2 + a_2 b_1)\hat{a}_3 - (a_1a_2 - b_1b_2)\hat{b}_3| \\ 
	& = |2(\delta-1)| \sqrt{(a_1 b_2 + a_2 b_1)^2 + (a_1a_2 - b_1b_2)^2} \\ 
	& =  |2(\delta-1)| \sqrt{(a_1^2 + b_1^2)(a_2^2 + b_2^2)} \\
	& = |2(\delta-1)||v_1| |v_2|.
	\end{align*}
	It remains only to check the second and third conditions. This requires computing the linearized operators $\Pi_{V_j^\perp} L_x \Pi_{V_j^\perp}$ for $x \in V_j$ with $|x| = 1$. For $j = 1$, let 
	$$x = (\bar{a}_1,0,\ldots,0,\bar{b}_1,0,\ldots,0)$$ 
	for $\bar{a}_1, \bar{b}_1 \in \R$ satisfying $\sqrt{\bar{a}_1^2 + \bar{b}_1^2} = 1$. For general $(a,b) \in \R^J \times \R^J$ we compute 
	$$
	\Pi_{V_1^\perp} L_x \Pi_{V_1^\perp} (a,b) = \begin{pmatrix}
	0 & 0 & -2\delta \bar{b}_1 & 2 \delta \bar{a}_1 \\ 
	0 & 0 & -2\delta \bar{a}_1 & -2\delta \bar{b}_1 \\ 
	2(\delta - 1)\bar{b}_1 & 2(\delta - 1) \bar{a}_1 & 0 & 0 \\ 
	-2(\delta - 1)\bar{a}_1 & 2(\delta - 1)\bar{b}_1 & 0 & 0
	\end{pmatrix}
	\begin{pmatrix}
	a_2 \\ b_2 \\ a_3 \\ b_3
	\end{pmatrix},
	$$
	with the components not shown being zero. The eigenvalues and associated eigenvectors of the matrix above are given by 
	\begin{equation} \label{eq:SABRAevector1}
	\lambda_{+,1} = 2 \sqrt{\delta(1-\delta)}, \quad E_{+,1} = \left\{ 
	\begin{pmatrix}
	-\bar{a}_1 \sqrt{\frac{\delta}{\delta-1}} \\ \bar{b}_1 \sqrt{\frac{\delta}{\delta-1}} \\ 
	0 \\
	1
	\end{pmatrix},
	\begin{pmatrix}
	\bar{b}_1 \sqrt{\frac{\delta}{\delta-1}} \\ \bar{a}_1 \sqrt{\frac{\delta}{\delta-1}} \\ 
	1 \\
	0
	\end{pmatrix}
	\right\},
	\end{equation}
	\begin{equation}
	\lambda_{-,1} = -2 \sqrt{\delta(1-\delta)}, \quad E_{-,1} = \left\{ 
	\begin{pmatrix}
	\bar{a}_1 \sqrt{\frac{\delta}{\delta-1}} \\ -\bar{b}_1 \sqrt{\frac{\delta}{\delta-1}} \\ 
	0 \\
	1
	\end{pmatrix},
	\begin{pmatrix}
	-\bar{b}_1 \sqrt{\frac{\delta}{\delta-1}} \\ -\bar{a}_1 \sqrt{\frac{\delta}{\delta-1}} \\ 
	1 \\
	0
	\end{pmatrix}
	\right\}.
	\end{equation}
	For $x = \bar{a}_2 \hat{a}_2 + \bar{b}_2 \hat{b}_2$ with $|x| = 1$ we similarly have  
	$$
	\Pi_{V_2^\perp}L_x \Pi_{V_2^\perp}(a,b) 
	=
	\begin{pmatrix}
	0 & 0 & 2\bar{b}_2 & -2 \bar{a}_2 & 0 & 0 \\ 
	0 & 0 & 2\bar{a}_2 & 2\bar{b}_2 & 0 & 0 \\ 
	2(\delta-1)\bar{b}_2 &  2(\delta-1)\bar{a}_2 & 0 & 0 &-4\delta \bar{b}_2 & 4 \delta \bar{a}_2 \\ 
	-2(\delta-1)\bar{a}_2 & 2(\delta-1)\bar{b}_2 & 0 & 0 & -4\delta \bar{a}_2 & -4\delta \bar{b}_2 \\ 
	0 & 0 & 4(\delta-1)\bar{b}_2 & 4(\delta-1)\bar{a}_2 & 0 & 0 \\ 
	0 & 0 & -4(\delta-1)\bar{a}_2 & 4(\delta-1)\bar{b}_2 & 0 & 0
	\end{pmatrix}
	\begin{pmatrix}
	a_1 \\ b_1 \\ a_3 \\ b_3 \\ a_4 \\ b_4
	\end{pmatrix}.
	$$
	Defining $c_\delta = 5\delta - 4\delta^2 - 1 > 0$, the eigenvalues and associated eigenvectors are given by 
	\begin{equation}
	\lambda_0 = 0, \quad E_0 = \left\{ \begin{pmatrix}
	\frac{-4 \bar{a}_2 \bar{b}_2 \delta}{\delta - 1} \\ \frac{2 \bar{b}_2^2 \delta - 2 \bar{a}_2^2 \delta}{\delta - 1} \\ 0 \\ 0 \\ 0 \\ 1
	\end{pmatrix}, 
	\begin{pmatrix}
	\frac{2 \bar{b}_2^2 \delta - 2 \bar{a}_2^2 \delta}{\delta - 1} \\ \frac{4 \bar{a}_2 \bar{b}_2 \delta}{\delta - 1} \\ 0 \\ 0 \\ 1 \\ 0
	\end{pmatrix}
	\right\},
	\end{equation}
	\begin{equation}
	\lambda_{+,2} = 2\sqrt{c_\delta}, \quad E_{+,2} = \left\{
	\begin{pmatrix} \frac{-\bar{a}_2 \bar{b}_2}{\delta - 1} \\ \frac{\bar{b}_2^2 -\bar{a}_2^2}{2(\delta - 1)} \\ \frac{-\bar{a}_2 \sqrt{c_\delta}}{2(\delta - 1)} \\ \frac{\bar{b}_2 \sqrt{c_\delta}}{2(\delta - 1)} \\ 0 \\ 1	
	\end{pmatrix}, 
	\begin{pmatrix} \frac{\bar{b}_2^2 - \bar{a}_2^2}{2(\delta - 1)} \\ \frac{\bar{a}_2 \bar{b}_2}{\delta - 1} \\ \frac{\bar{b}_2 \sqrt{c_\delta}}{2(\delta - 1)} \\ \frac{\bar{a}_2 \sqrt{c_\delta}}{2(\delta - 1)} \\ 1 \\ 0	
	\end{pmatrix}
	\right\},
	\end{equation}
	and 
	\begin{equation} \label{eq:SABRAevector5}
	\lambda_{-,2} = 2\sqrt{c_\delta}, \quad E_{-,2} = \left\{
	\begin{pmatrix} \frac{-\bar{a}_2 \bar{b}_2}{\delta - 1} \\ \frac{\bar{b}_2^2 -\bar{a}_2^2}{2(\delta - 1)} \\ \frac{\bar{a}_2 \sqrt{c_\delta}}{2(\delta - 1)} \\ \frac{-\bar{b}_2 \sqrt{c_\delta}}{2(\delta - 1)} \\ 0 \\ 1	
	\end{pmatrix}, 
	\begin{pmatrix} \frac{\bar{b}_2^2 - \bar{a}_2^2}{2(\delta - 1)} \\ \frac{\bar{a}_2 \bar{b}_2}{\delta - 1} \\ \frac{-\bar{b}_2 \sqrt{c_\delta}}{2(\delta - 1)} \\ \frac{-\bar{a}_2 \sqrt{c_\delta}}{2(\delta - 1)} \\ 1 \\ 0	
	\end{pmatrix}
	\right\}.
	\end{equation}
	Since, for each $j$, $\lambda_{+,j}$ is positive and independent of $x \in V_j \cap \S^{n-1}$, we see that the third condition of Theorem~\ref{thm:combination} is satisfied. Lastly, \eqref{eq:JCFassumption2} follows from the formula for the eigenvectors given in \eqref{eq:SABRAevector1}-\eqref{eq:SABRAevector5}. This completes the proof of Theorem~\ref{thm:SABRA}.
\end{proof}

\section{Lorenz-96 with a two-dimensional kernel} \label{sec:L96}
%%%

In this section, we consider the stochastic Lorenz-96 system for $x_t = (x_{t,1}, \ldots, x_{t,n}) \in \R^n$ (with $n \ge 6$) defined by
\begin{equation} \label{eq:L96sde}
dx_{t,j} = -a_j x_{t,j}dt + B_j(x_t,x_t)dt + \sigma_j dW_t^{(j)},
\end{equation}
where $x_{t,k} = x_{t,k+n}$ (periodic conditions), $a_j \ge 0$, and
\begin{equation}
B_j(x,x) = (x_{j+1} - x_{j-2})x_{j-1}.
\end{equation}
Consistent with our earlier notation, we write $A = \text{diag}(a_1,\ldots, a_n)$. A consequence of Theorem~\ref{thm:basic1sigma} is that \eqref{eq:L96sde} admits an invariant measure when $a_1 = 0$, $\sigma_{n-1} \neq 0$, and $a_j > 0$ for all $2 \le j \le n$. Indeed, in this case $B(x,x) = 0$ for every $x \in \mathrm{ker}A$ and moreover it is straightforward to check that for $z = (z_0,0,\ldots,0) \in \mathrm{ker}A$ the linearized operator $L_z^\perp$ is Jordan block unstable with 
$$ (|z| t)^3 \leqc \|e^{L_z^\perp t}\| \leqc 1 + (|z|t)^3 \quad \text{and} \quad |e^{L_z^\perp t} e_{n-1}| \gtrsim |z|t,$$
where $e_{n-1}$, which denotes the usual canonical basis vector, is a generalized eigenvector that is the last element of a Jordan chain. Our goal in this section is to show that an invariant measure can in fact also be constructed using our methods in the more degenerate case where $a_1 = a_2 = 0$. The main result is stated precisely as follows.
\begin{theorem} \label{thm:L96twomodes}
	Let $ 6 \le n < \infty$. The stochastic Lorenz-96 system with $a_1 = a_2 = 0$ and $a_j > 0$ for $3 \le j \le n$ admits an invariant measure $\mu_*$ provided that $\sigma_n, \sigma_{n-1} \neq 0$. Moreover, we have the moment bound 
	$$ \int_{\R^n} |x|^{p}\mu_*(dx) < \infty$$
	for every $0 < p < 1/3$.
\end{theorem}

As in the earlier sections, we will prove Theorem~\ref{thm:L96twomodes} by verifying Assumption~\ref{ass:time-average} using suitable approximation arguments for solutions in the vicinity of $\mathrm{ker}A$. To this end, for $K \gg 1$ and $\delta \in (0,1)$ we split the set 
$$B_{K,\delta} = \{x \in \R^n: |\Pi_{\mathrm{ker}A^\perp} x| \le \delta |\Pi_{\mathrm{ker}A}x|^{1/7} \text{ and } K/2 \le |x| \le 2K \}$$
as 
$$ B_{K,\delta} = B_{K,\delta}^1 \cup B_{K,\delta}^2 \cup B_{K,\delta}^3,$$
where, for some small parameter $\delta_1 \in (0,1)$, 
$$ B_{K,\delta}^1 = \{x \in B_{K,\delta}: |x_1| \ge K/\sqrt{32}\}, $$
$$ B_{K,\delta}^2 = \{x \in B_{K,\delta}: \delta_1 K^{1/7} \le |x_1| < K/\sqrt{32}\},$$
and
$$ B_{K,\delta}^3 = \{x \in B_{K,\delta}: |x_1| < \delta_1 K^{1/7}\}. $$
Note that $|x_2| \ge K/\sqrt{32}$ for $x \in B_{K,\delta}^2 \cup B_{K,\delta}^3$.

In the region  $B_{K,\delta}^1$, we can use a treatment similar to that used for Jordan block unstable equilibria in Theorem \ref{thm:basic1}.
Specifically, we show that a large $x_1$ induces a significant growth in $x_{n}$ through the interaction $\dot{x}_n = x_1 x_{n-1} + ...$. He we rely on the fact that $x_{n-1}$ is being driven by a Brownian motion (since $\sigma_{n-1} \neq 0$), which ensures it is non-trivial with high probability.
In the region $B_{K,\delta}^2$, we can use a treatment similar to that used in Theorem \ref{thm:hypo}, by noting that $\dot{x}_3 = -x_1 x_2 + ...$ and hence if both $x_1$ and $x_2$ are sufficiently large, then $x_3$ will rapidly grow.
The region $B_{K,\delta}^3$ is the region that is most different from previous cases. Here, the Jordan block instability of the equilibrium $e_2$ excites $x_1$, which is still in $\mathrm{ker} A$.
Heuristically, we show that  solutions which start in $B_{K,\delta}^3$ are basically ejected into $B_{K,\delta}^2$, where they are subsequently ejected into $(\mathrm{ker} A)^\perp$. 

By Lemma~\ref{lem:time-averaged_step1}, Theorem \ref{thm:L96twomodes} is a direct consequence of the following time-averaged coercivity estimates.

\begin{proposition} \label{lem:L96Assumption1}
	Let $\tau_1(K) = \tau_3(K) = K^{-4/7}$ and $\tau_2(K) = K^{-1}$.
	There exist $K_* \ge 1$, $c_* > 0$, and $\delta, \delta_1 \in (0,1)$ so that if $K \ge K_*$ and $x_0 \in B_{K,\delta}^j$ then 
	$$ \frac{1}{\tau_j(K)} \int_0^{\tau_j(K)} |\Pi_{\mathrm{ker}A^\perp} x_t|^2 dt \ge c_* K^{2/7}. $$
	Consequently, Assumption~\ref{ass:time-average} is satisfied with $r = 1/7$.
\end{proposition} 

As discussed above the region $B_{K,\delta}^3$ is the most involved.
The main difficulty here is to deduce growth of $\Pi_{\mathrm{ker}A^\perp}X_t$ for a suitable approximate solution when $X_0 \in B_{K,\delta}^3$, as proved in the next lemma. 

\begin{lemma} \label{lem:L96Case3_growth}
	Let $X_t$ solve 
	\begin{equation} \label{eq:L96approx}
	\begin{cases}
	dX_{t,n} = X_{t,1}X_{t,n-1}dt + \sigma_n dW_t^{(n)} \\ 
	dX_{t,1} = X_{t,2}X_{t,n}dt \\ 
	dX_{t,3} = -X_{t,1}X_{t,2}dt \\ 
	dX_{t,j} = 0 & j \not \in \{n,1,3\}.
	\end{cases}
	\end{equation}
	with initial condition $X_0 \in B_{K,\delta}^3$ and some $\sigma_n \neq 0$. For $\delta$ and $\delta_1$ chosen sufficiently small, there are constants  $\beta, c_* > 0$ (independent of $X_0$) so that for all $K$ sufficiently large and $\tau = K^{-4/7}$ there holds 
	\begin{equation}\label{eq:L96Case3_growth}
	\P\left(\frac{1}{\tau}\int_0^\tau |X_{t,3}| dt \ge c_* K^{4/7}\right)\ge \beta.
	\end{equation}
\end{lemma}

\begin{proof}
	
	Without loss of generality we set $\sigma_n dW_t^{(n)} = dW_t$ for a standard Brownian motion $W_t$. We also write $r = 1/7$, so that $\tau = K^{-4r}$ and moreover from the definition of $B^3_{K,\delta}$ we have
	\begin{align}
	\sum_{j=3}^n |X_{0,j}|^2 & \le \delta^2 (|X_{0,1}|^2 + |X_{0,2}|^2)^r \le 4 \delta^2 K^{2r}, \label{eq:BK3boundX01} \\ 
	|X_{0,1}| &\le \delta_1 K^r. \label{eq:BK3boundX02}
	\end{align}
	For $R \gg 1$ to be chosen we split into the cases $|X_{0,n}| \ge R \sqrt{\tau}$ and $|X_{0,n}| < R \sqrt{\tau}$. In the former, we approximate $X_{t,n} \approx X_{0,n}$  and in the latter we approximate $X_{t,n} \approx X_{0,n} + W_t$.
	
	\textbf{Case 1} ($|X_{0,n}| \ge R\sqrt{\tau}$): Write  $X_{t,n} = X_{0,n} + E_t$, where $E_t$ is an error to be controlled. Substituting this into the system we have
	\begin{equation} 
	\begin{cases}
	dE_t = X_{t,1}X_{0,n-1}dt + dW_t \\
	dX_{t,1} = X_{0,2}(X_{0,n} + E_t)dt \\ 
	dX_{t,3} = -X_{t,1}X_{0,2}dt.
	\end{cases}
	\end{equation}
	Thus,
	\begin{equation} \label{eq:L96case3X1}
	X_{t,1} = X_{0,1} + tX_{0,2}X_{0,n} + X_{0,2}\int_0^t E_s ds
	\end{equation}
	and 
	$$
	E_t = tX_{0,n-1}X_{0,1} + \frac{t^2}{2}X_{0,n-1}X_{0,n}X_{0,2} + X_{0,n-1}X_{0,2}\int_0^t (t-s)E_s ds + W_t.
	$$
	Recalling $|X_0| \le 2K$ and using the bounds \eqref{eq:BK3boundX01} and \eqref{eq:BK3boundX02} we have 
	\begin{equation}
	|E_t| \le 2\delta \delta_1 t K^{2r} + 2\delta t^2 K^{r+1}|X_{0,n}| + 4\delta t K^{r+1}\int_0^t |E_s|ds + |W_t|,
	\end{equation}
	and therefore
	\begin{equation}
	\E \sup_{0 \le t' \le t} |E_{t'}| \le 2t\delta \delta_1 K^{2r}  + 2\delta t^2 K^{1+r}|X_{0,n}| + 4 \delta t K^{1+r}\int_0^t \E\sup_{0 \le t' \le s}|E_{t'}|ds + \sqrt{t}
	\end{equation}
	By Gr\"{o}nwall's Lemma, $|X_{0,n}| \ge RK^{-2/7}$, and the definitions of $r$ and $\tau$, it follows that 
	\begin{align*}
	\E \sup_{0 \le t \le \tau} |E_t| &\le (2\tau \delta \delta_1 K^{2r} + 2\delta \tau^2 K^{1+r}|X_{0,n}| + \sqrt{\tau}) \exp(4\delta \tau^2 K^{1+r}) \\ 
	& \le C (K^{-2/7} + \delta|X_{0,n}|) \\ 
	& \le 2C \max(\delta,R^{-1})|X_{0,n}|,
	\end{align*}
	where $C$ is a constant that does not depend on $\delta$, $\delta_1$, or $K$. Putting in $R = \delta^{-1/2}$ we conclude 
	\begin{equation} \label{eq:L96approxCase3_1}
	\E \sup_{0 \le t \le \tau} |E_t| \le 2 C\sqrt{\delta} |X_{0,n}|.
	\end{equation}
	
	The goal is now to use \eqref{eq:L96approxCase3_1} to show that $X_{t,3}$ must grow. By \eqref{eq:L96approxCase3_1} and Chebyshev's inequality, for $\delta$ sufficiently small we have
	\begin{equation}
	\P \left(\sup_{0 \le t \le \tau} |E_t| \le \delta^{1/4}|X_{0,n}|\right) \ge \frac{1}{2}.
	\end{equation} 
	Suppose that $\omega \in \Omega$ is such that 
	\begin{equation} \label{eq:L96Case3Ebound1}
	\sup_{0 \le t \le \tau}|E_t(\omega)| \le \delta^{1/4}|X_{0,n}|.
	\end{equation} 
	We will show that in this case, for $\delta$ sufficiently small, there is $c_* \in (0,1]$ (independent of $\omega, X_0, K$) such that
	\begin{equation} \label{eq:L96Case3goal1}
	\int_0^\tau |X_{t,3}(\omega)|dt \ge c_* \tau K^{4/7},
	\end{equation}
	which is sufficient to imply \eqref{eq:L96Case3_growth} with $\beta = 1/2$. From \eqref{eq:L96case3X1} and \eqref{eq:L96Case3Ebound1} we have (suppressing now the dependence on $\omega$ from the notation)
	\begin{equation}
	|X_{t,1} - t X_{0,2} X_{0,n}| \le |X_{0,1}| + \delta^{1/4}|X_{0,2}||X_{0,n}|t \quad \forall t\in[0,\tau]. 
	\end{equation}
	Applying this bound in the formula for $X_{t,3}$ and using $|X_{0,3}| \le K^r$ we get, for $t \in [0,\tau]$, 
	\begin{align*}
	|X_{t,3}|  
	& \ge \left|\int_{0}^t sX_{0,2}^2X_{0,n} ds\right| - |X_{0,2}|\int_{0}^t (|X_{0,1}| + s\delta^{1/4}|X_{0,2}||X_{0,n}|)ds - K^r \\ 
	& \ge \frac{t^2}{2}|X_{0,2}|^2 |X_{0,n}|-t|X_{0,2}||X_{0,1}| - \delta^{1/4} t^2 |X_{0,2}|^2 |X_{0,n}| - K^r \\ 
	& \ge \frac{t^2}{4}|X_{0,2}|^2 |X_{0,n}|-t|X_{0,2}||X_{0,1}| -K^r,
	\end{align*}
	where in the last inequality we have assumed that $\delta$ is sufficiently small. We thus have
	\begin{equation} \label{eq:L96Case3X3lbd1}
	|X_{t,3}| \ge \frac{\tau^2}{16}|X_{0,2}|^2|X_{0,n}| - \tau|X_{0,2}||X_{0,1}| - K^r \quad \forall t \in [\tau/2,\tau].
	\end{equation}
	Using \eqref{eq:L96Case3X3lbd1}, $|X_{0,1}| \le K^r$, $\delta^{-1/2}K^{-2r} \le |X_{0,n}| \le 2\delta K^{r}$, and $K/\sqrt{32} \le |X_{0,2}| \le 2K$ it follows that for $t \in [\tau/2,\tau] = [K^{-4r}/2,K^{-4r}]$ and $\delta$ sufficiently small there holds
	\begin{align*}
	|X_{t,3}| & \ge \frac{\delta^{-1/2} K^{2-10r} }{512} - 2 K^{1-3r} - K^r  \ge \frac{1}{512} \delta^{-1/2} K^{4/7} - 3K^{4/7},
	\end{align*}
	where we have noted that the choice $r = 1/7$ implies $2-10r = 1-3r = 4/7$. Hence, for $\delta$ sufficiently small we have $|X_{t,3}| \ge K^{4/7}$ for $t \in [\tau/2,\tau]$, and so
	$$ \int_0^{\tau} |X_{t,3}| dt \ge \frac{1}{2} \tau K^{4/7},$$
	which proves \eqref{eq:L96Case3goal1}.
	
	\textbf{Case 2} ($|X_{0,n}| \le R \sqrt{\tau} = \delta^{-1/2}\sqrt{\tau}$): Let $X_{t,n} = X_{0,n} + W_t + E_t$, where again $E_t$ is an error to be bounded. Computations similar to those of Case 1 give 
	\begin{equation}  \label{eq:L96approxCase3_2}
	\E \sup_{0 \le t \le \tau} \frac{|E_t|}{\sqrt{t}} \le C R \delta \le C\sqrt{\delta},
	\end{equation}
	where $C$ is a constant that does not depend on $\delta$ or $K$. Now we justify the growth of $X_{t,3}$, which is also similar to above. We have
	\begin{equation}
	X_{t,1} = X_{0,1} +X_{0,2}X_{0,n}t + X_{0,2}\int_0^t W_s ds + X_{0,2}\int_0^t E_s ds.
	\end{equation}
	Without loss of generality we may assume that $X_{0,2}X_{0,n} \ge 0$. By the scaling and support theorems for Brownian motion, there exists $\alpha > 0$ that does not depend on $K$ such that 
	\begin{equation} \label{eq:sptthrm}
	\P\left(X_{0,2}\int_0^t W_s ds \ge |X_{0,2}| \tau^{3/2} \quad \forall t \in [\tau/4,\tau]\right) \ge \alpha. 
	\end{equation}
	By \eqref{eq:L96approxCase3_2} and \eqref{eq:sptthrm}, if $\delta$ is small enough we have 
	\begin{equation} \label{eq:L96Wienerapprox}
	\P\left( X_{0,2} \int_0^t W_s ds \ge |X_{0,2}| \tau^{3/2} \quad \forall t \in [\tau/4,\tau] \quad \text{and}\quad \sup_{0 \le t \le \tau}|E(t)|/\sqrt{t} \le \delta^{1/4})\right) \ge \frac{\alpha}{2}. 
	\end{equation}
	Let $\omega \in \Omega$ be such that the two bounds in \eqref{eq:L96Wienerapprox} hold true. We will prove that for such an $\omega$ one has 
	$$\int_0^\tau |X_{t,3}(\omega)|dt \ge c_* \tau K^{4/7}$$
	for $c_*$ sufficiently small. First, there is nothing to show if 
	$$ \int_{0}^\tau |X_{t,3}(\omega)| dt \ge \frac{\tau}{8000} K^{4/7},$$ 
	so suppose otherwise. In this case, there exists $t_0 \in [\tau/4,\tau/2]$ is such that $|X_{t_0,3}| \le K^{4/7}/2000$. Then, for $t \in [t_0,\tau]$ there holds 
	\begin{align*}
	|X_{t,3}| 
	& \ge \left|\int_{t_0}^t X_{0,2}^2\left(X_{0,n}s + \int_0^s W_{s'}ds'\right)ds\right| \\ 
	& \quad - \int_{t_0}^t |X_{0,2}|\left(|X_{0,1}| + |X_{0,2}|\int_0^s |E_{s'}|ds'\right)ds - \frac{K^{4/7}}{2000} \\ 
	& \ge |X_{0,2}|^2 (t-t_0)\tau^{3/2} - t|X_{0,2}||X_{0,1}| - \delta^{1/4}|X_{0,2}|^2 t^{5/2} - \frac{K^{4/7}}{2000},
	\end{align*} 
	where in obtaining the final inequality we have noted that $X_{0,n}$ and $\int_0^s W_{s'}ds'$ have the same sign for $s \ge t_0$ since $X_{0,n}X_{0,2} \ge 0$. Taking $\delta$ sufficiently small and $t \in [3\tau/4,\tau]$ to absorb the third term by the first we obtain
	\begin{equation}
	|X_{t,3}| \ge \frac{K^2 \tau^{5/2}}{256} - 2\delta_1\tau K^{1+r} - \frac{K^{4/7}}{2000} \ge \frac{K^{4/7}}{512} - 2\delta_1 K^{4/7},
	\end{equation}
	where we have recalled also that $K^2/32 \le X_{0,2}^2 \le 4K^2$, $|X_{0,1}| \le \delta K^r$, and $2-10r = 1-3r = 4/7$.
	For $\delta_1$ sufficiently small we conclude that $|X_{t,3}| \ge K^{4/7}/1024$ for $t \in [3\tau/4,\tau]$. Thus,
	$$ \P\left(\int_0^{\tau} |X_{t,3}|dt \ge \frac{1}{4096}\tau K^{4/7}\right)\ge \frac{\alpha}{2}, $$
	which completes the proof. 
\end{proof}

We are now ready to prove Proposition~\ref{lem:L96Assumption1}.

\begin{proof}[Proof of Proposition~\ref{lem:L96Assumption1}]
	Let $x_0 \in B_{K,\delta}$ for $K$ to be taken sufficiently large and $\delta,\delta_1$ chosen appropriately. As before, set $r = 1/7$ for the sake of simplifying the presentation of the estimates. Let $\tau_j(K)$ be as given in the statement of the proposition. There are three cases to consider.
	
	\textbf{Case 1} ($x_0 \in B_{K,\delta}^1)$: Consider the approximate solution $X_t$ defined by 
	\begin{equation}
	\begin{cases}
	dX_{t,n} = X_{t,1}X_{t,n-1}dt + \sigma_{n}dW_t^{(n)} \\ 
	dX_{t,n-1} = \sigma_{n-1}dW_t^{(n-1)} \\ 
	dX_{t,j} = 0 & j \not \in \{n,n-1\}
	\end{cases}
	\end{equation}
	and initial condition $X_0 = x_0$. We have 
	$$ X_{t,n} = X_{0,n} + t X_{0,1}X_{0,n-1} + \sigma_{n-1} X_{0,1} \int_0^t W_{s}^{(n-1)}ds + \sigma_n W_t^{(n)}. $$
	Similar to as in Case 2 from the proof of Lemma~\ref{lem:L96Case3_growth}, using the support theorem for Brownian motion we can show that
	\begin{equation}
	\P\left(|X_{t,n}| \ge K \tau_1^{3/2} \quad \forall t\in[\tau_1/2,\tau_1]\right)\ge \alpha
	\end{equation}
	for some $\alpha > 0$ that does not depend on $K$ or $\delta$. Consequently, since $K \tau_1^{3/2} = K^r$,
	\begin{equation} \label{eq:L96Case1_1}
	\P\left(\frac{1}{\tau_1}\int_0^{\tau_1}|X_{t,n}| dt \ge \frac{1}{2}K^r \right) \ge \alpha.
	\end{equation}
	
	Suppose now for the sake of contradiction that 
	\begin{equation} \label{eq:L96contr1}
	\E \int_0^{\tau_1} |\Pi_{\mathrm{ker}A^\perp} x_t|^2 dt = \E \sum_{j=3}^n\int_0^{\tau_1}|x_{t,j}|^2 \le \delta \tau_1 K^{2r}. 
	\end{equation}
	The error $X_{t,n} - x_{t,n}$ solves 
	\begin{equation}
	d(X_{t,n} - x_{t,n}) = x_{0,1}(X_{t,n-1} - x_{t,n-1})dt + x_{t,n-1}(x_{0,1} - x_{t,1})dt + a_n x_{t,n}dt + x_{t,n-2}x_{t,n-1}dt
	\end{equation}
	with zero initial condition, and so 
	\begin{equation}\label{eq:L96Case1error}
	\begin{aligned}
	|X_{t,n} - x_{t,n}| &\le 2K\int_0^t |X_{s,n-1} - x_{s,n-1}|ds + \int_0^t|x_{s,n-1}||x_{0,1} - x_{s,1}|ds \\ 
	& + |a_n|\int_0^t |x_{s,n}|ds + \int_0^t |x_{s,n-2}||x_{s,n-1}|ds.
	\end{aligned}
	\end{equation}
	We now obtain bounds on $|X_{t,n-1} - x_{t,n-1}|$ and $|x_{0,1} - x_{t,1}|$.
	By \eqref{eq:L96contr1} and $\E|x_t| \leqc K$ for $t \le 1$ we have
	\begin{equation} \label{eq:L96Case1_2}
	\E\sup_{0 \le t \le \tau_1} |x_{0,1} - x_{t,1}| \le C\left(\sqrt{\delta}\tau_1 K^{1+r} + \sqrt{\tau_1}\right) \le C (\sqrt{\delta}+K_*^{r-1}) K^{1-3r}.
	\end{equation}
	Moreover, a straightforward application of \eqref{eq:L96contr1} yields
	\begin{equation} \label{eq:L96Case1_3}
	\E \sup_{0 \le t \le \tau_1}|X_{t,n-1} - x_{t,n-1}| \le C \sqrt{\delta} K^{-2r}.
	\end{equation}
	Let
	\begin{align*} 
	\Omega_1 &= \left\{\omega \in \Omega: \sum_{j=3}^{n}\int_0^{\tau_1} |x_{s,j}|^2 ds \le \sqrt{\delta}K^{-2r}\right\}, \\ 
	\Omega_2 & = \left\{\omega \in \Omega: \sup_{0 \le t \le \tau_1} |x_{0,1} - x_{t,1}| \le K^{1-3r}\right\}, \\ 
	\Omega_3 & = \left\{ \omega \in \Omega: \sup_{0 \le t \le \tau_1}|X_{t,n-1} - x_{t,n-1}| \le \delta^{1/4}K^{-2r} \right\}, \\ 
	\tilde{\Omega} & = \Omega_1 \cap \Omega_2 \cap \Omega_3.	\end{align*}
	By \eqref{eq:L96contr1}, \eqref{eq:L96Case1_2}, and \eqref{eq:L96Case1_3} for $\delta$ sufficiently small and $K_*$ sufficiently large we have $\P(\tilde{\Omega}) \ge 1 - \alpha/2$. Now, by \eqref{eq:L96Case1error} and $\tau_1 = K^{-4r}$, for $\omega \in \tilde{\Omega}$ we have 
	$$
	\sup_{0 \le t \le \tau_1} |X_{t,n} - x_{t,n}| \le C\delta^{1/4}K^{1-6r} = C\delta^{1/4}K^r,
	$$
	and hence 
	\begin{equation} \label{eq:L96Case1_4}
	\P\left(\int_0^{\tau_1}|X_{t,n} - x_{t,n}|dt \le \tau_1 C\delta^{1/4}K^r\right) \ge 1-\frac{\alpha}{2}.
	\end{equation}
	Combining \eqref{eq:L96Case1_1} and \eqref{eq:L96Case1_4} we see that for $\delta$ sufficiently small there holds 
	$$ \P\left(\int_0^{\tau_1}|x_{t,n}|dt \ge \frac{\tau_1}{4}K^r\right)\ge \frac{\alpha}{2}. $$
	This is enough to yield a contradiction for $\delta$ sufficiently small. 
	
	\textbf{Case 2} ($x_0 \in B_{K,\delta}^2$): 
	Consider the approximate solution $X_t$ defined simply by 
	\begin{equation}
	\begin{cases}
	dX_{t,3} = -X_{t,1}X_{t,2}dt \\ 
	dX_{t,j} = 0 & j \neq 3
	\end{cases}
	\end{equation}
	with initial condition $X_0 = x_0$. We have then 
	$$X_{t,3} = X_{0,3} - tX_{0,1}X_{0,2}$$
	so that the bounds on $x_0 \in B_{K,\delta}^2$ imply
	$$\frac{1}{\tau_2}\int_0^{\tau_2}|X_{t,3}|dt \ge \frac{\tau_2}{64}\delta_1 K^{1+r} - \delta K^r = \left(\frac{\delta_1}{64}-\delta\right)K^r.$$
	Taking $\delta \ll \delta_1$ yields
	\begin{equation} \label{eq:L96Case2_growth}
	\frac{1}{\tau_2}\int_0^{\tau_2}|X_{t,3}| dt \ge \frac{\delta_1}{128}K^r.
	\end{equation}
	The error satisfies 
	\begin{equation}\label{eq:L96Case2_error}
	d(X_{t,3} - x_{t,3}) = (x_{t,1} - x_{0,1})x_{t,2}dt +x_{0,1}(x_{t,2}- x_{0,2})dt + a_3 x_{t,3}dt -x_{t,4}x_{t,2}dt - \sigma_3 dW_t^{(3)}. 
	\end{equation}
	Supposing for contradiction that 
	\begin{equation} \label{eq:L96contr2}
	\E \int_0^{\tau_2} |\Pi_{\mathrm{ker}A^\perp} x_t|^2 dt = \E \sum_{j=3}^n\int_0^{\tau_2}|x_{t,j}|^2 \le \delta \tau_2 K^{2r}
	\end{equation}
	we easily derive
	$$ \E \sup_{0\le t \le \tau_2}|x_{t,1} - x_{0,1}| + \E \sup_{0\le t \le \tau_2}|x_{t,2} - x_{0,2}| \le C(\sqrt{\delta}\tau_2 K^{1+r} + \sqrt{\tau_2})\le C\max(\sqrt{\delta},K_*^{-1/2})K^r. $$
	Therefore, choosing $K_* = \delta^{-1}$ and defining 
	\begin{align*}
	\Omega_1 & = \{\omega \in \Omega: \sup_{0\le t \le \tau_2}|x_{t,1} - x_{0,1}| + \sup_{0\le t \le \tau_2}|x_{t,2} - x_{0,2}| \le \delta^{1/4}K^r \},\\
	\Omega_2 & = \{\omega \in \Omega: \sup_{0\le t \le \tau_2}|\sigma_3||W_{t}^{(3)}|\le \delta^{1/4}K^r \}, \\ 
	\Omega_3 & = \{\omega \in \Omega: \sum_{j=3}^n \int_0^{\tau_2}|x_{t,j}|^2 dt \le \sqrt{\delta}\tau_2 K^{2r}\},
	\end{align*}
	we have $\P(\Omega_1 \cap \Omega_2 \cap \Omega_3) \ge 1/2$ for $\delta$ taken sufficiently small. Let $\omega \in \Omega_1 \cap \Omega_2 \cap \Omega_3$. Returning to \eqref{eq:L96Case2_error} we obtain 
	\begin{equation}
	\sup_{0 \le t \le \tau_2} |X_{t,3}(\omega) - x_{t,3}(\omega)| \le C \delta^{1/4}K^r.
	\end{equation}
	Hence,
	\begin{equation} \label{eq:L96Case2_errorsmall}
	\P\left(\int_0^{\tau_2} |X_{t,3} - x_{t,3}|dt \le C\tau_2 \delta^{1/4}K^r\right)\ge \frac{1}{2}.
	\end{equation}
	By choosing $\delta \ll \delta_1^4$, \eqref{eq:L96Case2_errorsmall} and \eqref{eq:L96Case2_growth} combined are enough to yield a contradiction. 
	
	\textbf{Case 3} ($x_0 \in B_{K,\delta}^3$): Now we turn to the final case. Let $X_t$ be as given in Lemma~\ref{lem:L96Case3_growth} and define $\bar{x}_{t,j} = X_{t,j} - x_{t,j}$. Observe that 
	\begin{equation}
	\begin{cases}
	d\bar{x}_{t,n} = x_{0,n-1}\bar{x}_{t,1} dt + x_{t,1}(x_{0,n-1} - x_{t,n-1})dt + a_n x_{t,n} + x_{t,n-2}x_{t,n-1}dt \\ 
	d\bar{x}_{t,1} = x_{0,2}\bar{x}_{t,n}dt + x_{t,n}(x_{0,2} - x_{t,2})dt + x_{t,n-1}x_{t,n}dt - \sigma_1 dW_t^{(1)}.
	\end{cases}
	\end{equation}
	Let
	$$F(t) = x_{t,1}(x_{0,n-1} - x_{t,n-1}) + a_n x_{t,n}+ x_{t,n-2}x_{t,n-1},$$
	$$ G(t) = x_{t,n}(x_{0,2} - x_{t,2}) + x_{t,n-1}x_{t,n},$$
	and
	$S(t)$ be the group generated by the corresponding (constant) linearization matrix:
	\begin{align*}
	S(t) := \exp \left( t \begin{pmatrix}
	0 & x_{0,n-1} \\ x_{0,2} & 0
	\end{pmatrix} \right). 
	\end{align*}
	Then, we have 
	\begin{equation} \label{eq:L96Case3_errorformula}
	\begin{pmatrix}
	\bar{x}_{t,n} \\ 
	\bar{x}_{t,1} \\ 
	\end{pmatrix}
	= 
	\int_0^t S(t-s)\begin{pmatrix}
	F(s) \\ 
	G(s) \\ 
	\end{pmatrix}ds 
	- 
	\int_0^t S(t-s)\begin{pmatrix}
	0 \\ 
	\sigma_1 dW_s^{(1)} \\ 
	\end{pmatrix}. 
	\end{equation}
	Note that since $\tau_3 K^{(1+r)/2} \leqc 1$, for any $s \le t \le \tau_3$ there holds
	\begin{equation} \label{eq:L96groupbound}
	\|S(t-s)\| \le \exp\left((t-s)\sqrt{|x_{0,n-1}||x_{0,2}|}\right) \le \exp\left((t-s)CK^{(1+r)/2}\right) \leqc 1.
	\end{equation}
	Thus,
	\begin{equation} \label{eq:L96Case3_convolution}
	\E\sup_{0 \le t \le \tau_3}\left|\int_0^t S(t-s)\begin{pmatrix}
	0 \\ 
	\sigma_1 dW_s^{(1)} \\ 
	\end{pmatrix}\right|^2 \leqc  \tau_3 = K^{-4r}.
	\end{equation}
	Suppose now for the sake of contradiction that 
	\begin{equation} \label{eq:L96Case3_contr}
	\E \int_0^{\tau_3} \sum_{j=3}^n|x_{t,j}|^2 dt \le \delta \tau_3 K^{2r}. 
	\end{equation}
	For $R \ge 1$, let $\Omega_0 \subseteq \Omega$ be the set where the following bounds hold:
	\begin{align*}
	\sup_{0 \le t \le \tau_3}\left|\int_0^t S(t-s)\begin{pmatrix}
	0 \\ 
	\sigma_1 dW_s^{(1)} \\ 
	\end{pmatrix}\right| &\le R \sqrt{\tau_3}, \\ 
	\int_0^{\tau_3} \sum_{j=3}^n |x_{t,j}|^2 dt & \le \sqrt{\delta} \tau_3 K^{2r}, \\
	\sup_{0 \le t \le \tau_3} |W_t| &\le R \sqrt{\tau_3}, \\ 
	\sup_{0 \le t \le \tau_3} |x_{t,1}| &\le (\delta_1^{1/4}+ \delta^{1/4})K^{4r} + 1.
	\end{align*}
	By \eqref{eq:L96Case3_convolution}, \eqref{eq:L96Case3_contr}, and 
	$$ \E \sup_{0 \le t \le \tau_3}|x_{t,1}| \leqc \delta_1 K^r + \sqrt{\delta} \tau_3 K^{1+r} + \sqrt{\tau_3} \leqc (\delta_1 + \sqrt{\delta})K^{4r} + K^{-2r}, $$
	for $R,K$ sufficiently large and $\delta, \delta_1$ sufficiently small we have $\P(\Omega_0) \ge 1-\beta/2$, where $\beta$ is as given in Lemma~\ref{lem:L96Case3_growth}. Note that for $\omega \in \Omega_0$ we have the additional estimates 
	\begin{align}
	\sup_{0 \le t \le \tau_3} |x_{t,n-1} - x_{0,n-1}| &\le CRK^{-2r}, \label{eq:L96dn-1} \\ 
	\sup_{0 \le t \le \tau_3} |x_{t,2} - x_{0,2}| &\le CR (\delta^{1/4}K^r + K^{-2r}). \label{eq:L96d2}
	\end{align}
	Moreover, by \eqref{eq:L96dn-1}, \eqref{eq:L96d2}, \eqref{eq:L96Case3_errorformula}, and \eqref{eq:L96groupbound} for $\omega \in \Omega_0$ there holds
	\begin{equation}
	\sup_{0\le t \le \tau_3}|\bar{x}_{t,n}| + \sup_{0\le t \le \tau_3}|\bar{x}_{t,1}| \le CR.
	\end{equation}
	Observe now that 
	\begin{equation}
	d(X_{t,3} - x_{t,3}) = -\bar{x}_{t,1}x_{t,2}dt + x_{t,1}(x_{t,2} - x_{0,2})dt +  \bar{x}_{t,1}(x_{t,2} - x_{0,2})dt - x_{t,4}x_{t,2}dt + a_3 x_{t,3} - \sigma_3 dW_t^{(3)},
	\end{equation}
	which together with the estimates above gives, for $\omega \in \Omega_0$,
	\begin{align*}
	\sup_{0 \le t \le \tau_3}|X_{t,3} - x_{t,3}| & \le \tau_3\sup_{0 \le t \le \tau_3}(|\bar{x}_{t,1}x_{t,2}| + |x_{t,1}||x_{t,2} - x_{0,2}| +  |\bar{x}_{t,1}||x_{t,2} - x_{0,2}|) \\ 
	& \quad + \int_0^{\tau_3}(|x_{t,4} x_{t,2}| + |a_3||x_{t,3}|)dt + \sup_{0\le t \le \tau_3}|\sigma_3 W_t^{(3)}| \\ 
	& \le CK^{4r}(\delta^{1/4}+R^2 K_*^{-r}).
	\end{align*}
	This error estimate (with $\delta$ taken sufficiently small and $K_*$ taken sufficiently large), together with Lemma \ref{lem:L96Case3_growth} on the growth of the approximate solution, allows us to obtain a contradiction as in our earlier arguments. This completes the proof of Proposition \ref{lem:L96Assumption1} (and hence also of Theorem \ref{thm:L96twomodes}). 
\end{proof}

\section{Stochastic triad model with non-trivial, invariant conservative dynamics in the kernel}

In this section we prove Theorem~\ref{thm:BadBad}. It is sufficient to prove the result after rotating coordinates so that $\mathrm{ker}A = \{x_2 = 0\}.$ In these new coordinates, the nonlinearity becomes 
\begin{equation} \label{eq:bbtriad}
B(x,y) = \begin{pmatrix}
x_1 y_3 \\ -x_2 y_3 \\ (x_2 - x_1)(y_2 + y_1) 
\end{pmatrix}. 
\end{equation}
Henceforth in this section, $x_t$ denotes the solution to \eqref{eq:SDE} with $n = 3$, $B$ given by \eqref{eq:bbtriad}, the non-negative definite matrix $A$ such that $\mathrm{ker}A = \{x_2 = 0\}$, and $\sigma \in \R^{3\times 3}$ satisfying $\mathrm{rank}(\sigma) = 3$.

The dynamical system $\dot{x} = B(x,x)$, with $B$ given by \eqref{eq:bbtriad}, has equilibria at $(0,0,\pm a)$ for any $a > 0$ and the stable/unstable manifold of each fixed point is joined to the other via a heteroclinic connection. The unstable manifold of $(0,0,a)$ is tangent to $\mathrm{ker}A$ and the associated heteroclinic connections with the stable manifold of $(0,0,-a)$ lie entirely in $\mathrm{ker}A$. The present example thus distinguishes itself from the previous ones in that there exist nontrivial conservative dynamics in $\mathrm{ker}A$. 

As in the earlier examples, our plan to prove Theorem~\ref{thm:BadBad} is to show that the Markov semigroup generated by \eqref{eq:SDE} satisfies Assumption~\ref{ass:time-average}. Again as before, we will deduce the growth required by \eqref{eq:time-averaged_main} by establishing it instead for a suitable approximate solution. The idea is to study the linearization of $\Pi_{\mathrm{ker}A^\perp} B(\cdot,\cdot)$ around $Z_t = (X_{t,1},0,X_{t,3}) \in \mathrm{ker}A$ solving 
\begin{equation} \label{eq:bbZt1}
\begin{cases}
\frac{d}{dt}X_{t,1}= X_{t,1}X_{t,3} \\ 
\frac{d}{dt}X_{t,3} = -X_{t,1}^2 \\ 
(X_{0,1},0,X_{0,3}) = \Pi_{\mathrm{ker}A} x_0.
\end{cases}
\end{equation}
Since $(0,0,-a)$ attracts all points on the circle $\{(x_1,0,x_3) \in \R^3: x_1^2 + x_3^2 = a^2\}$ except $(0,0,a)$ and has an unstable manifold perpendicular to $\mathrm{ker}A$, one expects this linearization to grow exponentially fast provided the noise has a nonzero projection onto both $(1,0,0)$ and $(0,1,0)$. Besides arguments analogous to those in previous sections used to study the linearization around \eqref{eq:bbZt1}, we construct a local Lyapunov function to estimate exit times of the process from the vicinity of the unstable fixed points $(0,0,a)$.

\subsection{Local Lyapunov function}

\begin{lemma} \label{lem:Lyapunov}
	For $K \ge 1$, let 
	$$\mathcal{B}_K = \{(x_1,x_2,x_3) \in \R^3: K/2 \le x_3 \le 2K, \quad |x_1| \le K, \quad  |x_2| \le K^{1/4}\}.$$
	There exists $\gamma \in (0,1)$ such that for all $K$ sufficiently large there is a smooth, strictly positive function $V_K:\R^3 \to \R$ such that for $x \in \mathcal{B}_{K}$,
	\begin{equation}
	\mathcal{L} V_K \le -\gamma K V_K
	\end{equation}
	and
	\begin{equation} \label{eq:VKpwbound}
	\gamma K^{-1} \le V_K \le \gamma^{-1} \sqrt{K}.
	\end{equation}
	Specifically, for some $R \ge 1$ sufficiently large, 
	\begin{align*}
	V_K = \frac{1}{\abs{x_1}}\chi_T(x_1) + \sqrt{K} \left(1 - \frac{K}{32}\frac{|x_1|^2}{R^2}\right) \chi_D(x_1), 
	\end{align*}
	where for an arbitrary smooth cutoff $\varphi:[0,\infty) \to [0,1]$ with $\varphi(y) = 1$ for $y \le 1/2$, $\varphi(y) = 0$ for $y \ge 1$, and $\varphi'(y) \le 0$, we define
	\begin{align*}
	\chi_D(x_1) = \varphi\left(\frac{\sqrt{K} |x_1|}{4R}\right), \quad \chi_T(x_1) = 1-\varphi\left(\frac{\sqrt{K} |x_1|}{R}\right). 
	\end{align*}
\end{lemma}
\begin{remark}
	The cutoff $\chi_T$ refers to `transport' as it is in the region $\abs{x_1} \gtrsim K^{-1/2}$ wherein the conservative dynamics (i.e. the first order terms in the generator) will be the most significant. The cutoff $\chi_D$ refers to `diffusive', as it is in the region $\abs{x_1} \ll K^{-1/2}$ in which the noise (i.e. the second order terms in the generator) will be dominant. 
\end{remark}

\begin{proof}
	Since $\mathrm{ker}A = \{x_2 = 0\}$, there exist $a_1, a_2, a_3 \in \R$ such that
	$$A x \cdot \grad = x_2\sum_{j=1}^3 a_j \partial_{x_j}.$$
	Defining $\Lambda = \sigma \sigma^T$, we can thus write the generator as
	\begin{equation} \label{eq:bbgenerator}
	\mathcal{L} = \frac{1}{2}\sum_{i,j=1}^3 \Lambda_{ij} \partial_{x_i x_j}  +x_1 x_3 \partial_{x_1} - x_2 x_3 \partial_{x_2} + (x_2^2 - x_1^2)\partial_{x_3} - x_2\sum_{j=1}^3 a_j \partial_{x_j}.
	\end{equation}
	Note that $\Lambda_{11} > 0$ since $\sigma$ is assumed full rank. Let $\varphi:[0,\infty) \to [0,1]$ be a smooth cutoff with $\varphi(y) = 1$ for $y \le 1/2$, $\varphi(y) = 0$ for $y \ge 1$ and $\varphi'(y) \le 0$. For $R \ge 1$ to be chosen sufficiently large independently of $K$, define
	$$\chi_T(x_1) = 1-\varphi\left(\frac{\sqrt{K} |x_1|}{R}\right), \quad V_{K,T}(x_1) = \chi_T(x_1) |x_1|^{-1}.$$
	For $x \in \mathcal{B}_K$ we compute 
	\begin{align*} 
	\mathcal{L} V_{K,T} &= - x_3 V_{K,T} + \chi_T \frac{\Lambda_{11}}{|x_1|^3} + \chi_T \frac{a_1 x_1 x_2}{|x_1|^3} \\ 
	& \quad  \left(-\Lambda_{11}\frac{x_1}{|x_1|^3} + \frac{x_1 x_3}{|x_1|} - \frac{x_2 a_1}{|x_1|}\right) \partial_{x_1} \chi_T + \frac{\Lambda_{11}}{2|x_1|}\partial_{x_1}^2 \chi_T \\ 
	& \le -\frac{K}{2} V_{K,T} + \left(\frac{\Lambda_{11}}{|x_1|^2} + \frac{|a_1||x_2|}{|x_1|}\right) V_{K,T} \\ 
	& \quad + C \frac{\sqrt{K}}{R} \left(\frac{1}{|x_1|^2} + |x_3| + \frac{|x_2| }{|x_1|}\right)\mathbf{1}_{RK^{-1/2}/2 \le |x_1| \le RK^{-1/2}} \\ 
	& \quad  + C\frac{K}{R^2} \frac{1}{|x_1|} \mathbf{1}_{RK^{-1/2}/2 \le |x_1| \le RK^{-1/2}},
	\end{align*}
	where $C$ is a constant that depends only on $A$, $\sigma$, and the choice of cutoff $\varphi$. We will continue to denote by $C$ such a constant, though it may change line-to-line. From the support properties of $\chi_T$ and $x \in \mathcal{B}_K$ we then obtain that for $R$ large depending only on $\Lambda_{11}$ and $|a_1|$ there holds
	\begin{equation} \label{eq:LVT}
	\mathcal{L}V_{K,T} \le - \frac{K}{4} V_{K,T} + C \frac{K^{3/2}}{R} \mathbf{1}_{RK^{-1/2}/2 \le |x_1| \le RK^{-1/2}}.
	\end{equation}
	Now define 
	$$ \chi_D(x_1) = \varphi\left(\frac{\sqrt{K} |x_1|}{4R}\right), \quad V_{K,D}(x_1) = \chi_D(x_1) \sqrt{K}\left(1 - \frac{K}{32}\frac{|x_1|^2}{R^2}\right)
	$$
	and note that 
	$$ \frac{\sqrt{K}}{2}\chi_D \le  V_{K,D} \le \sqrt{K} \chi_D.$$
	For $x \in \mathcal{B}_K$ we now compute
	\begin{align*}
	\mathcal{L}V_{K,D} & = \frac{K^{3/2} \chi_D}{16 R^2} \left( -\frac{\Lambda_{11}}{2} - |x_1|^2 x_3 + a_1 x_1 x_2 \right) \\ 
	& \quad + \partial_{x_1}\chi_D \left(-\Lambda_{11} \frac{K^{3/2}}{16 R^2} x_1 + \sqrt{K}\left(1-\frac{K|x_1|^2}{32 R^2}\right)(x_1 x_3 - a_1 x_2)\right) \\ 
	& \quad + \partial_{x_1}^2 \chi_D \frac{\Lambda_{11}}{2} \sqrt{K}\left(1-\frac{K |x_1|^2}{32R^2}\right) \\ 
	& \le -\frac{K \Lambda_{11}}{32 R^2} V_{K,D} - \frac{K^{5/2} \chi_D}{32 R^2} |x_1|^2 + C\frac{K^{5/4}}{R} \chi_D \\ 
	& \quad + C \frac{\sqrt{K}}{R}\left( \frac{K^{3/2}}{R^2} |x_1| + \sqrt{K}|x_2| \right) \mathbf{1}_{2R K^{-1/2} \le |x_1| \le 4R K^{-1/2}} \\ 
	& \quad + C \frac{K^{3/2}}{R^2}\mathbf{1}_{2R K^{-1/2} \le |x_1| \le 4R K^{-1/2}},
	\end{align*}
	where in the inequality we noted that $x_1 x_3 \partial_{x_1} \chi_D \le 0$ for $x \in \mathcal{B}_K$. Taking $K$ large enough so that 
	$$ C \frac{K^{5/4}}{R} \le \frac{\Lambda_{11} K^{3/2}}{128 R^2} $$ 
	it follows that 
	\begin{equation} \label{eq:LVD}
	\mathcal{L} V_{K,D} \le - \frac{K \Lambda_{11} }{64 R^2} V_{K,D} - \frac{K^{5/2} \chi_D}{32 R^2} |x_1|^2 + C \left(\frac{K^{3/2}}{R^2} + \frac{K^{5/4}}{R} \right) \mathbf{1}_{2R K^{-1/2} \le |x_1| \le 4R K^{-1/2}}.
	\end{equation}
	The plan is now to add \eqref{eq:LVT} and \eqref{eq:LVD}. Upon doing this, for $K$ and $R$ sufficiently large the second term in \eqref{eq:LVD} absorbs the second term in \eqref{eq:LVT} and the first term in \eqref{eq:LVT} absorbs the third term in \eqref{eq:LVD}. In particular, defining 
	$$V_K = V_{K,D} + V_{K,T}$$
	we have 
	\begin{equation}
	\mathcal{L} V_K \le -K \mathrm{min}\left(\frac{1}{8}, \frac{\Lambda_{11}}{64 R^2}\right) V_K.
	\end{equation}
	This completes the proof.
\end{proof}

The next lemma uses Lemma \ref{lem:Lyapunov} to obtain estimates on the exit times from neighborhoods of the north pole equilibria $x_1 = x_2= 0$, $x_3 > 0$.  
\begin{lemma} \label{lem:bbstopping}
	Let $x_0 \in \R^3$ satisfy 
	$$ |x_0| = K, \quad x_{0,3} > 0, \quad |x_{0,1}| < \delta K, \quad\text{and} \quad |x_{0,2}| < \delta K^{1/4} $$
	for $K \ge 1$ and $\delta \in (0,1)$. Define the stopping time 
	$$\tau(\omega) = \inf\{t \ge 0: |x_{t,1}(\omega)| \ge \delta K, \text{  } |x_{t,2}(\omega)| \le K^{1/4}\}.$$
	There exists $C_0 \ge 1$ so that for all $K$ sufficiently large and $\delta$ sufficiently small there holds 
	\begin{equation}
	\P\left(\tau \le \frac{C_0 \log K}{K}\right) \ge \frac{1}{2}.
	\end{equation}
\end{lemma}

\begin{proof}
	First note that taking at least $\delta^2 \le 7/32$ gives $3K/4 \le x_{0,3} \le K$, so we may assume that $x_0 \in \mathcal{B}_K$ as defined in Lemma~\ref{lem:Lyapunov}. Define $\mathcal{B}_{K,\delta} \subseteq \mathcal{B}_K$ by 
	$$ \mathcal{B}_{K,\delta} = \{x \in \R^3: K/2 < x_3 < 2K, \quad |x_1| < \delta K, \quad |x_2| <  K^{1/4}\} $$
	and let
	$$\tilde{\tau}(\omega) = \inf\{t \ge 0: x_t(\omega) \in \mathcal{B}_{K,\delta}^c\}. $$
	Let $V_K$ be as given in Lemma~\ref{lem:Lyapunov}. By Dynkin's formula, for any $t \ge 0$ there holds 
	\begin{equation}
	\E e^{\gamma K t \wedge \tilde{\tau}} V_K(x_{t \wedge \tilde{\tau}}) \le V_K(x_0) + \E\int_0^{t\wedge \tilde{\tau}}e^{\gamma K s}(\mathcal{L}+\gamma K)V_K(x_s)ds.
	\end{equation}
	Since $x_s(\omega) \in \mathcal{B}_K$ for $s \le \tilde{\tau}(\omega)$, it follows from Lemma~\ref{lem:Lyapunov} that 
	\begin{equation} \label{eq:tautilbound}
	\E e^{\gamma K t \wedge \tilde{\tau}} \le \gamma^{-2} K^{3/2}.
	\end{equation}
	From \eqref{eq:tautilbound} and Chebyshev's inequality we obtain, for any $C_0 > 0$,
	\begin{equation}
	\P\left(\tilde{\tau} \ge \frac{C_0 \log K}{K} \right) \le \gamma^{-2} K^{3/2 - \gamma C_0}.
	\end{equation}
	Hence, for $C_0 \ge 2/\gamma$ and $K$ sufficiently large there holds 
	\begin{equation} \label{eq:tautilbound2}
	\P\left(\tilde{\tau} \le \frac{C_0 \log K}{K} \right) \ge \frac{3}{4}.
	\end{equation}
	Now set $t_* = C_0\log(K)/K$. By \eqref{eq:tautilbound2} and the definitions of $\tilde{\tau}$ and $\tau$ we have
	\begin{align*}
	\P\left( \tau \le t_* \right) &\ge \frac{3}{4} - \P\left(\{|x_{\tilde{\tau},2}| \ge K^{1/4}\} \cap \left\{\tilde{\tau} \le t_*\right\}\right) \\ 
	& \quad - \P\left(\{x_{\tilde{\tau},3} \ge 2K \text{ or } x_{\tilde{\tau},3} \le K/2\}\cap \left\{\tilde{\tau} \le t_*\right\}\right),
	\end{align*}
	and so to complete the proof it suffices to show that 
	\begin{equation} \label{eq:tautilbounds}
	\P\left(\{|x_{\tilde{\tau},2}| \ge K^{1/4}\} \cap \left\{\tilde{\tau} \le t_*\right\}\right)  
	+ \P\left(\{x_{\tilde{\tau},3} \ge 2K \text{ or } x_{\tilde{\tau},3} \le K/2\}\cap \left\{\tilde{\tau} \le t_*\right\}\right) \le \frac{1}{4}.
	\end{equation}
	To bound the first term we begin by using Dynkin's formula to obtain
	\begin{equation}
	\E|x_{t_* \wedge \tilde{\tau},2}|^2 = |x_{0,2}|^2 + \E \int_0^{t_* \wedge \tilde{\tau}} \left(\Lambda_{22} - 2x_{s,3}|x_{s,2}|^2 - 2a_2 |x_{s,2}|^2\right)ds. 
	\end{equation}
	Since $a_2 \ge 0$ and $x_{s,3} > 0$ for $s \le \tilde{\tau}$ it follows that
	\begin{align*}
	\E |x_{t_* \wedge \tilde{\tau},2}|^2 &\le |x_{0,2}|^2 + \Lambda_{22} t_*  \le \delta^2 \sqrt{K} + \Lambda_{22}t_*.
	\end{align*}
	Thus, 
	\begin{equation}
	\P\left(\{|x_{\tilde{\tau},2}| \ge K^{1/4}\} \cap \left\{\tilde{\tau} \le t_*\right\}\right) \le \frac{1}{\sqrt{K}} \E|x_{t_* \wedge \tilde{\tau}}|^2 \le \delta^2 + CK^{-1/2},
	\end{equation}
	which implies 
	\begin{equation}
	\P\left(\{|x_{\tilde{\tau},2}| \ge K^{1/4}\} \cap \left\{\tilde{\tau} \le t_*\right\}\right) \le \frac{1}{8}
	\end{equation}
	for $\delta$ sufficiently small and $K$ sufficiently large. To bound the second term in \eqref{eq:tautilbounds}, observe that for $K \ge 8$ and $\delta^2 \le 1/8$ we have
	$$ \P\left(\{x_{\tilde{\tau},3} \ge 2K \text{ or } x_{\tilde{\tau},3} \le K/2\}\cap \left\{\tilde{\tau} \le t_*\right\}\right) \le \P\left(\sup_{0 \le t \le t_*} \left| |x_t|^2 - K^2\right| \ge K^2/2\right). $$
	Now, for $K$ sufficiently large, Lemma~\ref{lem:basicenergy} implies
	$$ \P\left(\sup_{0 \le t \le t_*} \left| |x_t|^2 - K^2\right| \ge K^2/2\right) \le \frac{1}{8}, $$
	which completes the proof.
\end{proof}

\subsection{Growth of approximate solution}

The next lemma gives the growth of an approximate solution for initial conditions that are in the vicinity of $\text{ker}A$ but not too close to the north pole equilibria.

\begin{lemma} \label{lem:bbgrowth}
	Fix $r, \epsilon,\delta \in (0,1/4)$ and let $X_t$ solve 
	\begin{equation}\label{eq:bbapprox}
	\begin{cases}
	d X_{t,1} = X_{t,1} X_{t,3} dt \\ 
	d X_{t,2} = - X_{t,2}X_{t,3}dt + \sum_{j=1}^3 \sigma_{2 j} dW_t^{(j)} \\ 
	d X_{t,3} = -X_{t,1}^2dt 
	\end{cases}
	\end{equation}
	with an initial condition $X_0 \in \R^3$ that satisfies 
	$$|X_{0,2}| \le (|X_{0,1}|^2 + |X_{0,3}|^2)^{r/2}$$
	and at least one the bounds
	\begin{equation} \label{eq:bbgrowth1}
	X_{0,3} \le 0 \quad \text{or} \quad |X_{0,1}| \ge \delta |X_0|.
	\end{equation}
	There exist $K_*(\epsilon) \ge 1$, $C_0(\epsilon,\delta)\ge 1$, $c_*(\epsilon,\delta) > 0$, and $\beta > 0$ so that for $|X_0| \ge K_*$ and
	$$\tau = \frac{C_0(\epsilon,\delta)}{|X_0|} + \frac{(1/2 + r+\epsilon)\log(|X_0|)}{(1-\epsilon)|
		X_0|} $$
	there holds 
	\begin{equation}
	\P\left(\frac{1}{\tau} \int_0^\tau |X_{t,2}|dt \ge c_* |X_0|^r\right)\ge \beta.
	\end{equation}
\end{lemma}
\begin{remark}
	Observe that the dynamics of $Z_t:=(X_{t,1},0,X_{t,3}) \in \mathrm{ker}A$ is decoupled from $X_{t,2}$ and satisfies 
	\begin{equation} \label{eq:Ztconservation}
	|Z_t|^2 = |Z_0|^2
	\end{equation}
	for all $t \ge 0$.
\end{remark}
\begin{proof}
	By \eqref{eq:Ztconservation} and $|X_{0,2}| \le |Z_0|^r$, for $K_*(\epsilon)$ sufficiently large we have $|Z_t| \ge (1-\epsilon/2)|X_0|$ for all $t\ge 0$. It follows then by \eqref{eq:bbgrowth1} and rescaling $Z_t$ back to the unit circle that there exists $C_0(\epsilon,\delta)$ such that for $\tau_1 = C_0/|X_0|$ there holds 
	\begin{equation} 
	X_{t,3} \le -\left(1-\epsilon\right)|X_0| 
	\end{equation}
	for all $t \ge \tau_1$. Now, we may assume without loss of generality that $\sum_{j=1}^3 \sigma_{2j}dW_t^{(j)} = dW_t$ for a standard Brownian motion $W_t$. The formula for $X_{t,2}$ then reads
	\begin{equation}
	X_{t,2} = e^{-\int_0^t X_{s,3}ds}X_{0,2} + \int_0^t e^{-\int_s^t X_{s',3}ds'}dW_s.
	\end{equation}
	Since $X_{t,3}$ is deterministic, we have that $\int_0^\tau X_{t,2}dt$ is a Gaussian random variable with variance 
	\begin{align*}
	\E\left|\int_0^\tau X_{t,2} dt - \int_0^\tau e^{-\int_0^t X_{s,3}ds} X_{0,2}dt\right|^2 & = \E \abs{ \int_0^\tau \int_0^t e^{-\int_s^{t} X_{s',3} ds'} dW_{s} dt }^2 \\
	& = \E  \int_0^\tau \abs{\int_s^\tau e^{-\int_s^{t} X_{s',3} ds'} dt}^2 ds. 
	\end{align*}
	Using $X_{s,3} \le |X_0|$ for any $s$ and $X_{s,3} \le -(1-\epsilon)|X_0|$ for $s \ge \tau_1$ we obtain, for $K_*$ sufficiently large, the lower bound 
	\begin{equation}
	\begin{aligned}
	\E  \int_0^\tau \abs{\int_s^\tau e^{-\int_s^{t} X_{s',3} ds'} dt}^2 ds \gtrsim_{\epsilon,\delta} \frac{\tau}{|X_0|^2} e^{2(1-\epsilon)|X_0|(\tau-\tau_1)}= \tau |X_0|^{2r+2\epsilon - 1}.
	\end{aligned}
	\end{equation}
	It follows that there are $c_*(\epsilon,\delta), \beta > 0$ sufficiently small so that
	\begin{equation}\label{eq:bbgrowth2}
	\P\left(\frac{1}{\tau} \int_0^\tau |X_{t,2}| dt \ge c_* \frac{|X_0|^{r+\epsilon}}{\sqrt{\log|X_0|}} \right) \ge \beta.
	\end{equation}
	This completes the proof for $K_*(\epsilon)$ sufficiently large.
\end{proof}

\subsection{Justifying the approximation}

Theorem~\ref{thm:BadBad} follows immediately from Lemma~\ref{lem:time-averaged_step1} and the following proposition.

\begin{proposition}
	Fix $r \in (0,1/4)$ and for $\delta \in (0,1/4)$ define
	$$B_{K}^\delta = \{(x_1,x_2,x_3) \in \R^3: |x_2|^2 \le \delta (|x_1|^{2} + |x_3|^{2})^r, \quad (1-\delta)K^2 \le |x|^2 \le (1+\delta)K^2\}.$$
	For $\delta$ sufficiently small there exist $\eta(K) \approx \log(K)/K$, $c_* > 0$, and $K_* \ge 1$ so that for any $K \ge K_*$ large there holds 
	\begin{equation} \label{eq:bbapprpoxgoal} 
	x_0 \in B_K^\delta \implies \frac{1}{\eta(K)}\E \int_0^{\eta(K)} |x_{t,2}|^2 dt \ge c_* K^{2r}.
	\end{equation}
\end{proposition}

\begin{proof}
	Let 
	$$S_{\delta_1} = \{x \in \R^3: x_3 > 0, |x_1| < \delta_1 |x|\} $$
	and
	$$ B_K^{\delta,\delta_1} = S_{\delta_1}^c \cap \{x\in \R^3: |x_2|^2 \le (|x_1|^2 + |x_3|^2)^r, (1-2\delta)K^2 \le |x|^2 \le (1+2\delta)K^2\}.$$
	We first show that it is sufficient to prove that if $\delta$ sufficiently small then for every $\delta_1 \in (0,1/4)$ there is $\eta_1(K) \approx_{\delta_1} \log(K)/K$ and $c > 0$ so that for $K$ sufficiently large there holds
	\begin{equation} \label{eq:bbreduction1}
	x_0 \in B_K^{\delta,\delta_1} \implies \frac{1}{\eta_1(K)} \int_0^{\eta_1(K)} \Pt_t D(x_0) dt \ge c K^{2r},
	\end{equation}
	where $D: \R^3 \to \R^3$ is defined by $D(x) = |x_2|^2$. Fix $x_0 \in B_{K}^\delta$ and let
	$$\tilde{\tau}(\omega) = \inf\left\{t \ge 0: x_t(\omega) \in B_K^{\delta, \delta_1}\right\}.$$
	By Lemma~\ref{lem:bbstopping} and Lemma~\ref{lem:basicenergy}, there are $\delta_1,C_1 > 0$ so that for all $\delta$ small enough, if $K$ is taken sufficiently large depending on $\delta$ there holds 
	\begin{equation} \label{eq:tautil1}
	\P\left(\tilde{\tau}_1 \le \frac{C_1 \log(K)}{K} \right) \ge \frac{1}{4}.
	\end{equation}
	Let now $\eta_1$ and $c$ be as in \eqref{eq:bbreduction1} applied with $\delta_1$ chosen so that \eqref{eq:tautil1} holds. Define 
	$$\eta(K) = \eta_1(K) + \frac{C_1 \log(K)}{K}.$$
	By \eqref{eq:tautil1}, the strong Markov property and \eqref{eq:bbreduction1}, for $K$ sufficiently large we have
	\begin{align*} 
	\E \int_0^{\eta} |x_{t,2}|^2 dt 
	& \ge \int_{\Omega} \int_0^{\eta - \tilde{\tau}\wedge \eta} D(x_{\tilde{\tau} \wedge \eta + t}(\omega))dt d\P \\ 
	& \ge \frac{1}{4} \inf_{\tilde{\tau}(\omega) \le C_1 \log(K)/K} \int_0^{\eta_1} \Pt_t D(x_{\tilde{\tau}}(\omega)) dt \\ 
	& \ge \frac{c}{4} \eta_1(K) K^{2r}.
	\end{align*}
	The bound \eqref{eq:bbapprpoxgoal} then follows since $\eta_1(K) \approx \eta(K)$.
	
	We now prove \eqref{eq:bbreduction1}. Let $x_0 \in B_{K}^{\delta,\delta_1}$ for some $\delta_1 \in (0,1/4)$. For $\epsilon \in (0,1/4)$ to be chosen later, define
	$$ \tau = \frac{C_0(\epsilon,\delta_1)}{|x_0|} + \frac{(1/2+r+\epsilon) \log(|x_0|)}{(1-\epsilon)|x_0|}, $$
	where $C_0$ is as defined in Lemma~\ref{lem:bbgrowth}.
	Suppose now for contradiction that 
	\begin{equation} \label{eq:bbcontradiction}
	\frac{1}{\tau} \E \int_0^\tau |x_{t,2}|^2 dt \le \delta_2 K^{2r}
	\end{equation}
	for some $\delta_2 \in (0,1)$ to be chosen sufficiently small. Let $X_t$ solve \eqref{eq:bbapprox} with initial condition $x_0$. By Lemma~\ref{lem:bbgrowth} there exists $c_*(\epsilon,\delta_1), \beta > 0$ so that for $K$ sufficiently large (depending on $\epsilon$) there holds
	\begin{equation} \label{eq:bbapproxlemgrowth}
	\P\left( \frac{1}{\tau}\int_0^\tau |X_{t,2}| dt \ge c_* |x_0|^r \right) \ge \beta
	\end{equation}
	and $X_{t,3} \le -(1-\epsilon)|x_0|$ for all $t \ge \tau_1:=C_0/|x_0|$. As in our earlier proofs, the plan is now to proceed by deriving suitable estimates on the error $|X_{t,2} - x_{t,2}|$.
	
	We begin by estimating $|\Pi_{\mathrm{ker}A}(X_t-x_t)|$. This is slightly more involved than in earlier arguments since $\Pi_{\mathrm{ker}A}X_t$ is not constant. We denote $Z_t = (Z_{t,1},0,Z_{t,3})= (X_{t,1},0,X_{t,3}) \in \mathrm{ker}A$ and $Y_t = X_{t,2} \in \mathrm{ker}A^\perp$ and define $z_t$ and $y_t$ similarly. Moreover, for $z \in \mathrm{ker}A$ we define the linear operator $\tilde{L}_z:\R^3 \to \R^3$ by 
	$$\tilde{L}_z(x) = \Pi_{\mathrm{ker}A}(B(z,\Pi_{\mathrm{ker}A}x)+B(\Pi_{\mathrm{ker}A}x,z)) = \Pi_{\mathrm{ker}A}L_z \Pi_{\mathrm{ker}A}.$$	 
	The error $\bar{z}_t = Z_t - z_t$ then solves
	$$ d{\bar{z}_t} =  \tilde{L}_{Z_t}(\bar{z_t})dt - B(\bar{z}_t,\bar{z}_t)dt - B(y_t,y_t)dt
	+ \Pi_{\mathrm{ker}A}\sigma dW_t.
	$$
	For $f:[0,\infty) \to \mathrm{ker}A$ and $h \in \R^3$ we write $S_{f}(t,s)h$ for the solution to the problem 
	$$
	\begin{cases}
	\frac{d}{dt} S_{f}h = \tilde{L}_{f(t)} S_{f}h, & t > s \\ 
	S_{f}(s,s)h = h.
	\end{cases}
	$$
	With this notation, $\bar{z}_t$ satisfies
	$$\bar{z}_t = -\int_0^t S_{Z}(t,s)[B(\bar{z}_s,\bar{z}_s) + B(y_s,y_s)]ds
	+ \int_0^t S_{Z}(t,s) \Pi_{\mathrm{ker}A}\sigma dW_s.$$
	By H\"{o}lder's inequality and Fubini's theorem, for any $T \le \tau$ we have 
	\begin{equation} \label{eq:ztbar1}
	\begin{aligned}
	\|\bar{z}_t\|_{L^2(0,T)} &\lesssim \left(\int_0^T\int_0^t \norm{S_Z(t,s)}^2\left(\norm{\bar{z}_s}^2 + \norm{y_s}^2 \right) ds dt\right)^{1/2} \left(\|\bar{z}_t\|_{L^2(0,T)} + \|y_t\|_{L^2(0,T)}\right) \\
	& \quad +  \left(\int_0^\tau \left| \int_0^t S_{Z}(t,s) \Pi_{\mathrm{ker}A}\sigma dW_s\right|^2 dt\right)^{1/2} \\ 
	&\leqc  \left(\sup_{0 \le s \le \tau}\int_0^\tau \mathbf{1}_{s\le t} \|S_{Z}(t,s)\|^2 dt \right)^{1/2}\left(\|\bar{z}_t\|_{L^2(0,T)}^2 + \|y_t\|_{L^2(0,T)}^2\right) \\
	& \quad +  \left(\int_0^\tau \left| \int_0^t S_{Z}(t,s) \Pi_{\mathrm{ker}A}\sigma dW_s\right|^2 dt\right)^{1/2}.
	\end{aligned}
	\end{equation}
	To proceed we need estimates for $\|S_Z(t,s)\|$. A straightforward computation shows that the top eigenvalue of the symmetric part of $\tilde{L}_z$ is bounded above by $(z_3 + |z|)/2$. Thus, using that $Z_{t,3} \le -(1-\epsilon)|x_0|$ for $t \ge \tau_1$, we have 
	\begin{equation}
	\|S_{Z}(t,s)\| \le \exp\left(\int_s^t \frac{Z_{t',3}+|Z_{t'}|}{2} dt' \right) \le e^{C_0} e^{(t-\tau_1)\epsilon |x_0|/2}.
	\end{equation}
	Consequently,
	\begin{equation}\label{eq:bbsemigroupbound}
	\begin{aligned}
	\left(\sup_{0 \le s \le \tau}\int_0^\tau \mathbf{1}_{s\le t} \|S_{Z}(t,s)\|^2 dt \right)^{1/2} \le e^{C_0} \left(\int_0^{\tau} e^{(t-\tau_1) \epsilon |x_0|}dt\right)^{1/2}  \le \frac{e^{C_0}}{\sqrt{\epsilon}} |x_0|^{\epsilon - 1/2}
	\end{aligned}
	\end{equation}
	and
	\begin{equation} \label{eq:bbnoisebound}
	\E \int_0^\tau \left| \int_0^t S_{Z}(t,s) \Pi_{\mathrm{ker}A}\sigma dW_s\right|^2 dt \leqc \int_0^\tau \int_0^t \|S_{Z}(t,s)\|^2 ds dt \leqc \tau \frac{e^{2C_0}}{\epsilon}|x_0|^{2\epsilon - 1}.
	\end{equation}
	For $R \ge 1$, define 
	$$\Omega_1 = \left\{\omega \in \Omega: \left(\int_0^\tau\left|\int_0^t S_{Z}(t,s)\Pi_{\mathrm{ker}A}\sigma dW_s\right|^2 dt\right)^{1/2} \le R \sqrt{\tau} |x_0|^{\epsilon - 1/2}\right\}$$
	and 
	$$ \Omega_2 = \{\omega \in \Omega: \int_0^\tau |y_t|^2 dt \le \sqrt{\delta_2} \tau K^{2r}\}.$$
	By \eqref{eq:bbnoisebound} and \eqref{eq:bbcontradiction}, for $R(\epsilon,\delta_1,\beta)$ sufficiently large and $\delta_2(\beta)$ sufficiently small there holds
	\begin{equation}\label{eq:bbomega0}
	\P(\Omega_1 \cap \Omega_2) \ge 1 - \frac{\beta}{2}.
	\end{equation}
	By \eqref{eq:bbsemigroupbound}, \eqref{eq:ztbar1}, and $r < 1/4$, for $\omega \in \Omega_1 \cap \Omega_2$ there holds 
	\begin{align*}
	\|\bar{z}_t(\omega)\|_{L^2(0,T)} &\leqc \frac{e^{C_0}}{\sqrt{\epsilon}} K^{\epsilon - 1/2}\left(\|\bar{z}_t(\omega)\|^2_{L^2(0,T)} + \sqrt{\delta_2}\tau K^{2r}\right)+ R \sqrt{\tau} K^{\epsilon - 1/2} \\ 
	& \leqc \frac{e^{C_0}}{\sqrt{\epsilon}} K^{\epsilon - 1/2} \|\bar{z}_t(\omega)\|_{L^2(0,T)}^2 + R\frac{e^{C_0}}{\sqrt{\epsilon}}\sqrt{\tau}K^{\epsilon - 1/2}
	\end{align*}
	for any $T \le \tau$. From a standard continuity argument, for $K$ sufficiently large we have
	\begin{equation} \label{eq:ztbar2}
	\|\bar{z}_t(\omega)\|_{L^2(0,\tau)} \leqc R \frac{e^{C_0}}{\sqrt{\epsilon}} \sqrt{\tau} K^{\epsilon - 1/2}.
	\end{equation}
	
	Now we use \eqref{eq:ztbar2} to bound $\bar{y}_t:=Y_t - y_t$ for $\omega \in \Omega_1 \cap \Omega_2$. We have 
	\begin{equation}
	\frac{d}{dt}\bar{y}_t = -Z_{t,3}\bar{y}_t - \bar{z}_{t,3}y_t,
	\end{equation}
	so that 
	\begin{equation}
	\bar{y}_t = - \int_0^t \exp\left(-\int_s^t Z_{t',3}dt'\right)\bar{z}_{s,3}y_s ds.
	\end{equation}
	Using the rough bound 
	$$ \exp\left(-\int_s^t Z_{t',3}dt'\right) \le e^{(t-s)|x_0|}$$
	and Young's convolution inequality we obtain
	\begin{equation}
	\|\bar{y}_t\|_{L^1(0,\tau)} \leqc \frac{e^{|x_0|\tau}}{|x_0|} \int_0^\tau |\bar{z}_{t,3}| |y_t|dt.
	\end{equation}
	Thus, utilizing \eqref{eq:ztbar2}, for $\omega \in \Omega_1 \cap \Omega_2$ and $\epsilon$ small enough we have 
	\begin{equation} \label{eq:bbapprox2}
	\begin{aligned}
	\|\bar{y}_t(\omega)\|_{L^1(0,\tau)} &\leqc e^{C_0} K^{\frac{1/2 + r+\epsilon}{1-\epsilon} - 1} R\frac{e^{C_0}}{\sqrt{\epsilon}} \sqrt{\tau} K^{\epsilon -1/2} \delta_2^{1/4} \tau K^r \\ 
	& \leqc \delta_2^{1/4}R\frac{e^{2C_0}}{\sqrt{\epsilon}} \tau K^r.
	\end{aligned}
	\end{equation}
	As in our earlier proofs, \eqref{eq:bbapprox2}, \eqref{eq:bbapproxlemgrowth}, and \eqref{eq:bbomega0} together are enough to obtain a contradiction for $\delta_2$ sufficiently small. The result is that there is $c>0$ so that for $K$ sufficiently large there holds 
	\begin{equation} \label{eq:bbapprox3}
	\frac{1}{\tau}\E\int_0^{\tau}|x_{t,2}|^2dt \ge c K^{2r}.  
	\end{equation}
	For $K$ large and $\delta$ small one has $\tau \approx_{\delta_1} \log(K)/K$, and so from \eqref{eq:bbapprox3} the proof is completed by setting $\eta_1(K) = C_{\delta_1}\log(K)/K$ for some large constant $C_{\delta_1}$. 
\end{proof}

\appendix
\section{A basic energy estimate}
The following lemma quantifies how $B(x,x) \cdot x = 0$ and the additive nature of the noise imply that the energy level of a trajectory can only change a small amount in a short time.
\begin{lemma} \label{lem:basicenergy}
	Fix $\epsilon \in (0,1)$. There exist $K_*(\epsilon)\ge 1$, $\tau_*(\epsilon) \le 1$, and $C > 0$ (which does not depend on $\epsilon$) such that for $0 \le \tau \le \tau_*$ and any $x_0 \in \R^n$ with $|x_0| = K \ge K_*$ there holds
	\begin{equation}
	\P\left(\sup_{0\le t \le \tau} \left| |x_{t}|^2 - K^2\right| \ge \epsilon K^2 \right) \le \frac{C\tau}{\epsilon^2 K^2}.
	\end{equation}
\end{lemma}

\begin{proof}
	It suffices to show 
	\begin{equation} \label{eq:energylemma1}
	\P\left(\sup_{0 \le t \le \tau} |x_t|^2 \ge (1+\epsilon/2)|x_0|^2 + R^2 \right) \le C \frac{\tau |x_0|^2}{R^4}
	\end{equation}
	and 
	\begin{equation} \label{eq:energylemma2}
	\P\left(\inf_{0 \le t \le \tau}|x_t|^2 \le (1-\epsilon/2)|x_0|^2 - R^2\right) C \frac{\tau |x_0|^2}{R^4}
	\end{equation}
	for some constant $C$ that does not depend on $\epsilon$. Indeed, the desired result follows immediately by taking $R = \sqrt{\epsilon/2}|x_0|$ in \eqref{eq:energylemma1} and \eqref{eq:energylemma2}.
	
	We begin with the proof of \eqref{eq:energylemma1}. By It\^{o}'s formula and $B(x,x)\cdot x = 0$, we have 
	\begin{equation} \label{eq:exptail1}
	|x_t|^2- |x_0|^2 = 2\int_0^t x_s \cdot \sigma dW_s - 2\int_0^t Ax_s \cdot x_s ds + t \sum_{i,j=1}^n |\sigma_{ij}|^2 .
	\end{equation}
	Thus, 
	\begin{equation} \label{eq:exptail2}
	\E |x_t|^2 \le |x_0|^2 + C_\sigma t,
	\end{equation}
	where we have set $C_\sigma = \sum_{i,j=1}^n |\sigma_{ij}|^2$. Using the martingale inequality followed by It\^{o} isometry and \eqref{eq:exptail2} in \eqref{eq:exptail1} gives
	\begin{align}
	&\P\left(\sup_{0 \le t \le \tau}  |x_t|^2 - |x_0|^2 - C_\sigma t \ge R^2 \right)  \le \P\left( 2\sup_{0 \le t \le \tau}\left| \int_0^t x_s \cdot \sigma dW_s \right|\ge R^2 \right) \nonumber \\ 
	& \qquad  \le \frac{4}{R^4}\E\left|\int_0^\tau x_s \cdot \sigma dW_s\right|^2  \le \frac{C}{R^4} \int_0^\tau \E |x_s|^2 ds \le C \frac{\tau |x_0|^2}{R^4}, \label{eq:exptail4}
	\end{align}
	where in the last inequality we assume that $K_*$ is sufficiently large. The bound \eqref{eq:energylemma1} then follows provided $K_* \ge \sqrt{2C_\sigma/\epsilon}$.
	
	Now we turn to the proof of \eqref{eq:energylemma2}. Let $\tilde{x}_t = e^{\lambda_A t}x_t$, where $\lambda_A > 0$ is the largest eigenvalue of $A$. Then, 
	\begin{equation} 
	d \tilde{x}_t = \lambda_A \tilde{x}_t dt + e^{\lambda_A t}B(x_t,x_t)dt - A\tilde{x}_t dt + e^{\lambda_A t}\sigma dW_t.
	\end{equation}
	Since $\tilde{x}_t \cdot e^{\lambda_A t}
	B(x_t,x_t) = e^{2\lambda_A t} B(x_t,x_t) \cdot x_t = 0$, another application of It\^{o}'s lemma gives
	\begin{equation} \label{eq:exptail3}
	d|\tilde{x}_t|^2 = 2\lambda_A |\tilde{x_t}|^2 dt - 2 A\tilde{x}_t \cdot \tilde{x}_t dt + 2 e^{\lambda_A t} \tilde{x}_t \cdot \sigma dW_t + C_\sigma e^{2\lambda_A t} dt.
	\end{equation}
	Using $\lambda_A |\tilde{x}_t|^2 \ge A\tilde{x}_t \cdot \tilde{x}_t$, \eqref{eq:exptail3} implies
	\begin{equation}
	|\tilde{x}_t|^2 \ge |x_0|^2 - 2 \left| \int_0^t e^{\lambda_A s}\tilde{x}_s \cdot \sigma dW_s\right|,
	\end{equation}
	and hence for $0\le t \le \tau$ there holds
	\begin{equation}
	|x_t|^2 \ge e^{-2 \lambda_A \tau}|x_0|^2 - 2\left| \int_0^t e^{2 \lambda_A s}x_s \cdot \sigma  dW_s\right|.
	\end{equation}
	Proceeding as in the proof of \eqref{eq:exptail4} we obtain
	\begin{align}
	\P\left(\inf_{0\le t \le \tau} |x_t|^2 \le e^{-2\lambda_A \tau}|x_0|^2 - R^2\right) \le \frac{C}{R^4}\int_0^\tau \E |x_s|^2 ds \le C \frac{\tau |x_0|^2}{R^4}.	
	\end{align}
	The desired bound follows provided that $\tau_*$ is small enough so that $e^{-2\lambda_A \tau_*} \ge 1-\epsilon/2$.
\end{proof}

\phantomsection
\addcontentsline{toc}{section}{References}
\bibliographystyle{abbrv}	

\end{document}